\documentclass{article}

\usepackage{amssymb}
\usepackage{amsmath}
\usepackage{latexsym}
\usepackage{bm}
\usepackage{graphicx}
\usepackage{extarrows}
\usepackage{bm}
\usepackage{bbm}
\usepackage{amsfonts}
\usepackage{amssymb}
\usepackage{indentfirst}
\usepackage{bbm}
\usepackage[all]{xy}
\usepackage{mathtools}
\usepackage{indentfirst}

\usepackage{float}
\usepackage{amsthm}
\newtheorem{definition}{Definition}[section]
\newtheorem{theorem}{Theorem}[section]
\newtheorem{proposition}[theorem]{Proposition}
\newtheorem{lemma}[theorem]{Lemma}
\newtheorem{corollary}[theorem]{Corollary}
\newtheorem{example}[theorem]{Example}
\newtheorem*{remark}{Remark}

\newcommand{\stirlingii}{\genfrac{\{}{\}}{0pt}{}}

\DeclareMathOperator{\End}{End}
\DeclareMathOperator{\Hom}{Hom}
\DeclareMathOperator{\rk}{rank}

\DeclareMathOperator{\tr}{tr}
\DeclareMathOperator{\Tr}{Tr}
\DeclareMathOperator{\spn}{span}

\begin{document}
\title{Motzkin Algebras and the $A_n$ Tensor Categories of Bimodules}

\author{Vaughan F.R. Jones  \and  Jun Yang}
\date{}

\maketitle

\begin{abstract}
We discuss the structure of the Motzkin algebra $M_k(D)$ by introducing a sequence of idempotents and the basic construction.
We show that $\cup_{k\geq 1}M_k(D)$ admits a factor trace if and only if $D\in \{2\cos(\pi/n)+1|n\geq 3\}\cup [3,\infty)$
and higher commutants of these factors depend on $D$.
Then a family of irreducible bimodules over these factors are constructed.
A tensor category with $A_n$ fusion rule is obtained from these bimodules.
\end{abstract}
\tableofcontents

\section{Introduction}

A unitary fusion category is a tensor category with finitely many simple objects whose morphism spaces all admit a positive definite invariant inner product, and satisfy certain other axioms - see \cite{BK01}.
One way to obtain fusion categories is to take a family of bimodules (correspondences in the sense of Connes) over a von Neumann algebra, closed under tensor product, and look at the category they generate (morphisms being bimodule maps between Hilbert spaces).
It is perhaps not surprising that this method is universal.
Using a graded construction related to random matrices \cite{GJS}, it was shown in \cite{HP} that any suitable unitary tensor category can be obtained from a system of bimodules over the von Neumann algebra of the free group on infinitely many generators.
But realising these categories over the hyperfinite $\text{II}_1$ factor $R$ is actually more difficult.
In a beautiful paper \cite{Ya04} Yamagami has shown how to do it for unitary fusion categories and more general, “amenable” tensor categories.

The simplest tensor categories known to exist are ones with a generator V whose “fusion graph” is $A_n$,
that is to say there are objects $V_i$ for $i=0,1,2,\dots,n-1$ and such that

\begin{equation*}
V_1\otimes V_i=
\left\{
             \begin{array}{ll}
             V_1, & \text{if~} i=0 \\
             V_{i-1}\oplus V_{i+1}, & \text{if~} 1\leq i\leq n-2  \\
             V_{n-2}, & \text{if~} i=n-1
             \end{array}
\right.
\end{equation*}

We became interested in the most elementary way to realise such a system of bimodules over the hyperfinite $\text{II}_1$ factor $R$.
One way would be to use Wassermann’s result in conformal field theory - \cite{Was98} but that is the opposite of elementary.
A way that, surprisingly, does not work is to consider the bimodule system coming from a subfactor $N\subset M$ with principal graph $A_n$.
From a distance the fusion rules are ok but in fact the bimodule system is graded-there are $N-N$,$N-M$,$M-N$ and $M-M$ bimodules.
One can try to convert them all into $R-R$ bimodules by using isomorphisms with R but such isomorphisms will destroy the nice fusion rules.

One could always construct a whole category containing such $V_i$ and use Yamagami’s theorem.
The crucial issue is positive definiteness in the category which is in general a subtle and difficult question-see \cite{J83,J99,Liu16}.
But if one considers “Temperley-Lieb” diagrams, positivity on the morphism spaces between simple objects is well established and can be deduced from \cite{J83}.
To apply \cite{Ya04} one would then have to extend positivity to morphisms between direct sums of simple objects which can be done.

We attempted a very “naive” approach using just TL diagrams in which the
(Hilbert space) bimodules would be linearly spanned by diagrams of the following form:
\begin{figure}[H]
  \centering
  \includegraphics[width=5cm]{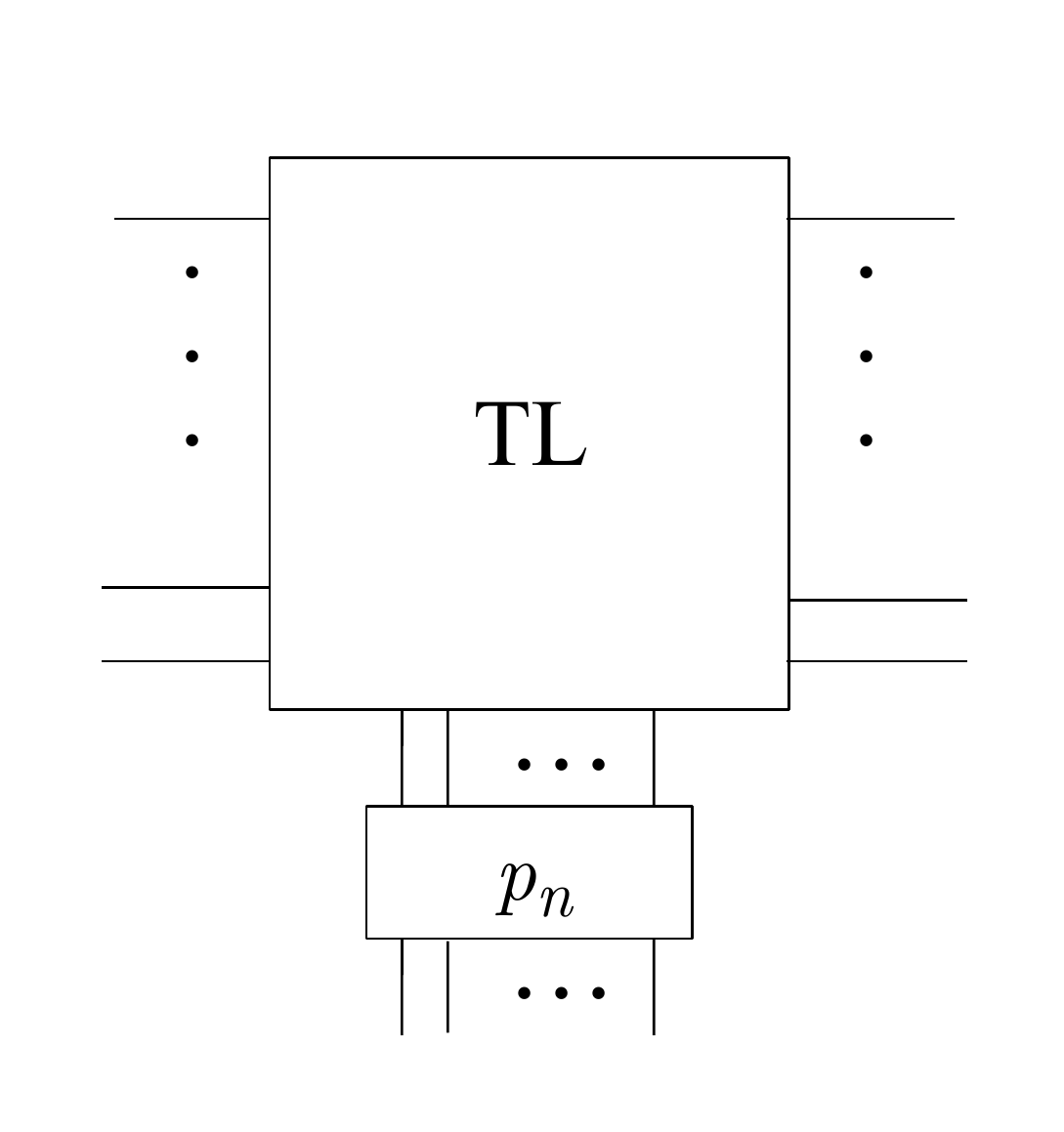}\\
  \caption{Temperley-Lieb bimodules}\label{}
\end{figure}
where $p_n$ is a Jones-Wenzl projection \cite{J83,W87} - a linear combination of TL diagrams whose coefficients are complicated but can be known explicitly \cite{FK,GL97,M17}. The action of $R$ would be by concatenation of diagrams on the left and right. But we have a problem here.
If $n$ is odd (which is the case for the generator $V_1$ that we are after), the number of boundary points on the left and right of the diagram are necessarily different modulo $2$.
This means that the embeddings for the direct limit Hilbert space involve sloping lines which create an insurmountable left-right asymmetry.
One obtains a system of $R-R$ bimodules with infinitely many simple objects.
This is the same construction as turning all the various $N-M$, etc.  bimodules into $R-R$ bimodules using the “shift by one” endomorphism of $R$ coming from adding one string to the left of a TL diagram.

We see we have a parity problem with this approach. Our observation in this paper is that the parity can be corrected using the “Motzkin algebra” of \cite{BH}.
From our point of view the Motzkin algebra is just the planar algebra generated by TL and a single “1-box” which is pictorially represented by a string ending in that 1-box, which is just a string ending in the diagram \cite{J99}.
Thus a Motzkin diagram is one of the form:
\begin{figure}[H]
  \centering
  \includegraphics[width=7cm]{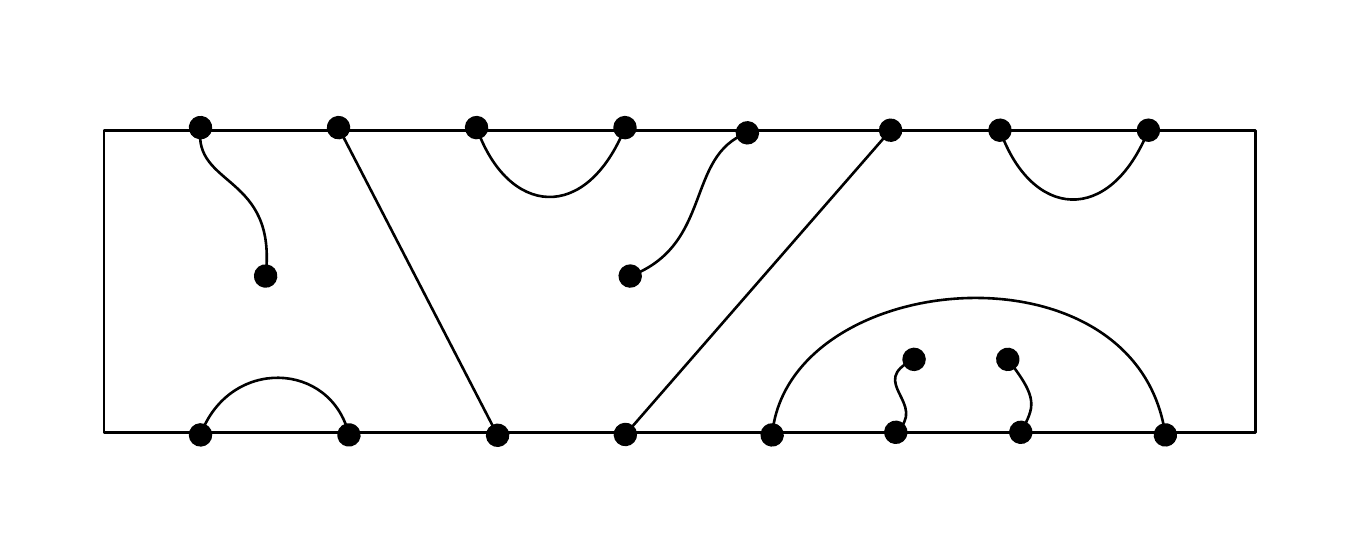}\\
  \caption{A Motzkin diagram}
\end{figure}
So Motzkin diagrams can have an odd number of boundary points.

In \cite{BH}, Benkart and Halverson gave a penetrating analysis of the Motzkin algebras from the algebraic point of view, calculating the “generic” structure in great detail.
But the main thing for our construction is positive definiteness at roots of unity, i.e. when the loop parameter of $M_n(D)$ is $D=1+2\cos(\pi/n)$ for $n=3,4,5,\dots$.
We will demonstrate positivity and exhibit the relevant quotient of the abstract Motzkin algebra by using the basic construction of \cite{J83} for finite dimensional $C^*$-algebras with positive definite trace.
For completeness we also calculate the dimensions and principal graphs of the corresponding towers of algebras. We also introduce the “Motzkin” Jones-Wenzl idempotents
which play the role of the usual TL Jones-Wenzl idempotents.

Once positivity is established,
the construction goes smoothly as outlined above except that
the diagrams of figure are allowed to be Motzkin diagrams.
The subset of all bimodules obtained in this way is closed under tensor product and has $A_n$ as its fusion graph.

Finally, we get the whole fusion algebra for the Motzkin bimodules at $D=2\cos\frac{\pi}{n}$, $n\geq 3$.
As a subalgebra, we have achieved our goal of constructing a family of irreducible $R-R$ bimodules with the required fusion rules. It is perhaps debatable whether this approach is “simpler” than constructing the whole TL fusion category and applying Yamagami’s theorem.

\section{Motzkin Algebras}

Motzkin algebra is an example of planar algebra \cite{J99} and the algebra structure we are about to define is exactly the one coming from planar algebra structure.
In this section, we introduce Motzkin diagrams with the notations from \cite{BH} (see also \cite{Wah20} that has been brought to our attention for related results).  
Then we define the algebra structure and a sequence of special idempotents within it.

\subsection{Motzkin Diagrams}

In this section, we introduce the Motzkin planar algebra.
It is an example of the planar algebra \cite{J99,J19} and the algebra structure is exactly the one from general planar algebra.
The algebra structure on diagrams in \cite{BH} is a consequence of the planar algebra picture. 
(One need to identify rectangles with circles, the marked boundary interval with left side of the rectangle.)

\begin{figure}[H]
  \centering
  \includegraphics[width=7cm]{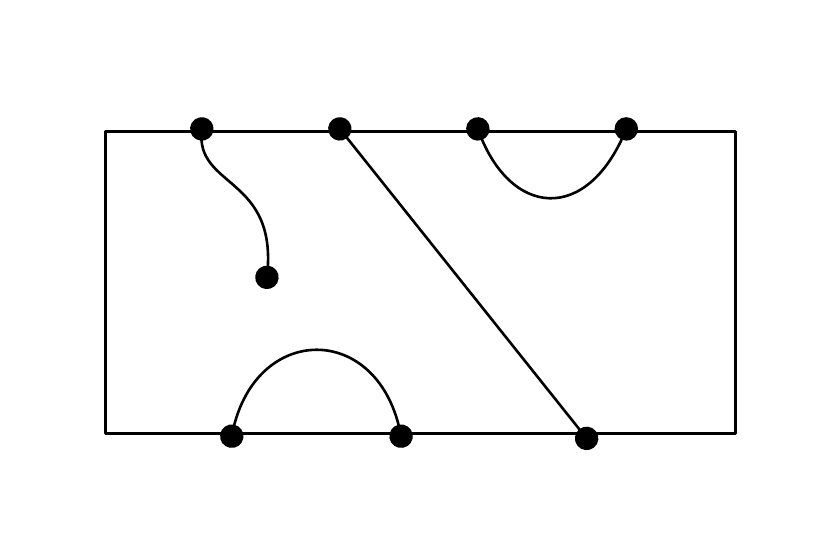}\\
  \caption{A Motzkin $(4,3)$-tangle}
\end{figure}

\begin{definition}
Given $m,n \in \mathbb{N}$, we define
a rectangular Motzkin $(m,n)$-tangle $x$ as the smooth isotopy class of a planar graph in the real plane as follows.
\begin{enumerate}
\item A rectangle bounded by $0\leq x\leq 1$ and $0\leq y\leq 1$.
\item There are $m$ vertices on the upper edge: \begin{center}{$(\frac{1}{m+1},1),\dots,(\frac{m}{m+1},1)$}\end{center}
    and $n$ vertices on the lower edge:  \begin{center}{$(\frac{1}{n+1},0),\dots,(\frac{n}{n+1},0)$.}\end{center}
\item Each vertex is connected to at most one other vertex by smooth curves in the rectangle called strings which do not intersect with each other.
\end{enumerate}
We denote the set of all Motzkin $(m,n)$-tangles by $\mathbf{M}(m,n)$. 
If $m+n=k$, an Motzkin $(m,n)$-tangle is also called Motzkin $k$-tangle and denoted by $\mathbf{M}(k)$. 

\end{definition}

One can check if $m_1+n_1=m_2+n_2=k$, then both the cardinalities of  $(m_1,n_1)$-tangles and Motzkin $(m_1,n_1)$-tangles are the same, which is $\mathcal{M}_k$. 
We prove that they are the well-known Motzkin numbers \cite{Mot}.
The Motzkin numbers can form a sequence: 
\begin{center}
    $1, 1, 2, 4, 9, 21, 51, 127, \dots$
\end{center}
\begin{lemma}
$\mathcal{M}_{k}=\sum\limits_{i=0}^{\lfloor n/2 \rfloor}\frac{1}{i+1}\binom{k}{2i}\binom{2i}{i}$.
Moreover, $\frac{1-t-\sqrt{1-2t-3t^2}}{2t^2}=\sum_{k\geq 0}\mathcal{M}_{k}t^k$
is the generating function for $\mathcal{M}_{k}$.
\end{lemma}
\begin{proof}
For any Motzkin tangle $x$, we consider an arbitrary marked boundary point. 

If it is connected to no other boundary points, there are certainly $\mathcal{M}_{k-1}$ possibilities for the remaining $k-1$ boundary points.

Otherwise, it is connected to the $i$-th boundary point (counting clockwise), $2 \leq i\leq m+n-1$.
Then the tangle is separated by this string into two parts.
And these two parts are $(i-1)$-tangle and $(m+n-i-1)$-tangle.
There are totally $\sum\limits_{i=1}^{k-1}\mathcal{M}_{i-1}\mathcal{M}_{k-i-1}$ choices of connections.

Then the numbers $\mathcal{M}_{k}$'s can be obtained by induction.
\end{proof}

\begin{definition}
Given $x\in \mathbf{M}(m,n)$ and $y \in \mathbf{M}(p,q)$, we let
\begin{enumerate}
\item $x^*\in \mathbf{M}(m,n)$, the adjoint of $x$, be the $(n,m)$-tangle that interchanges the upper and lower edges and vertices of $x$ but keeps all connections.
\item $x|y\in \mathbf{M}(m+p,n+q)$ be the $(m+p,n+q)$-tangle juxtaposed by $x,y$.
\end{enumerate}
\end{definition}

\begin{figure}[H]
  \centering
  \includegraphics[width=10cm]{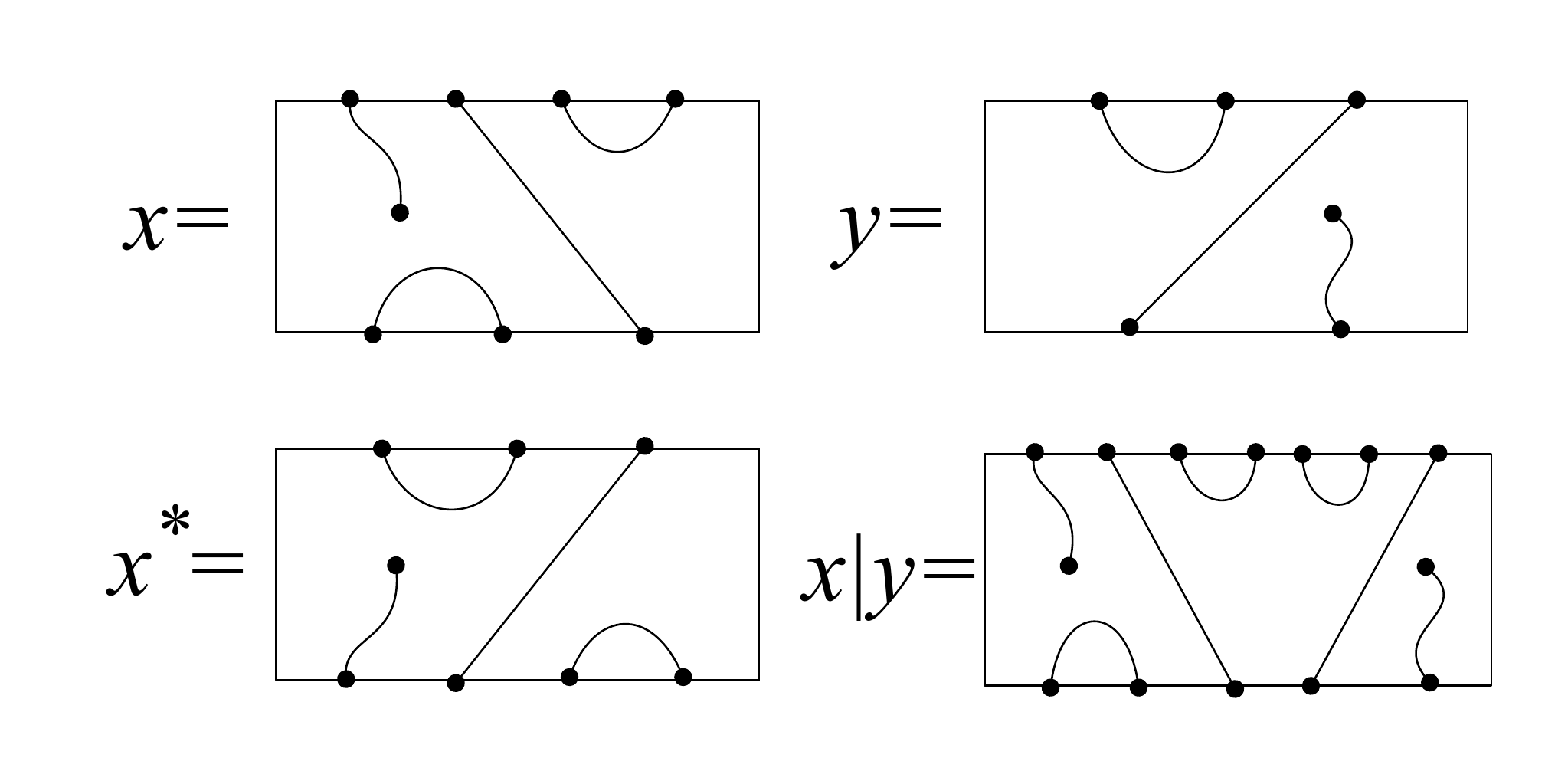}\\
  \caption{Adjoint $x^*$ and Juxtaposition $x|y$}
\end{figure}


\begin{definition}
A Motzkin $n$-diagram is a Motzkin $(n,n)$-tangle.
We denote the set of all Motzkin $n$-diagrams by $\mathbf{M}_n$ .
\end{definition}
\begin{figure}[H]
  \centering
  \includegraphics[width=7cm]{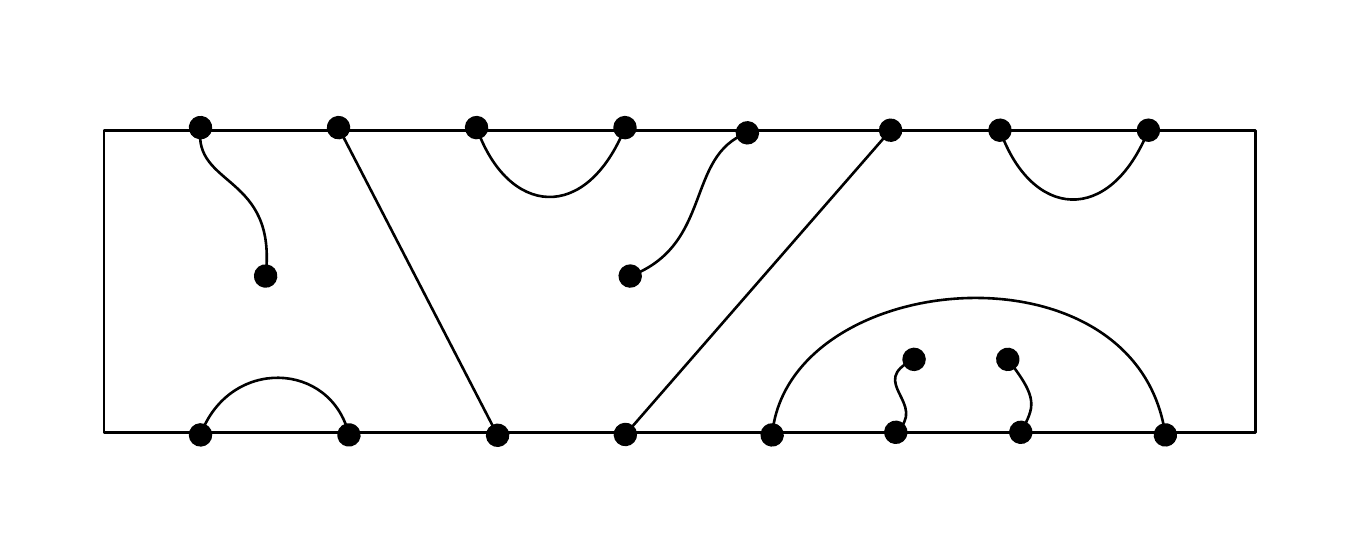}\\
  \caption{A Motzkin $8$-diagram of rank $2$}\label{}
\end{figure}
Note that each Motzkin $n$-diagram defined above has $2n$ boundary points which is a Motzkin $2n$-tangle.

Moreover, for $x \in \mathbf{M}_n$, we define $\rk(x)$ to be the number of through strings.
Let $\mathbf{M}_{n,r}$ be the set of all Motzkin $n$-diagrams of rank $r$.

\subsection{The Motzkin Algebra $M_n(D)$}

Assume $F$ is an arbitrary field and we take $D\in F$.
With the parameter $D$, we will define the algebra structure on $\mathbf{M}_n$, which is the usual algebra structure of a planar algebra \cite{J99}.

We define the product of two Motzkin $n$-diagrams $x,y$,  $xy=D^{\kappa(x,y)}z$, as follows.
\begin{enumerate}
\item Identify the bottom-row vertices of $x$ and the top-row vertices of $y$.
\item $z$ is a Motzkin $n$-diagram with its top-row vertices from top-row vertices of $x$ and its bottom-row vertices from bottom-row vertices of $y$.
\item Join the strings smoothly from the bottom-row of $y$ to the top-row of $x$.
\item $\kappa(x,y)$ is the number of loops inside.
\end{enumerate}
Obviously, the Motzkin $n$-diagram with $n$ vertical strings is the identity element in $M_n(D)$ and we denote it by $1_n$.
\begin{figure}[H]
  \centering
  \includegraphics[width=7cm]{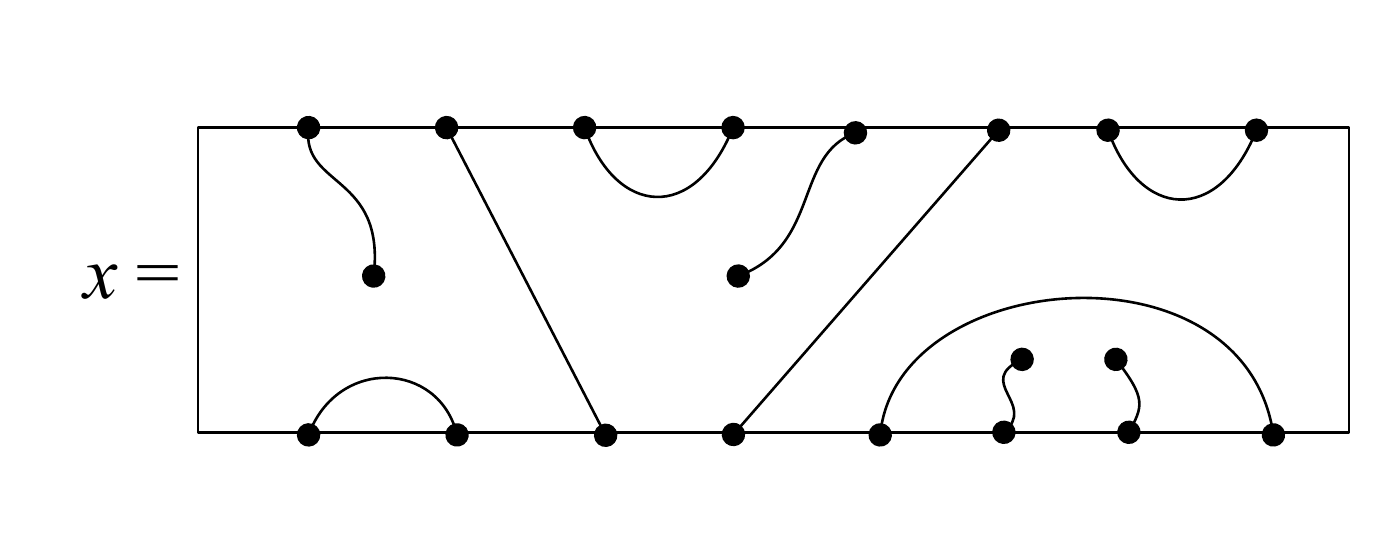}\\
  \caption{Motzkin $n$-diagram $x$}\label{}
  \includegraphics[width=7cm]{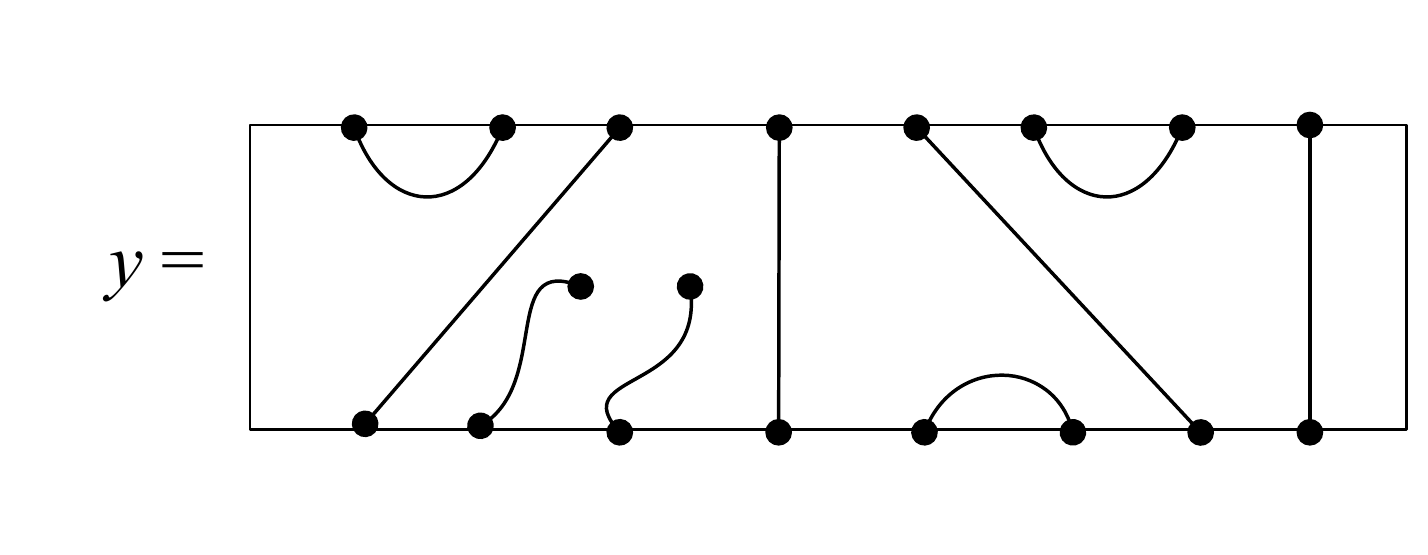}\\
  \caption{Motzkin $n$-diagram $y$}\label{}
\end{figure}
\begin{figure}[H]
  \centering
  \includegraphics[width=14cm]{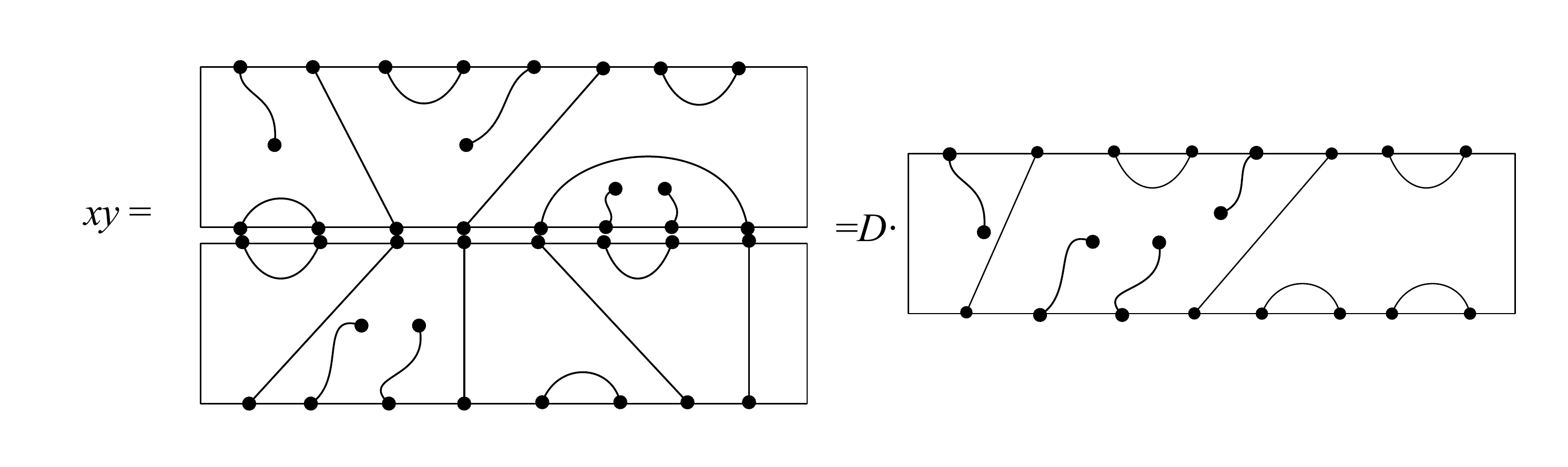}\\
  \caption{Example of a product: $xy$}\label{}
\end{figure}

\begin{definition}
Given $n\in \mathbb{N}$, the Motzkin algebra $M_n(D)$ is the unital associative algebra over $F$ generated by
the $n$-Motzkin diagrams with multiplication defined above. 

If the field $F$ has an involution $\sigma$, $M_n(D)$ becomes a $*$-algebra under the  $\sigma$-linear extension from the basis $\mathbf{M}_n$ to $M_n(D)$. 
\end{definition}

We have a natrual embedding $M_n(D)\to M_{n+1}(D)$ where we add one more through string on the right.

Note that for two Motzkin $n$-diagrams $x,y$, we have
\begin{center}
$\rk(xy)\leq \min(\rk(x),\rk(y))$.
\end{center}

Let $\mathbf{M}_{n,r}=\{x\in \mathbf{M}_n|\rk(x)\leq r\}$ be the set of $n$-diagrams with rank no greater than $r$.
We define
\begin{center}
$I_{n,r}=F\cdot \mathbf{M}_{n,r}=\spn_{F}\{x\in\mathbf{M}_n|\rk(x)\leq r\}$
\end{center}
to be the subalgebra linearly spanned by diagrams with rank no greater than $k$. It is a $2$-sided ideal in $M_n(D)$.
Also, we let $M_{n,r}$ be the $F$-linear space generated by $\{x\in \mathbf{M}_n|\rk(x)= r\}$.  we have a tower of two-sided ideals
\begin{center}
$I_{n,0}\subseteq I_{n,1} \subset \cdots \subset I_{n,n}=M_n(D)$.
\end{center}

For the algebra $M_n(D)$, let's consider the elements $r_i,l_i,e_i$ with $1\leq i\leq n-1$ defined by the following diagrams.
\begin{figure}[H]
  \centering
  \includegraphics[width=12cm]{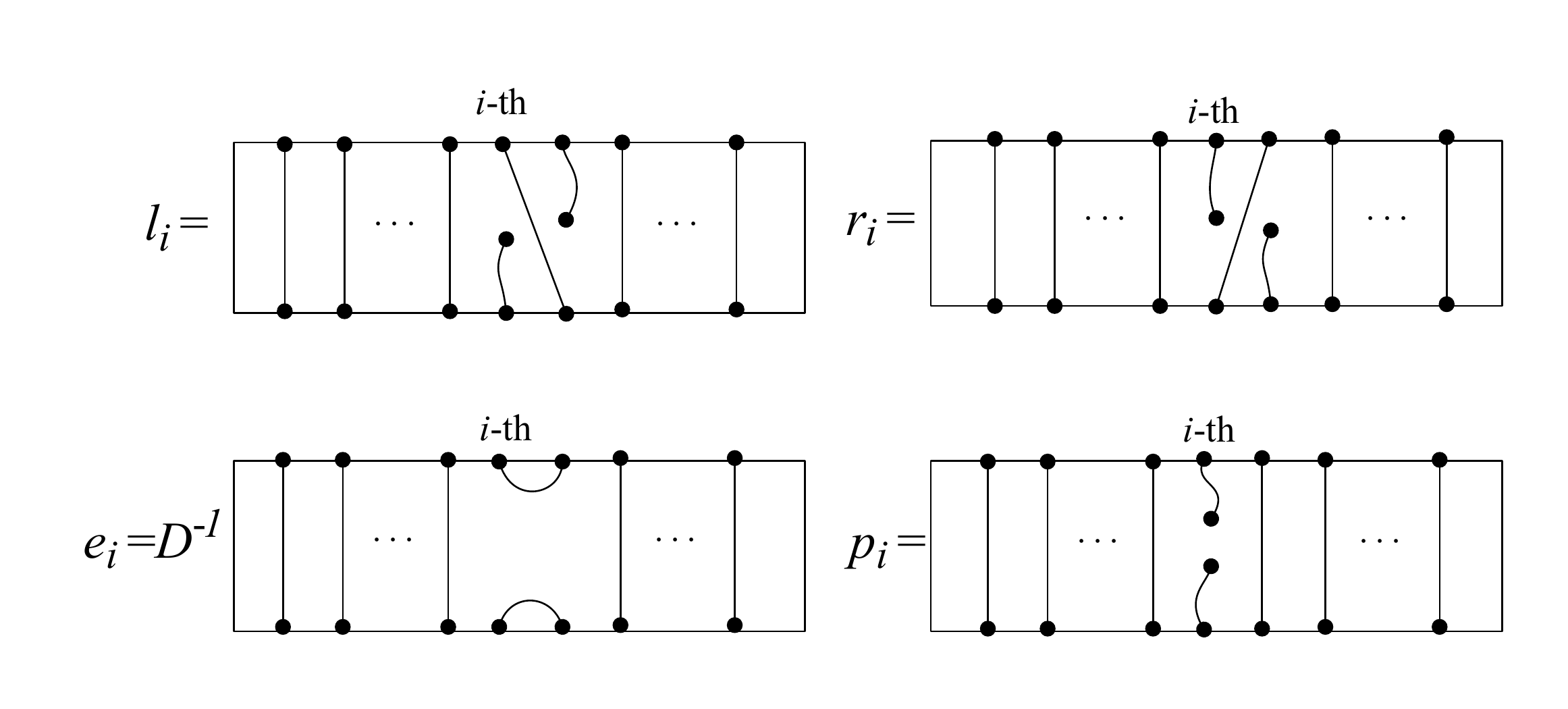}\\
  \caption{$l_i,r_i,e_i,p_i$}\label{}
\end{figure}

We let $p_i$ denote the diagram defined above, i.e. the one with $n-1$ vertical edges while its $i$-th pair of vertices are not connected.

Then we have $p_1=r_1l_1$, $p_i=r_il_i=l_{i-1}r_{i-1}$ for $2\leq i\leq n-1$ and $p_n=l_{n-1}r_{n-1}$.
One can show that $M_n(D)$ is generated by $1_n$ and the elements $e_i,l_i,r_i$ (note that we have no $p_i$) defined above:
\begin{center}
$M_n(D)=\langle 1_n,r_i,l_i,e_i|1\leq i\leq n-1\rangle$.
\end{center}

As in \cite{J99}, we have a canonical linear functional $\tr_n$ called trace on $M_n(D)$ defined by:
\begin{figure}[H]
  \centering
  \includegraphics[width=12cm]{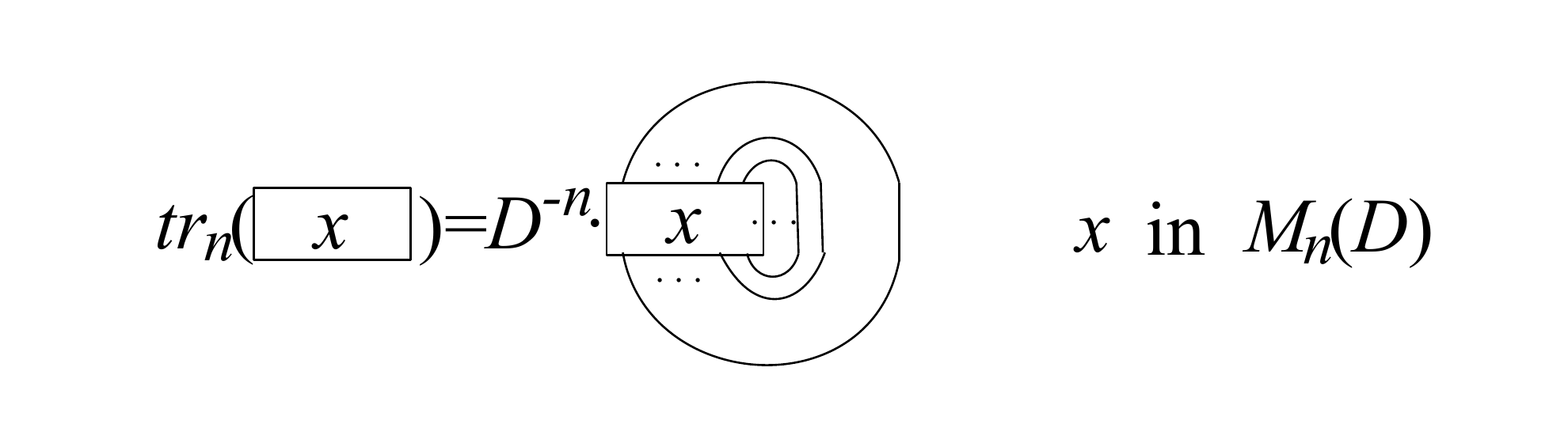}\\
\end{figure}
which satisfies $\tr(xy)=\tr(yx)$ and is also compatible with respect to the embedding $i_{n+t}:M_n(D)\hookrightarrow M_{n+t}(D)$:
\begin{center}
$\tr_n(x)=\tr_{n+t}(i_{n+t}(x))$, $x\in M_n(D)$.
\end{center}

\begin{definition}
We call the trace on $M_n(D)$ defined above
\begin{enumerate}
\item nondegenerate if the radical of $\tr$ is trivial, i.e. \begin{center}
    $\{x\in M_n(D)|\tr(xy)=0$ for all $x\in M_n(D)\}=0$,
    \end{center}
\item positive if $F=\mathbb{C}$ or $F\mathbb{R}$ and $\tr(xx^*)>0$ for all $x \in M_n(D)$
\end{enumerate}
\end{definition}

By \cite{J99}, the trace gives a general planar algebra a $C^*$-algebra structure if it satisfies some certain properties.
We will return to this in Section 3.

\subsection{The Motzkin Paths}

In this part, we mainly consider the $n$-digrams of a fixed rank $r$.
Most results here are already done in \cite{BH}.

\begin{definition}
A Motzkin $n$-path is a sequence of $p=(a_1,\dots,a_n)$ with $a_i\in \{-1,0,1\}$ such that $a_1+\cdots+a_k\geq 0$ for all $1\leq k\leq n$.
Let $\mathcal{P}_n$ denote all Motzkin $n$-paths.

Define the rank of a Motzkin $n$-path $p$ to be $a_1+\cdots+a_n$ and denote it by $\rk(p)$.
Let $\mathcal{P}_{n,r}=\{p\in \mathcal{P}_n|\rk(p)=r\}$.
\end{definition}

Also, we let $\mathcal{P}_{n,r}=\phi$ if $r<0$ or $r>n$.
Let $m_{n,r}$ be the cardinality of $\mathcal{P}_{n,r}$. 
\begin{lemma}
For $n\geq 2$, we have
\begin{enumerate}
\item $m_{n,0}=m_{n-1,0}+m_{n-1,1}$,
\item $m_{n,r}=m_{n-1,r-1}+m_{n-1,r}+m_{n-1,r+1}$ for $1\leq r\leq n-2$,
\item $m_{n,n-1}=m_{n-1,n-2}+m_{n-1,n-1}$,
\item $m_{n,n}=1$.
\end{enumerate}
Moreover, it can be counted by the following graph where each number here stands for the number of descending paths from the top point to it.
\end{lemma}

\begin{figure}[H]
  \centering
  \includegraphics[width=8cm]{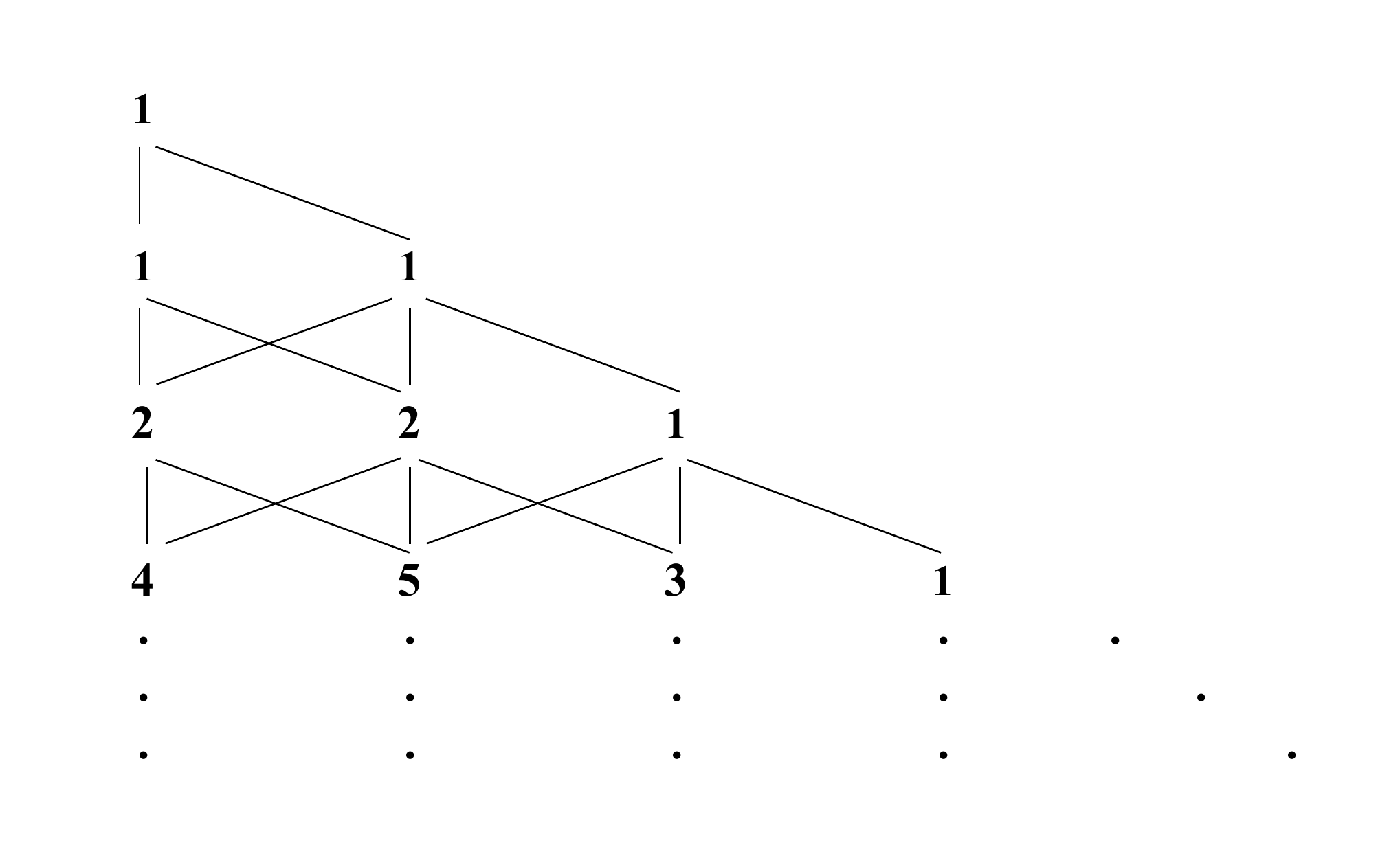}\\
  \caption{Graph for $m_{n,r}$}\label{}
\end{figure}

\begin{proof}
Since $(1,\dots,1)$ is the only path in $\mathcal{P}_{n,n}$,  we have $m_{n,n}=1$ for all $n$.

Now, take a fixed $r$ with $0\leq r\leq n$,
by definition, we have
\begin{center}
$\mathcal{P}_{n+1,r}=\{(a_1,\dots,a_{n+1})|a_k\in \{-1,0,1\},\sum\limits_{i=1}^{k}a_i \geq 0,\sum\limits_{i=1}^{n+1}a_i=r,\forall 1\leq k\leq n+1\}$.
\end{center}

Let us consider the first $n$ terms $a_1,\dots,a_n$.
As $a_{n+1}\in \{-1,0,1\}$, we have $\sum\limits_{i=1}^{n}a_i$ is $r-1,r$ or $r+1$.
So $(a_1,\dots,a_n)\in \mathcal{P}_{n,r-1},\mathcal{P}_{n,r}$ or $\mathcal{P}_{n,r+1}$.
(Note that the cases $r=0,n,n+1$ will lead to empty set).

Conversely, given a path $(a_1,\dots,a_n)\in \mathcal{P}_{n,r-1},\mathcal{P}_{n,r}$ or $\mathcal{P}_{n,r+1}$, the path $(a_1,\dots,a_{n+1})$ where $a_{n+1}=1,0,-1$ respectively is in $\mathcal{P}_{n+1,r}$.
Hence we have a bijection between sets $\mathcal{P}_{n+1,r}$
and $\bigsqcup\limits_{k=r-1}^{k=r+1}\mathcal{P}_{n,k}$, which establishes all the equalities.

\end{proof}

Then, by induction, we have the following consequence.

\begin{corollary}
$m_{n,r}=\sum\limits_{i=0}^{\lfloor (n-r)/2\rfloor}\binom{n}{r+2i}\stirlingii{r+2i}{i}$, where $\stirlingii{n}{m}=\binom{n}{m}-\binom{n}{m-1}$ for $1\leq m\leq \lfloor n/2\rfloor$.
\end{corollary}

\subsection{A Sequence of Idempotents}

We introduce a sequence of idempotents $g_k\in M_k(D)$.
The construction of these idempotents depends on the parameter $D$.

Let $\{P_n(x)\}_{n\geq 0}$ be the Chebyshev polynomials over $F$.
They are defined as follows.
\begin{center}
$P_0(x)=P_1(x)=1$, $P_{n+1}(x)=P_n(x)-x\cdot P_{n-1}(x)$.
\end{center}
Let  $d=D-1$ and $\tau=d^{-2}$.
\begin{definition}
For $D\in F$ and $\tau=(D-1)^2$, we say $D$ is $n$-generic if $P_k(\tau)\neq 0$ for $1\leq k\leq n$ and generic if $P_k(\tau)\neq 0$ for $k\in \mathbb{N}$.
\end{definition}

Suppose $D$ is $n$-generic, we define a sequence of idempotents $g_k\in M_k(D)$ for $1\leq k\leq n-1$ inductively by
\begin{center}
$g_{k+1}=g_k\cdot (1-p_{k+1})-\frac{D}{d}\frac{P_{k-1}(\tau)}{P_{k}(\tau)}g_k e_k g_k$,
\end{center}
with $g_1=1-p_1$. We set $\lambda_k=\frac{D}{d}\frac{P_{k-1}(\tau)}{P_{k}(\tau)}$.

Moreover, for each $n\geq 2$, we define $E:M_n(D)\to M_{n-1}(D)$ by the  graphical action shown in Figure 11.
\begin{figure}[H]
  \centering
  \includegraphics[width=5cm]{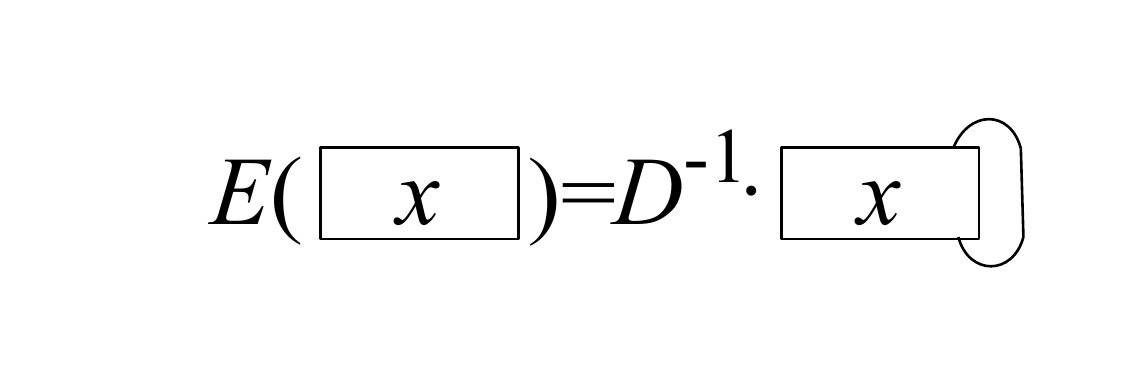}\\
  \caption{$E:M_n(D)\to M_{n-1}(D)$}\label{}
\end{figure}

\begin{proposition}
Suppose $D$ is $n$-generic, for all $2\leq k\leq n$, we have
\begin{enumerate}
\item $g_kl_i=g_kr_i=g_ke_i=l_ig_k=r_ig_k=e_ig_k=0$ for $1\leq i\leq k-1$,
\item $g_k=g_k^*=g_k^2$,
\item $E(g_k)=\frac{d\cdot P_{k}(\tau)}{D\cdot P_{k-1}(\tau)}g_{k-1}$,  
\item $g_ig_j=g_j$ if $1\leq i\leq j\leq n$,
\item $g_k=1+\sum_{i}c_i w_i$ where $c_i\in F$ and $w_i\neq 1$ is a word of $\{l_j,r_j,e_j|1\leq j \leq k-1\}$.
\end{enumerate}
\end{proposition}
\begin{proof}
Note all of these hold for $g_1,g_2$.
Suppose these also hold for all $2\leq i\leq k$ where $k\leq n-1$.

For 1, by definition, we have $g_{k+1}l_i=g_{k+1}r_i=g_{k+1}e_i=l_ig_{k+1}=r_ig_{k+1}=e_ig_{k+1}=0$ for $1\leq i\leq k-1$.
It suffices to prove they hold for $i=k$.
\begin{align*}
 r_k g_{k+1}&=r_kg_k (1-p_{k+1})-\lambda_k  r_k g_k e_k g_k\\
 &=r_kp_{k+1}g_k(1-p_{k+1})-\lambda_k r_k p_{k+1} g_k e_k g_k \\
  &=r_k g_k p_{k+1}(1-p_{k+1})-\lambda_k r_k g_k p_{k+1}e_k g_k \\
  &=0-\lambda_k r_k g_k p_{k}e_k g_k=0,
\end{align*}
where we apply $p_{k+1}g_k=g_kp_{k+1}$ and $p_{k+1}e_k=p_{k}e_k$. And for $e_k$, we have
\begin{align*}
 e_k g_{k+1}&=e_k g_k p_{k+1}-\lambda_k  e_k g_k e_k g_k\\
 &=e_k g_k p_{k+1}-\lambda_k  e_k E(g_k)g_k\\
 &=e_k( g_k p_{k+1}-g_k)=-e_k p_{k+1}g_k=-e_k p_{k}g_k=0,
\end{align*}
where we apply $E(g_k)=\frac{dP_{k}(\tau)}{P_{k-1}(\tau)}g_{k-1}$, $g_{k-1}g_k=g_k$ and $p_{k+1}e_k=p_{k}e_k$.

So we have proved $l_kg_{k+1}=e_kg_{k+1}=0$.
Then, by symmetry, we have $g_{k+1}l_k=g_{k+1}r_k=g_{k+1}e_k=l_kg_{k+1}=r_kg_{k+1}=e_kg_{k+1}=0$.

For 2, it is obvious that $g_{k+1}^*=g_{k+1}$. It suffices to prove $g_{k+1}$ is an idempotent:
\begin{align*}
 g_{k+1}^2&=(g_k (1-p_{k+1})-\lambda_k   g_k e_k g_k)^2\\
 &=g_k (1-p_{k+1})-2\lambda_k Dd^{-1} g_k e_kg_k+\lambda_k^2g_k e_k g_ke_k g_k\\
 &=g_k (1-p_{k+1})+\lambda_k  g_k e_kg_k=g_{k+1}.
\end{align*}

Affirmation for 3,4 and 5 is now easy.
\end{proof}

\begin{corollary}
For $x\in M_n(D)$ with $n$-generic $D$, $xg_k\neq 0$ if $x$ is a linear combination of diagrams whose ranks are less than $k$.
\end{corollary}

\begin{corollary}
For $x\in M_n(D)$ with $n$-generic $D$, $xg_n\neq 0$ if and only if $x$ contains a nonzero scalar, i.e $c \cdot 1_n$ with $c\neq 0$ as a summand.
\end{corollary}

Here we prove a lemma that will apply in next section.

\begin{lemma}
Let $w\neq 1$ be a word of $\{e_i,l_i,r_i|1\leq i\leq n-1\}$, then $r_nw,wl_n$ are words in $\{e_i,l_i,r_i|1\leq i\leq n-1\}$ and $e_n$.
\end{lemma}
\begin{proof}
Without loss of generality, we will only prove this for $r_nw$.
As $w$ is not the identity of $M_n(D)$, we have $\rk(w)\leq n-1$.
There are two cases.
\begin{enumerate}
\item $\exists p_j$ with $1\leq j\leq n$ such that $p_{w}w=w$.
\item $\exists e_l$ with $1\leq l\leq n-1$ such that $e_{l}w=w$.
\end{enumerate}

For the first case, if $j\leq n-1$, we have
\begin{align*}
r_nw&=r_np_jw=(r_j\cdots r_{n-2})p_{n-1}r_n( l_{n-2}\cdots l_j)w\\
&=(r_j\cdots r_{n-2})(r_{n-1} e_{n-1} e_{n} l_{n-1})( l_{n-2}\cdots l_j)w,
\end{align*}
and if $j=n$,
\begin{align*}
r_nw&=r_np_nw=p_np_{n+1}w=l_{n-1}r_{n-1}e_n l_{n-1}r_{n-1}w.
\end{align*}
Both of them are also words in $\{e_i,l_i,r_i|1\leq i\leq n-1\}$ and $e_n$.

For the second case, if $l\leq n-2$, we have
\begin{align*}
l_nw&=l_ne_lw=(e_l\cdots e_{n-3})e_{n-2}l_n( e_{n-3}\cdots e_l)w\\
&=(e_l\cdots e_{n-3})e_{n-2}r_{n-1}e_n e_{n-1}e_{n-2}( e_{n-3}\cdots e_l)w.
\end{align*}
And if $l=n-1$, we have
\begin{align*}
l_nw&=l_ne_{n-1}w=r_{n-1}e_n e_{n-1}w.
\end{align*}
Both of them are also words in $\{e_i,l_i,r_i|1\leq i\leq n-1\}$ and $e_n$.
\end{proof}

\section{The $C^*$-Algebra Structure}

We assume $F=\mathbb{C}$ from now on and give $M_n(D)$ their $*$-algebra strucure.
This section will be mainly devoted to prove the following two theorems.
\begin{theorem}
If $D$ is $(n-1)$-generic , $M_k(D)$ is semisimple for $1\leq k\leq n-1$ and \begin{center} $M_k(D)\cong\bigoplus\limits_{i=0}^{k}\mathrm{Mat}_{m_{k,r}}(\mathbb{C})$.
\end{center}
where $m_{k,r}=\sum\limits_{i=0}^{\lfloor (k-r)/2\rfloor}\binom{k}{r+2i}(\binom{r+2i}{i}-\binom{r+2i}{i-1})$.
\end{theorem}

\begin{theorem}
The trace $\tr$ on $M_{\infty}(D)=\bigcup\limits_{k\in\mathbb{N}} M_k(D)$ is positive-semidefinite (with respect to the $*$-structure defined in section 2.1) if and only if
\begin{center}
$D\in\{2\cos\frac{\pi}{n}+1|n\geq 3\}\cup [3,\infty)$.
\end{center}
The representation $\pi$ with respect to $\tr$ of all $M_k(D)$ (constructed in section 3.1) satisfies:
\begin{enumerate}
\item For $D\in [3,\infty)$, $\pi|_{M_n(D)}$ is faithful for all $n\in \mathbb{N}$. And $\pi(M_n(D))\cong \bigoplus\limits_{r=0}^{n}\mathrm{Mat}_{m_{k,r}}(\mathbb{C})$
for any $n\in \mathbb{N}$. It can be described by the Bratteli diagram shown in Figure 12.
\begin{figure}[H]
  \centering
  \includegraphics[width=12cm]{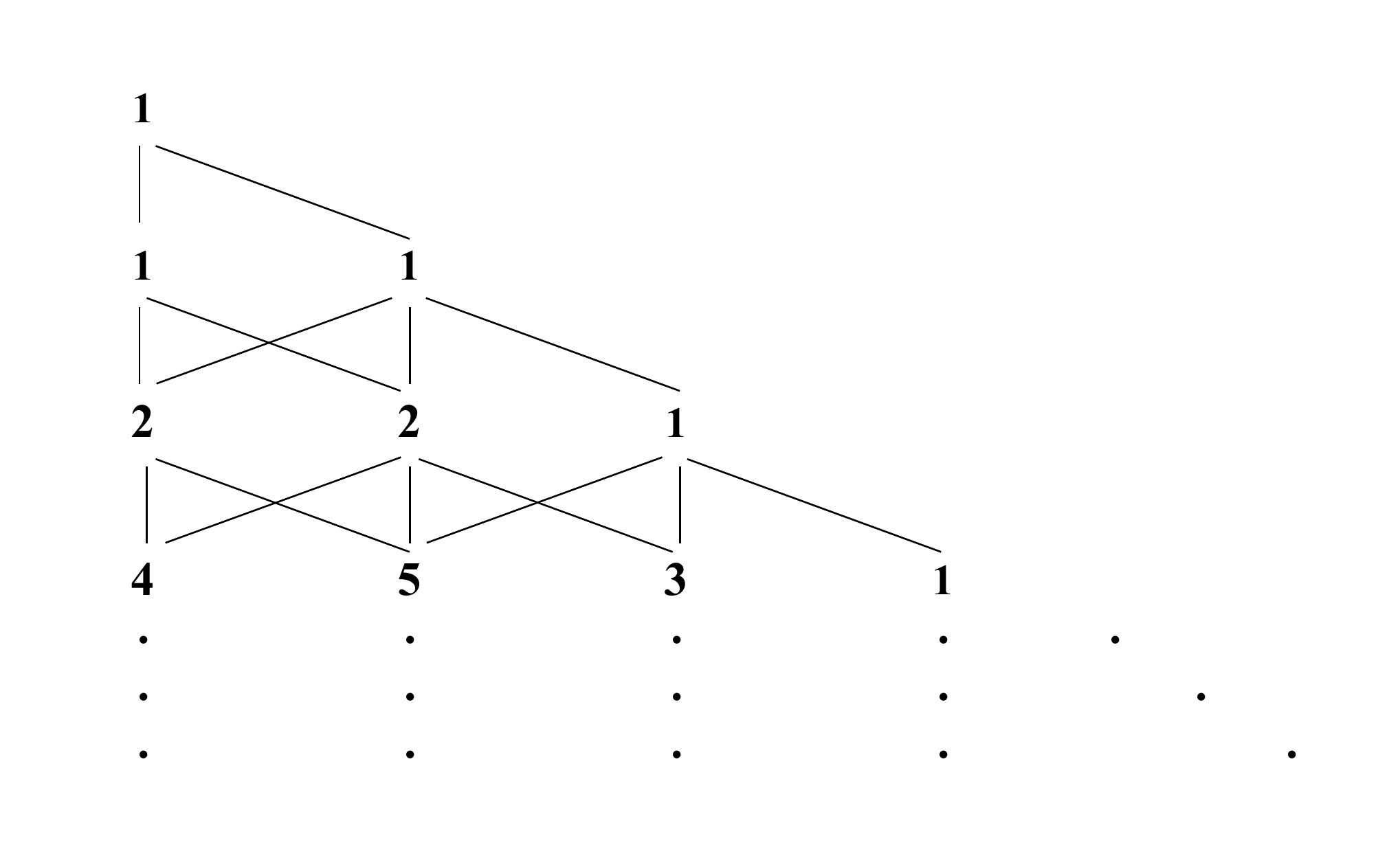}\\
  \caption{The Bratteli diagrams of $\pi(M_n(D))$ for $D\geq 3$} \label{}
  \end{figure}
\item  For $D=2\cos\frac{\pi}{n+1}+1$, $\pi|_{M_k(D)}$ is faithful if and only if $1\leq k\leq n-1$.

\begin{enumerate}
\item For $1\leq k\leq n-1$, $\pi(M_k(D))\cong \bigoplus\limits_{r=0}^{k}\mathrm{Mat}_{m_{k,r}}(\mathbb{C})$.
\item For $k\geq n$, $\pi(M_k(D))\cong \bigoplus\limits_{r=0}^{n-1}\mathrm{Mat}_{l_{k,r}}(\mathbb{C})$.
\end{enumerate}
\begin{figure}[H]
  \centering
  \includegraphics[width=12cm]{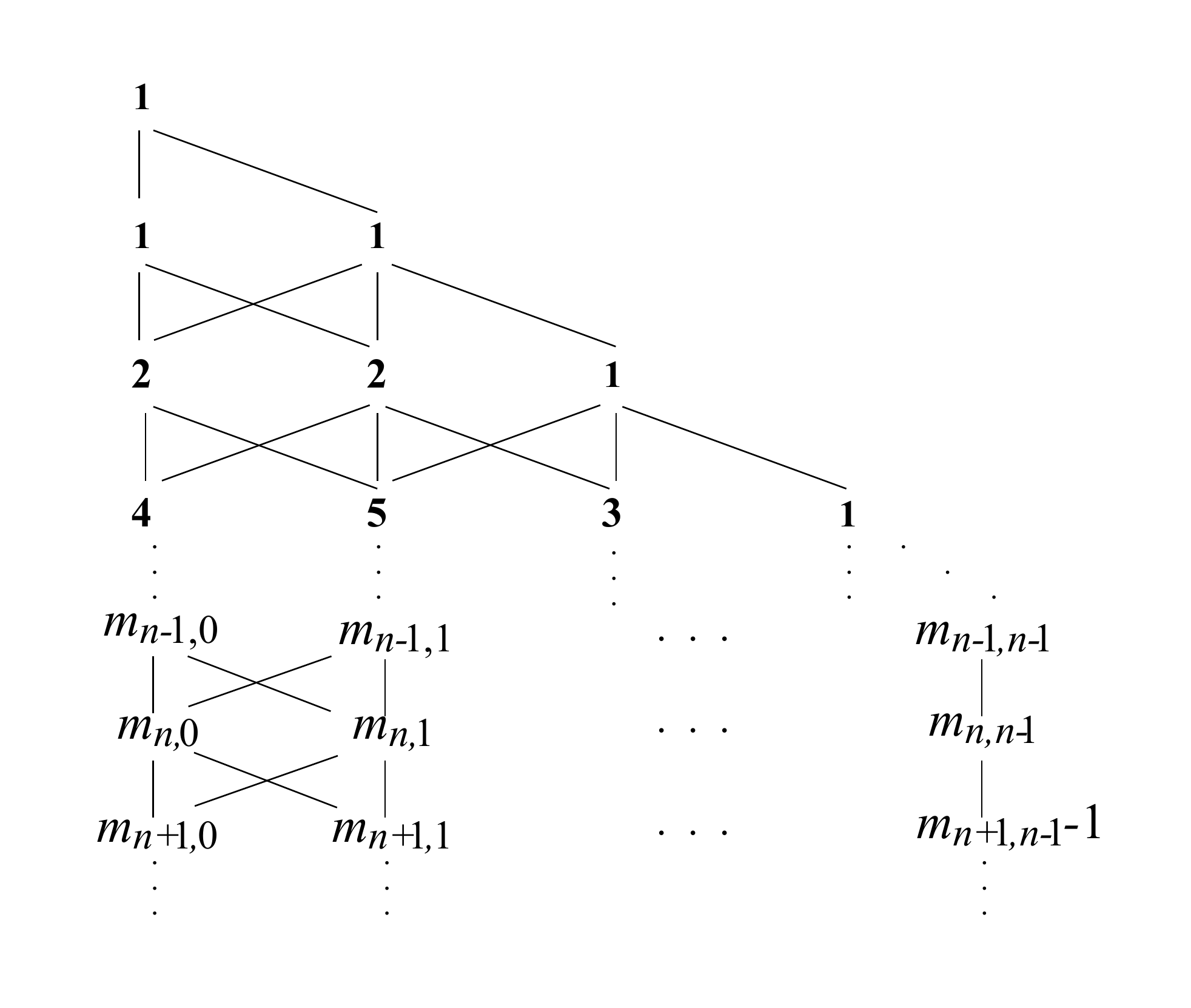}\\
  \caption{The Bratteli diagrams of $\pi(M_n(D))$ for $D=2\cos\frac{\pi}{n+1}+1$} \label{}
  \end{figure}
where the numbers ${l_{k,r}}$ are defined in Section 3.3 and can  also be described by the truncated Bratteli diagram shown in Figure 13. 

Moreover, the principal graphs for these two cases are $A_{\infty}$ and $A_{n}$ in Figure 14.
\begin{figure}[H]
  \centering
  \includegraphics[width=12cm]{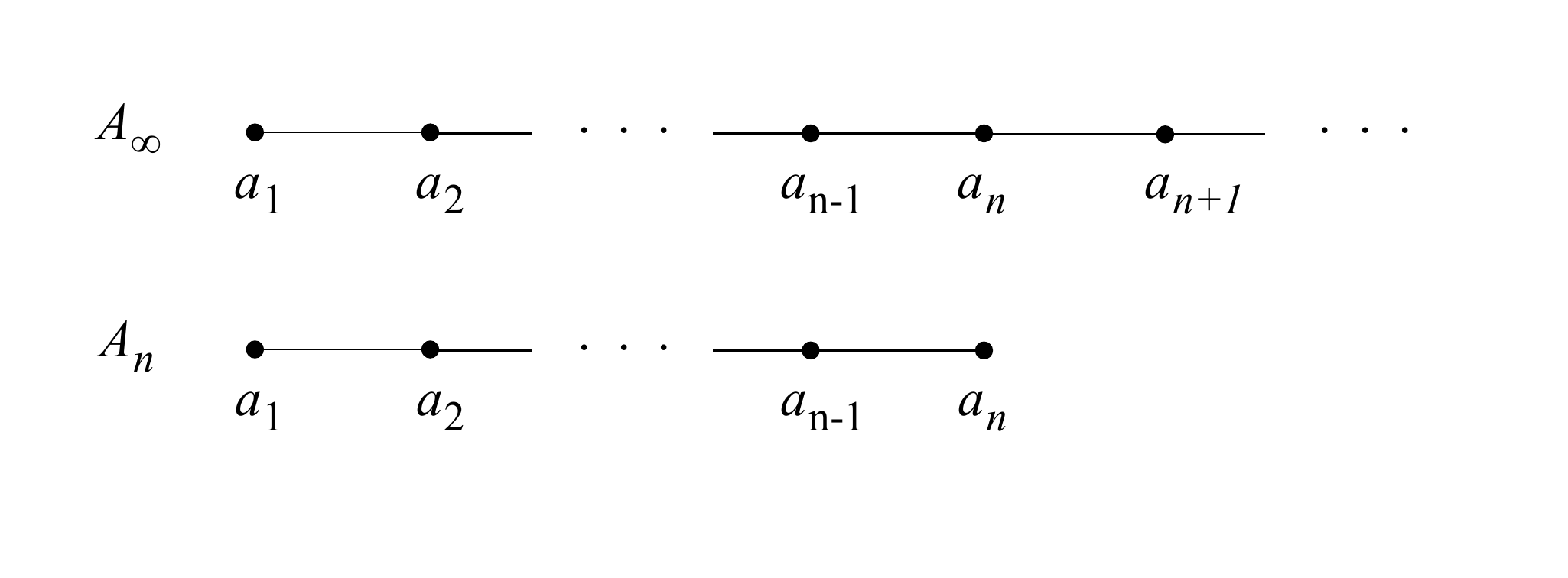}\\
  \caption{Principal graphs $A_{\infty}$ and $A_{n}$} \label{}
  \end{figure}
\end{enumerate}
\end{theorem}

\begin{example}
If $D=2\cos\frac{\pi}{4}+1=\sqrt{2}+1$, the Batteli diagram of $\{M_k(\sqrt{2}+1)\}_{k\geq 0}$ is given by Figure 15. 
One can find the width of the diagram will not get bigger once it reaches $3$. 

\begin{figure}[H]
  \centering
  \includegraphics[width=7cm]{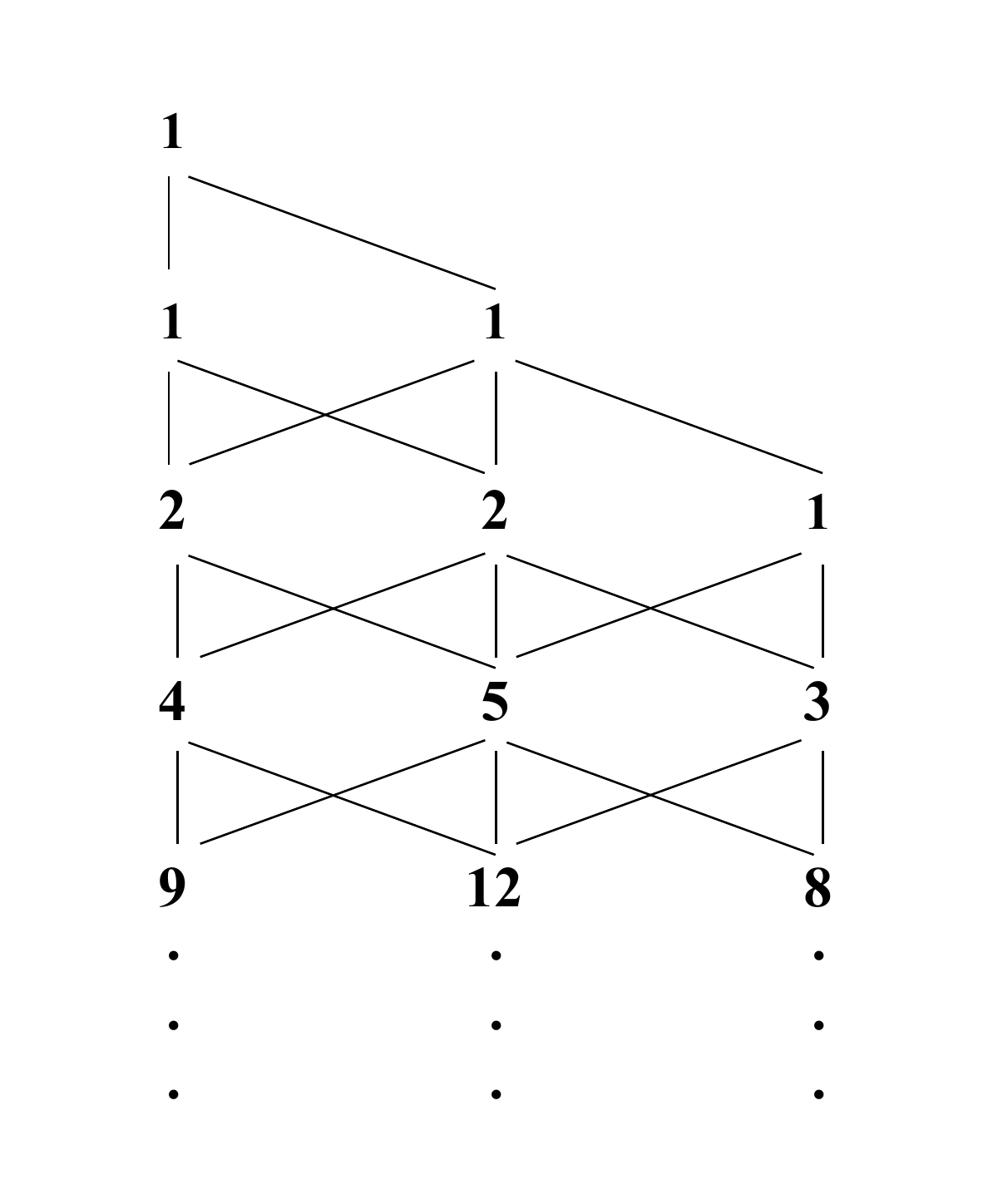}\\
  \caption{The Bratteli diagrams of $\{M_k(\sqrt{2}+1)\}_{k\geq 0}$} \label{}
  \end{figure}
\end{example}

\subsection{Hilbert Space Representations of $M_n(D)$}

In this part, we assume that $tr$ is positive semi-definite, i.e. $\tr_n(x^*x)\geq 0$ for all $x\in M_n(D)$ and for all $n$.

Let $R_n$ be the radical of the trace which is a two-sided ideal of $M_n(D)$, i.e. $R_n=\{x\in M_n(D)|\tr_n(xy)=0 \text{~for all }y\in M_n(D)\}$. We define a finite dimensional Hilbert space:
\begin{center}
$H_n(D)=(M_n(D)/R_n,\langle \cdot,\cdot\rangle_n)$
\end{center}
The inner product is defined as $\langle [x],[y]\rangle_n=\tr(y^*x)$,
where $[y]$ is the image of $y$ under the quotient map $M_n(D)\to H_n(D)=M_n(D)/R_n$.

Define a representation $\pi_k:M_{k}(D)\to B(H_n(D))$ by
\begin{center}
$\pi_k(x)([y])=[xy]$
\end{center}
where $y\in M_k(D)$ and $[y]\in H_k(D)$.
\begin{enumerate}
\item $A_k=\pi_k(M_k(D))$.
\item Let $B_{k+1}=\pi_k(A_ke_kA_k)$ be the subalgebra of $A_{k+1}$ which may not be unital.
\end{enumerate}

\begin{lemma}
The only possible values for $D$ such that $\tr(g_k)\geq 0$ for all $k$ are $\{1+2\cos\frac{\pi}{n}|n\geq 3\}\cup [3,\infty)$.
\end{lemma}
\begin{proof}
By induction, we have $tr(g_n)=\frac{d^n}{D^n} P_n(\tau)\geq 0$.
As $g_n$ is always a projection under the representation, we have $P_n(\tau)\geq 0$ for all $n$.

Then by corollary 4.2.5 of \cite{J83}, $\tau \in \{4\sec^2\frac{\pi}{n}|n\geq 3\}\cup (0,1/4]$.
Then $d\in \{2\cos\frac{\pi}{n}|n\geq 3\}\cup [2,\infty)$ as $\tau d^2=1$.
Hence $D\in \{1+2\cos\frac{\pi}{n}|n\geq 3\}\cup [3,\infty)$.
\end{proof}

\subsection{Basic Constructions}

We will iterate the basic construction to get the structure of these representations.
The basic construction comes from the inclusion of a pair of finite von Neumann algebras.
We mainly follow \cite{J83,TK3} for all the notations below.

Suppose $N\subset M$ is a pair of finite von Neumann algebras, with a faithful normal normalized tracial state $\tr$.
Then there exists a conditional expectation $E_N:M\to N$ defined by $\tr(E_N(x)y)=tr(x,y)$ for $x\in M,y\in N$.

Now, let $M$ act on $H=L^2(M,\tr)$ and $\xi$ be the cyclic vector (the identity of $M$) in $H$.
Then the conditional expectation extends to a projection $e_N$ on $H$ by
$e_N(x\xi)=E_N(x)\xi$.
\begin{lemma}
\begin{enumerate}
\item For $x\in M$, $e_Nxe_N=E_N(x)e_N$.
\item If $x\in M$, then $x\in N$ iff $e_Nx=xe_N$.
\end{enumerate}
\end{lemma}

\begin{definition}
The basic construction of the pair $N\subset M$ of finite von Neumann algebras is the von Neumann algebra $\langle M,e_N\rangle$ generated by $M$ and $e_N$ acting on $L^2(M,\tr)$.
\end{definition}

\begin{proposition}
\begin{enumerate}
\item Operators of the form $\sum_{i=1}^{N}a_ie_N b_i$ with $a_i,b_i\in M$ give a $\sigma-$weakly dense $*$-subalgebra of $\langle M,e_N\rangle$.
\item The central support of $e_N$ in $\langle M,e_N\rangle$ is 1.
\item If $N\subset M$ are both finite dimensional with inclusion matrix $A$, the inclusion matrix for $M\subset \langle M,e_N\rangle$ is $A^T$, the transpose of $A$.
\item In the finite dimensional case, the weight vectors $u,v,w$ of the traces on $N\subset M\subset \langle M,e_N\rangle$ satisfy $A^Tw=v,Av=u$.
\end{enumerate}
\end{proposition}

Now we apply the basic construction with the conditional expectation $e_{A_{k+1}}:A_{k+1}\to A_{k}$ given by
\begin{center}
$e_{A_{k+1}}(\pi_{k+1}(x))=\pi_{k}(E(x))$, $x\in M_{k+1}(D)$
\end{center}
which is induced from the map  $E:M_{k+1}(D)\to M_k(D)$ (defined in 3.1).
As $\tr(E(x)y)=\tr(xy)$ for $x\in M_{k+1}(D),y\in M_k(D)$,
$e_{A_{k+1}}$ is the conditional expectation with respect to $\tr$.

For any fixed $k$, let $e=e_{A_k}:A_k\to A_{k-1}$ be the conditional expectation.
By the basic construction of the inclusion $A_{k-1}\subseteq A_k$, we will get a complex algebra $C_{k+1}=\langle A_k,e\rangle$.

Now we discuss the multi-matrices structure of $A_1,A_2$.

\begin{lemma}
$M_1(D)\cong \mathbb{C}\oplus\mathbb{C}$.
And $\tr$ is positive and faithful on $M_1(D)$ if and only if $D>1$.
In this case, we have
\begin{center}
$A_1=\mathbb{C}p_1\oplus\mathbb{C}g_1\cong \mathbb{C}\oplus\mathbb{C}$
\end{center}
and the weight vector of the trace is $(\frac{1}{D},\frac{D-1}{D})$.
\end{lemma}
\begin{proof}
Note that $M_n(D)=\mathbb{C}p_1\oplus\mathbb{C}g_1$ as vector spaces with $p_1=p_1^2=p_1^*,g_1=g_1^2=g_1^*$ and $p_1g_1=g_1p_1=0$.
And the traces of $p_1,g_1$ are $\tr(p_1)=\frac{1}{D},\tr(g_1)=\frac{D-1}{D}$, which completes the proof.
\end{proof}

\begin{lemma}
Assume $D\neq 1$ so that $A_1\cong \mathbb{C}\oplus\mathbb{C}$, then
\begin{enumerate}
\item If $D\neq 2$, $\pi_2$ is faithful, 
\begin{center}
$A_2$ is a $C^*$-algebra isomorphic to  $\mathrm{Mat}_{2}(\mathbb{C})\oplus\mathrm{Mat}_{2}(\mathbb{C})\oplus\mathbb{C}$
\end{center}
and the weight vector of the trace is $(\frac{1}{D^2},\frac{D-1}{D^2},\frac{D(D-2)}{D^2})$.
Hence if $D>2$, $\tr$ is positive definite on $M_2(D)$.
\item If $D=2$, $\pi_2$ is not faithful with kernel $R_2=\mathbb{C}g_2$, 
\begin{center}
$A_2$ is a $C^*$-algebra isomorphic to  $\mathrm{Mat}_{2}(\mathbb{C})\oplus\mathrm{Mat}_{2}(\mathbb{C})$
\end{center}
and the weight vector of the trace is $(\frac{1}{4},\frac{1}{4})$.
\end{enumerate}
\end{lemma}
\begin{proof}
Define the following elements in $M_2(D)$.
\begin{enumerate}
\item $e^{(1)}_{1,1}=p_1p_2,
    e^{(1)}_{1,2}=\frac{1}{\sqrt{D-1}}(p_1p_2-Dp_1e_1),\\
    e^{(1)}_{2,1}=\frac{1}{\sqrt{D-1}}(p_1p_2-De_1 p_1),
    e^{(1)}_{2,2}=\frac{1}{D-1}(p_1p_2+De_1-Dp_1e_1-De_1p_1)$,
\item $e^{(2)}_{1,1}=p_2g_1p_2,
    e^{(2)}_{1,2}=p_2g_1l_1,
    e^{(2)}_{2,1}=r_1g_1p_2,
    e^{(2)}_{2,2}=r_1g_1l_1$,
\item $e^{(3)}=g_2$.
\end{enumerate}
We leave readers to check that they form orthogonal systems of matrix units. This gives an homomorphism of $C^*$-algebras from $M_2(D)$ to $\mathrm{Mat}_{2}(\mathbb{C})\oplus\mathrm{Mat}_{2}(\mathbb{C})\oplus\mathbb{C}$.

Note that the trace has the weight vector $(\frac{1}{D^2},\frac{D-1}{D^2},\frac{D(D-2)}{D^2})$.
If all these terms are nonzero, $\tr$ is non-degenerate and $\pi$ is faithful. Then we can get the isomorphisms by counting the dimensions.
The faithfulness and positivity of $\tr$ on $M_2(D)$ also follows from this weight vector.
\end{proof}

\begin{corollary}
The inclusion matrix of $A_1\hookrightarrow A_2$ is given by
\begin{enumerate}
\item $\begin{pmatrix} 1 &1&0 \\ 1 & 1&1 \\ \end{pmatrix}$ if $D>2$,
\item $\begin{pmatrix} 1 &1 \\ 1 & 1 \\ \end{pmatrix}$ if $D=2$.
\end{enumerate}
\end{corollary}
\begin{proof}
The central projections in $A_1$ are $p_1,g_1$
and the central projections in $A_2$ are $q_1=e^{(1)}_{1,1}+e^{(1)}_{2,2},q_2=e^{(2)}_{1,1}+e^{(2)}_{2,2},q_3=e^{(3)}$ (under the representation $\pi$).
We have
\begin{enumerate}
\item $q_1=q_1p_1+q_1g_1$,
\item $q_2=q_2p_1+q_2g_1$,
\item $q_3=q_3g_1=g_1q_3$, (which is $0$ if $D=2$)
\end{enumerate}
as the decompositions into minimal projections.
Hence we get the inclusion matrices.
\end{proof}

\begin{proposition}
Suppose $\tr$ is positive definite on $M_k(D)$ and
$M_k(D)\cong A_k$ is a $C^*$-algebra isomorphic to  $
\bigoplus\limits_{i=0}^{k}\mathrm{Mat}_{m_{k,r}}(\mathbb{C})$ for all $1\leq k\leq n-1$.
Then
\begin{enumerate}
\item if $\tr(g_{n})\neq 0$, $M_{n}(D)\cong A_{n}$ is a $C^*$-algebra isomorphic to  $\cong\bigoplus\limits_{i=0}^{n}\mathrm{Mat}_{m_{n,r}}(\mathbb{C})$. Moreover, if $\tr(g_{n})> 0$, $\tr$ is positive definite on $M_{n}(D)$.
\item if $\tr(g_{n})=0$, $M_{n}(D)/\langle g_{n}\rangle\cong A_{n}$ is a $C^*$-algebra isomorphic to  $\bigoplus\limits_{i=0}^{n-1}\mathrm{Mat}_{m_{n,r}}(\mathbb{C})$.
\end{enumerate}
\end{proposition}
\begin{proof}
Let us consider the basic construction $C_{n}=\langle A_{n-1},e\rangle$ where $e=e_{A_{n-1}}:A_{n-1}\to A_{n-2}$.

In this finite dimensional case, we have
\begin{center}
$C_{n}=\langle A_n,e\rangle=\{\sum_{i=1}^{N}a_ie_N b_i|a_i,b_i\in A_{n-1}\}$.
\end{center}
We define a trace $\Tr$ on $C_{n}$ by $\Tr(aeb)=\tr(ae_{n-1}b)$.

By Proposition 3.5, any elements of $C_{k+1}$ can be expressed in the following form
\begin{center}
$\sum\limits_{i=1}^{N}a_ieb_i$ with $a_i,b_j\in A_k$.
\end{center}

Consider the algebra $B_{n}=\pi_{n}(A_{n-1}e_{n-1}A_{n-1})$ (may not be unital).
The multiplication here is defined by $e_{n-1}xe_{n-1}=e(x)e_{n-1}$.
Then the map $\phi_{n}:C_{n}\to B_{n}$ given by $\phi_{n}(aeb)=ae_{n-1}b$ gives us a surjective isometry for the definite hermitian products induced from $\Tr$ and $\tr$.
Hence we have $B_{n}\cong C_{n}$.

Moreover, let us consider the following elements in $A_{n}$:
\begin{enumerate}
\item $E^{(n-1)}_{i,j}=r_{i}r_{i+1}\cdots r_{n}g_{n-1}l_{n}l_{n-1}\cdots l_{j}$, $1\leq i\leq n$,
\item $E^{(n-1)}_{n,j}=p_{n}g_{n-1}l_{n}l_{n-1}\cdots l_{j}$, $1\leq j\leq n-1$,
\item $E^{(n-1)}_{i,n}=r_{i}r_{i+1}\cdots r_{n}g_{n-1}p_{n}$, $1\leq i\leq n-1$,
\item $E^{(n-1)}_{n,n}=p_{n}g_{n-1}$,
\item $E^{(n)}=g_{n}$.
\end{enumerate}
One can check $\tr(E^{(n-1)}_{i,j}E^{(n)})=0$ for all $1\leq i,j\leq n$ and $\{E^{(n-1)}_{i,j}\}_{1\leq i,j\leq n}$ form a system of matrix units.

As $M_{n-1}(D)\cong A_{n-1}$, we have $\pi_{n-1}$ is faithful on $M_{n-1}(D)$.
It also implies $\tr(g_{n-1})\neq 0$.
Let $e^{(n-1)}_{i,j}=\pi_{n}(E^{(n-1)}_{i,j})$ and $e^{(n)}=\pi_{n}(E^{(n+1)})$.
These $e^{(n-1)}_{i,j}$ also form a system of matrix units, which implies
\begin{center}
$Q_{n}=\langle e^{(n-1)}_{i,j}|1\leq i,j\leq n \rangle\cong \mathrm{Mat}_{n}(\mathbb{C})=\mathrm{Mat}_{m_{n,n-1}}(\mathbb{C})$.
\end{center}

By proposition 2.3, all these elements are orthogonal to the diagrams with rank no greater than $n-2$.
So $\tr(e^{(n-1)}_{i,j}z)=\tr(e^{(n)}z)=0$ for any $z\in B_{n}$.

Now, it remains to discuss the term $e^{(n)}=g_{n}$.
If $\tr(g_{n-1})=0$, this term is annihilated by $\pi_{n}$.
In these case, by showing $\dim(C_{n}\oplus e^{(n)})=\dim(B_{n})+(n)^2=\dim(M_{n}(D))-1$, we have
$A_{n}=M_{n}(D)/\langle g_{n}\rangle$
\begin{center}
$A_{n}\cong C_{n}\oplus \mathrm{Mat}_{n}(\mathbb{C}) \cong\bigoplus\limits_{i=0}^{n-1}\mathrm{Mat}_{m_{n,r}}(\mathbb{C})$.
\end{center}

For the case $\tr(g_{n})>0$, $\pi_{n}(g_{n})\neq 0$.
The statement follows similarly.
\end{proof}

\begin{proposition}
Suppose $\tr$ is positive definite on $M_k(D)$ and
$M_k(D)\cong A_k\cong
\bigoplus\limits_{i=0}^{k}\mathrm{Mat}_{m_{k,r}}(\mathbb{C})$ for all $1\leq k\leq n-1$. And we assume $\tr(g_{n})=0$, then for all $i\geq 0$ we have
\begin{center}
$A_{n+i}=\langle A_{n+i-1},e_{A_{n+i-1}} \rangle$
\end{center}
which is the basic construction of the pair $A_{n+i-2}\subset A_{n+i-1}$.
\end{proposition}

\begin{proof}
We have already proved the case $i=0$ in the previous proposition which is $A_{n}=C_{n}$ i.e. the basic construction for $A_{n-2}\subseteq A_{n-1}$.

Let us consider $i=1$ and the basic construction for $A_{n-1}\subseteq A_{n}$.
We have shown $B_{n+1}\cong C_{n+1}$ with the proof same as the one in the previous proposition.
So it suffices to show $\pi(l_{n}),\pi(l_{n})\in B_{n+1}$.

Note that as $\tr(g_{n})=0$, we have $\pi_{n}(g_{n})=0$ under the GNS representation.
Recall Proposition 2.3.5,
$\pi(\sum_{i=1}^{N}c_i\cdot w_i)=1$ with each $w_i\neq 1$ is a word of $e_i,l_i,r_i$ ($1\leq i\leq n-1$).
Multiply it by $\pi(l_{n})$, we have
\begin{center}
$\pi(l_{n})=\pi(\sum_{i=1}^{N}c_i\cdot w_i l_{n-1} )$.
\end{center}

Then, to show $\pi(l_{n})\in B_{n+1}$, it reduces to prove $\pi(w l_{n} )\in C_{n}$ for any non-identity word $w$ of $e_i,l_i,r_i$ ($1\leq i\leq n-1$).
This has been proved by Lemma 2.6.

For $i\geq 2$, the statement follows similarly.
\end{proof}

\subsection{The $C^*$-Algebras}
We give the explicit $C^*$-structure of the representations $A_n=\pi(M_n(D))$.
\begin{lemma}
Assume $D\in\{2\cos\frac{\pi}{n}+1|n\geq 3\}\cup [3,\infty)$.
The $*$-structure of $M_k(D)$ coincides with the $*$-structure of the Hilbert space representation $A_k=\pi(M_k(D))$ for all $k$, i.e. $\pi(x^*)=\pi(x)^*$.
\end{lemma}
\begin{proof}

Recall the basic construction which take a pair of two $C^*$-algebra $N\subset M$ and generate a third one $M_1$. 
The Bratteli diagram at $M\subset M_1$ is just the reflection of $N\subset M$ with respect to $M$. 
Then by Proposition 3.10, to obtain the next $C^*$ algebra, we can just add one or two extra simple sumands to the $C^*$ algebra generated from the basic construction.  



\end{proof}

\begin{proposition}
Assume $D\in [3,\infty)$, then $tr$ is positive definite on all $M_n(D)$.
Hence $A_n=\pi(M_n(D))$ is a $C^*$-algebra isomorphic to  $ \bigoplus\limits_{i=0}^{n}\mathrm{Mat}_{m_{n,r}}(\mathbb{C})$
for all $n$.
\end{proposition}
\begin{proof}
If $D\in [3,\infty)$, $\tr(g_n)>0$ for all $n$.
The argument is true for $n=1,2$ as shown in Lemma 3.6 and 3.7.

Now, by induction, Proposition 3.9.1 applies for all $n$.
\end{proof}

Suppose $\tr(g_k)>0$ for all $1\leq k\leq n-1$ but $\tr(g_{n})=0$.
This implies $D=2\cos\frac{\pi}{n+1}+1$ and we get the following description for the representations.

\begin{proposition}
$D=2\cos\frac{\pi}{n+1}+1$, we have\begin{enumerate}
\item For $1\leq k\leq n-1$, $A_k=\pi(M_k(D))$ is a $C^*$-algebra isomorphic to  $ \bigoplus\limits_{r=0}^{k}\mathrm{Mat}_{m_{k,r}}(\mathbb{C})$.
\item For $k\geq n$, $A_k=\pi(M_k(D))$ is a $C^*$-algebra isomorphic to  $ \bigoplus\limits_{r=0}^{n-1}\mathrm{Mat}_{l_{k,r}}(\mathbb{C})$.
    Here $l_k=(l_{k,0},\dots,l_{k,n-1})^\intercal=
    X^{k-1}(m_{n-1,0},\dots,m_{n-1,n-1})^\intercal$
where $X\in \mathrm{Mat}_{n}(\mathbb{C})=[a_{i,j}]_{1\leq i,j\leq n}$ is the inclusion matrix of $A_{n-1}\subseteq A_{n}$ given by $a_{i,j}=1$ if $|i-j|\leq 1$ and $0$ elsewhere.
\end{enumerate}
\end{proposition}

\begin{proof}
As $D=2\cos\frac{\pi}{n+1}+1$, we have $\tr(g_k)>0$ for all $1\leq k\leq n-1$ but $\tr(g_{n})=0$.
We can apply Proposition 3.9.1 until $k=n$ where we get $A_{n}$ by Proposition 3.9.2.

Now, for $k\geq n$, Proposition 3.10 says each $A_k$ is obtained from the basic construction of $A_{k-1}\subseteq A_k$.
The desired result follows Proposition 3.5-4.
\end{proof}

\begin{proof} [Proof of Theorem 3.1 and 3.2.]
Theorem 3.2 is just a corollary of Proposition 3.11 and 3.12.

For Theorem 3.1, if $D$ is $n$-generic, one can check by induction that all the weight vectors on $A_k$ ($1\leq k\leq n-1$) contains no zero terms.
Hence $\tr$ is always non-degenerate and the structure remains the same as the positive definite case.
\end{proof}

Theorem 3.1 can also be obtained by showing the nonsigularity of the Gram matrix of the Motzkin diagram or paths. One can refer \cite{BH} for this proof including a base change of the Motzkin paths.

By induction, we can also obtain the following results from the basic construction above.
\begin{corollary}
For $D=2\cos\frac{\pi}{n+1}+1$, the weight vectors $w_k$ of each $A_k$ are given by:
\begin{enumerate}
\item For $1\leq k\leq n-1$, $w_k=(\frac{1}{D^k}P_0(\tau),\dots,\frac{d^k}{D^k}P_k(\tau))$.
\item For $k\geq n$, $w_k=(\frac{1}{D^k}P_0(\tau),\dots,\frac{d^{n-1}}{D^k}P_{n-1}(\tau))$.
\end{enumerate}
\end{corollary}

\begin{corollary}
For $D=2\cos\frac{\pi}{n+1}+1$, the representatives of the equivalence classes of minimal projections in $A_k$ are given by the image of following elements under $\pi$:
\begin{enumerate}
\item $q_{k,0}=p_1\cdots p_k,~q_{k,1}=g_1p_2\cdots p_k,~q_{k,2}=g_2p_3\cdots p_k,~\dots,q_{k,k}=g_{k}$ if $1\leq k\leq n-1$,
\item $q_{k,0}=p_1\cdots p_k,~q_{k,1}=g_1p_2\cdots p_k,~q_{k,2}=g_2p_3\cdots p_k,~\dots,q_{k,n-1}=g_{n-1}p_{n}\cdots p_k$ if $k\geq n$.
\end{enumerate}
\end{corollary}

\subsection{Odd Motzkin Spaces}

We show there is a positive definite inner product on odd Motzkin tangles as the even case. 
Then we get the Hilbert space from these odd tangles. 

Let $\mathbf{M}(k)$ be the complex space spanned by Motzkin $k$-tangles. 
As the vector spaces of a unital $*$-planar algebra with $1$-dimensional zero-box space \cite{J99}, the Motzkin vector spaces $\mathbf{M}(k)$ have a sesquilinear form $\langle x,y\rangle$ given diaframmatically by: 
\begin{figure}[H]
  \centering
  \includegraphics[width=12cm]{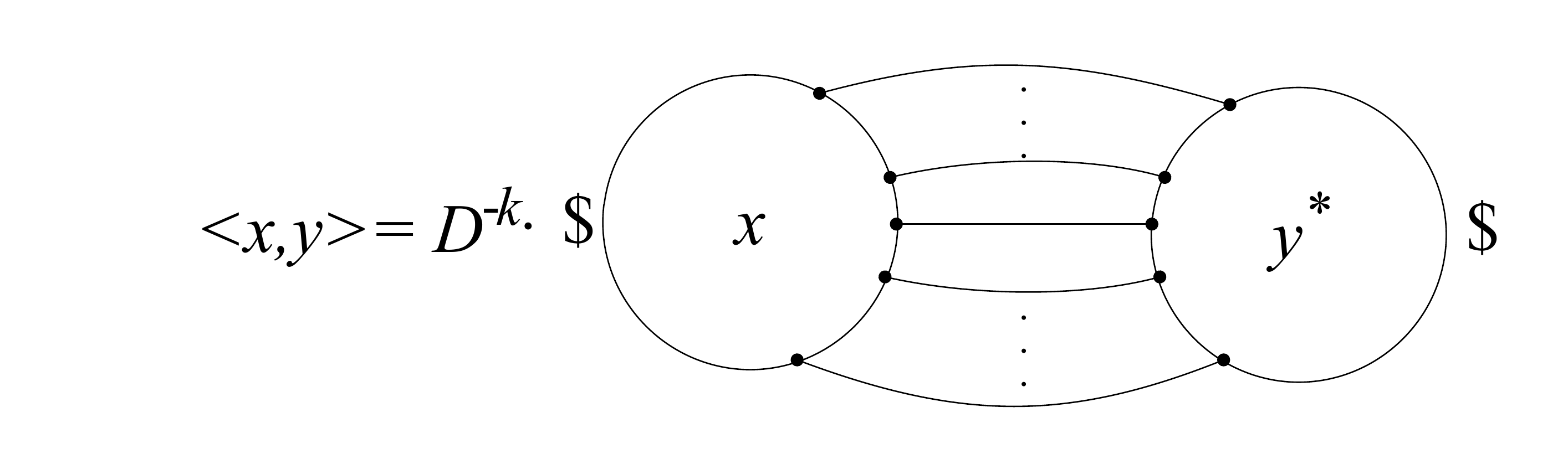}\\
\end{figure}
Here $M_0$ is canonically identified with $\mathbb{C}$, the empty tangle being $1$. 
The Motzkin planar algebra $M_n(D)$ is given by the quotient
\begin{center}
$M_n(D)=\mathbf{M}(n)/\{x\in \mathbf{M}(n)|\langle x,y\rangle=0~\forall y\}$.
\end{center}
We have shown that for even $n$, say $n=2k$,  $\langle x,y\rangle=\tr(y^*x)$ which gives a positive definite inner product on $\mathbf{M}(2k)$ whenever $D\geq 3$ or $D=2\cos\frac{\pi}{n}+1$, $n=3,4,5,\dots$. 

In this section we will see that this implies that $\mathbf{M}(2k+1)$ is also a Hilbert space under $\langle ~,~\rangle$, by identifying $\mathbf{M}(2k+1)$ with $p_1\mathbf{M}(2k+2)$ where $p_1$ is the projection defined in Section 2. 
We also give the dimesions of these spaces, at least in principle, whenever $\langle ~,~\rangle$ is positive definite on $\mathbf{M}_{2k+1}$ for all $k$.

Now, for any odd number $2k+1$, we take two Motzkin $2k+1$ tangles $x,y\in \mathbf{M}(2k+1)$. 
Recall that if $m_1+n_1=m_2+n_2$, $\mathbf{M}(m_1,n_1)$ can be identified with $\mathbf{M}(m_2,n_2)$ up to isotopy. 
So, without loss of generality, we assume $x,y\in \mathbf{M}(k,k+1)$. 
We define the inner product
\begin{center}
$\langle x,y\rangle=\tr(y^*x)$.
\end{center}
as above. 
Here $\tr$ is the trace on Motzkin algebras defined in Section 2.2 and $y^*x \in M_{k+1}(D)$. 
For example, let $k=3$ and consider the following two Motzkin $(3,4)$-tangles $x,y$. 
\begin{figure}[H]
  \centering
  \includegraphics[width=12cm]{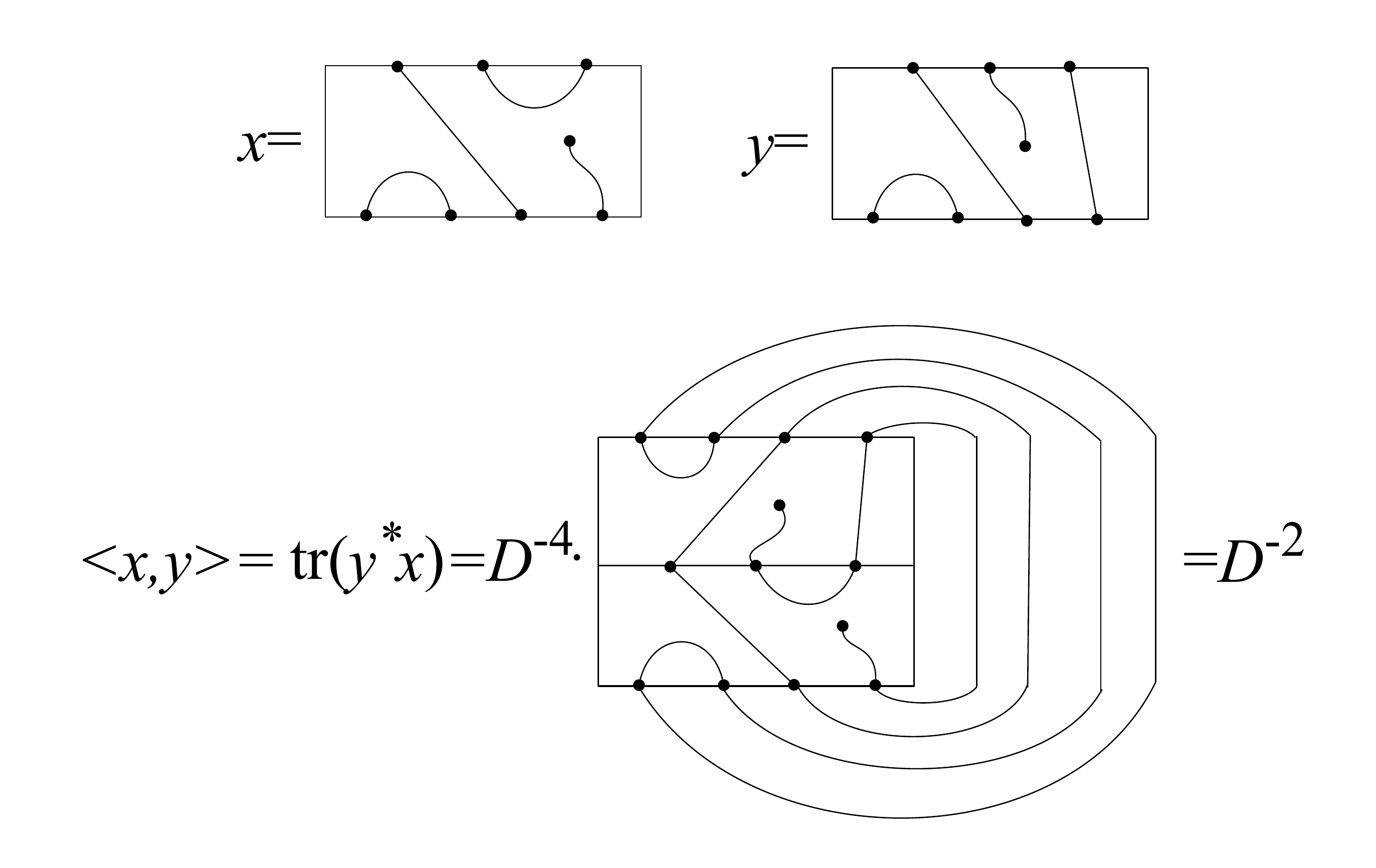}\\
  \caption{the inner product of two Motzkin $(3,4)$-tangles} \label{}
\end{figure}

\begin{lemma}
Assume $D\in\{2\cos\frac{\pi}{n}+1|n\geq 3\}\cup [3,\infty)$. 
The inner product defined above on $\mathbf{M}(m)$ with $m$ odd is positive definite. 
\end{lemma}
\begin{proof}
We have proved that the inner product on  $\mathbf{M}(2k+2)$ is positive definite in Section 3.2. 
It suffices to prove $\mathbf{M}(2k+1)$ is a subspace of $\mathbf{M}(2k+2)$ with the same inner product. 

Consider the following map $\phi:\mathbf{M}(2k+1)=\mathbf{M}(k,k+1) \to \mathbf{M}(2k+2)$ given by:
\begin{figure}[H]
  \centering
  \includegraphics[width=14cm]{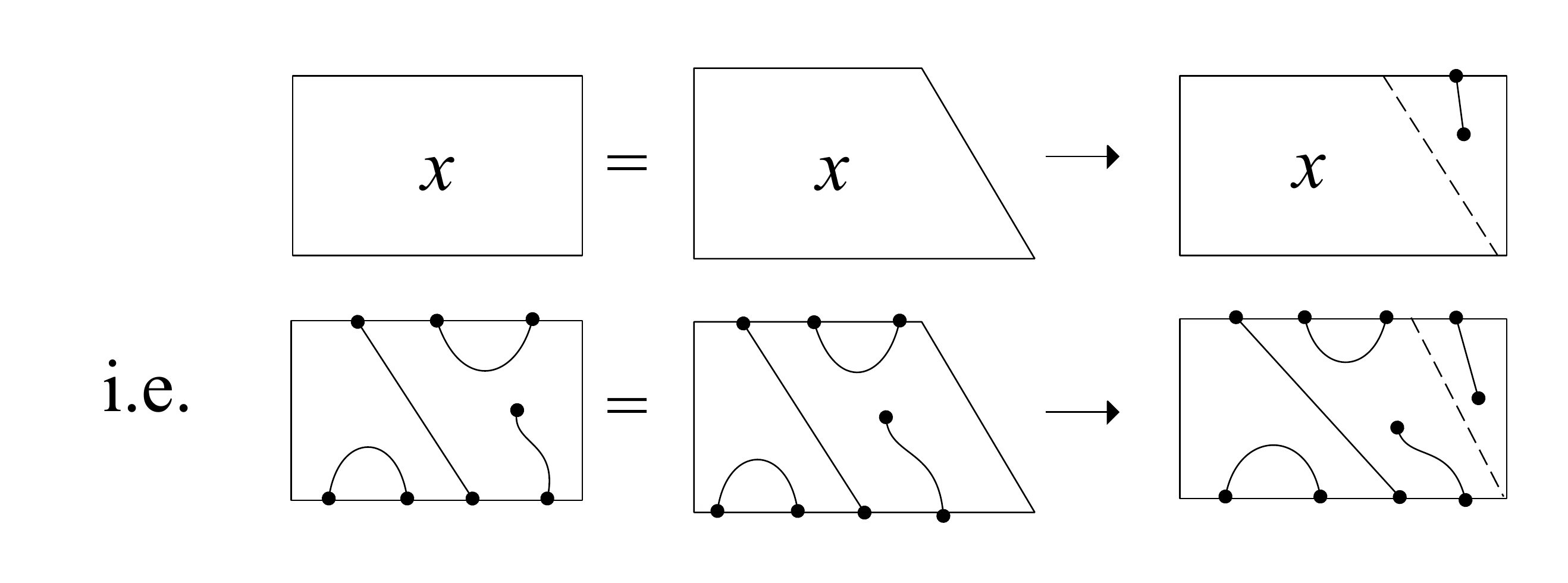}\\
  \caption{the map $\phi$} \label{}
\end{figure}
, which is an isometry. 
\end{proof}

Let $m$ be an positive integer, 
we let $R_m$ be the radical of this inner product
\begin{center}
$R_m=\{x \in \mathbf{M}(m)|\langle x,x\rangle=0\}$.    
\end{center} 
By taking the quotient space of the complex linear span $\mathbf{M}(m)$ of all Motzkin $m$-tangles by this radical, we obtain a Hilbert space. 
\begin{definition}
For $m\geq 1$, we call the finite dimensional Hilbert space 
\begin{center}
$H_m=\mathbf{M}(m)/R_m$   
\end{center}
the $m$-Motzkin space. 
\end{definition}

\begin{proposition}
As complex vector spaces, we have
\begin{enumerate}
    \item $H_{2k}\cong A_k(D)$,
    \item $H_{2k+1}\cong A_{k+1}(D)p_i$ for any $1\leq i\leq k+1$. 
\end{enumerate}
where $A_k=\pi(M_k(D))$'s are the $C^*$-algebras in Theorem 3.2. 
\end{proposition}
\begin{proof}
The even case is clear as shown in Section 3.2. 

For the odd case, it follows from the embedding constructed in Lemma 3.17 and the fact that $A_{k+1}(D)p_i\cong A_{k+1}(D)p_j$ for all $1\leq i,j\leq k+1$. 
\end{proof}

Let us consider the dimensions of these Hilbert spaces. 
The even ones are just the $C^*$-algebras in Section 3.3 whose dimensions are clear. 

For the odd case, we first consider two examples. 
\begin{example}
Assume $D>3$ and the Bratteli diagrams are given in Theorem 3.2. 
We just proved that $H_{2k+1}\cong A_{k+1}(D)p_i$. Now we let $i=1$. 
\begin{enumerate}
    \item Note $A_1=\mathbb{C}\oplus \mathbb{C}$, $p_1= 1\oplus 0$ in $A_1$ and its range is of dimension $1$. So $\dim H_1=1$. 
    \item $A_2=\mathrm{Mat}_2(\mathbb{C})\oplus\mathrm{Mat}_2(\mathbb{C})\oplus \mathbb{C}$. By the Bratteli diagram, $p_1= \begin{pmatrix} 1 &0 \\ 0& 0 \\ \end{pmatrix}\oplus\begin{pmatrix} 1 &0 \\ 0& 0 \\ \end{pmatrix}\oplus0$ in $A_2$ whose range is of dimension $1\cdot 2+1\cdot 2=4$. So $\dim H_3=4$. 
    \item $A_3=\mathrm{Mat}_4(\mathbb{C})\oplus\mathrm{Mat}_5(\mathbb{C})\oplus \mathrm{Mat}_3(\mathbb{C})\oplus \mathbb{C}$.
    \begin{center}
        $p_1=\begin{pmatrix} 1 &0&0&0 \\ 0&1&0& 0 \\0 &0&0&0\\0 &0&0&0\\ \end{pmatrix}\oplus\begin{pmatrix} 1 &0&0&0&0 \\ 0&1&0&0& 0 \\0 &0&0&0&0\\0 &0&0&0&0\\ 0 &0&0&0&0\\\end{pmatrix}\oplus\begin{pmatrix} 1 &0&0 \\ 0&0&0 \\0 &0&0\\ \end{pmatrix}\oplus 0$,
    \end{center}
    whose range is of dimension $2\cdot 4+2\cdot 5+1\cdot 3=21$. $\dim H_5=21$. 
\end{enumerate}
In next section, we can further show that $\dim H_k=\mathcal{M}_k$, the $k$-th Motzkin number. 
\end{example}
\begin{example}
Assume $D=2\cos\frac{\pi}{4}+1$ and the Bratteli diagrams are given in Example 3.3. 

The dimensions of $H_k$ are the same as Example 3.19 until $k=7$. 
Consider the algebra $A_4$. We have $A_4=\mathrm{Mat}_9(\mathbb{C})\oplus\mathrm{Mat}_{12}(\mathbb{C})\oplus \mathrm{Mat}_8(\mathbb{C})$. 
Note $p_1= I_4\oplus I_5\oplus I_3$ whose range is of dimension $4\cdot 9+5\cdot 12+3\cdot 8=120$. 

So $\dim H_7=120$, which is less than $127$, the $7$-th Motzkin number. 
\end{example}

\subsection{Explicit Dimension Formulas}

Now we will get the explicit formula for the Motzkin spaces in Section 3.4. 

\begin{theorem}
The dimensions of the Motzkin spaces and the generating functions are given by
\begin{enumerate}
    \item for $D=2\cos\frac{\pi}{n}+1$ with $n\geq 3$, we have
    \begin{center}
     $\dim H_k=\frac{2}{n}\sum_{j=1}^{n-1}(2\cos\frac{j\pi}{n}+1)^k\sin^2\frac{j\pi}{n}$.
    \end{center}
    and the generating function is $\sum_{k=1}^{\infty}\dim H_k\cdot x^k=\frac{P_{n-1}(x)}{P_{n}(x)}$.  
    \item for $D\geq 3$, we have $\dim H_k=\mathcal{M}_k$, the $k$-th Motzkin number so that the generating function is $\frac{1-x-\sqrt{1-2x-3x^2}}{2x^2}$.  
\end{enumerate}
\end{theorem}
First we prove it for the generic case $D\geq 3$. 
\begin{proposition}
For $D\geq 3$, we have $\dim H_k=\mathcal{M}_k$ (the $k$-th Motzkin number ) for all $k$. 
\end{proposition}
\begin{proof}
The case $k\leq 5$ are proved in Example 3.19. 
In Section 3.3, we already know that it is true for even $k$. 

The induction goes similarly as Example 3.19. 
Consider $H_{2k+1}$. 
Note that $p_1$ in $A_{2k+2}$ can be presented by 
\begin{center}
$p_1=\oplus_{r=0}^{k}I_{m_{k,r}}\oplus 0 \in A_{2k+2}= \oplus_{r=0}^{k+1}\mathrm{Mat}_{m_{k+1,r}}$. 
\end{center}
By Proposition 3.18, $\dim{H_{2k+1}}$ is just the dimension of the range of $p_1$ in $A_{2k+2}$ that can be given as
$\sum\limits_{r=0}^{k}m_{k,r}m_{k+1,r}=\mathcal{M}_{2k+1}$.  
\end{proof}

Next we turn to the non-generic case, say $D=2\cos\frac{\pi}{n}+1$. 
Let $Q_n\in \mathrm{Mat}_{n-1}(\mathbb{C})$ be given by 
\begin{equation}
Q_n=[q_{i,j}^{(n)}]_{(n-1)\times(n-1)}
=\begin{bmatrix}
1&1&0&0&\cdots\ &0\\
1&1&1&0&\cdots &0\\
0&1&1&1&\cdots &0\\
\vdots&\vdots &\vdots &\vdots&\ddots&\vdots\\
0&\cdots&0&1&1 &1\\
0&\cdots&0&0&1 &1
\end{bmatrix}
\end{equation}
in which $q_{i,j}^{(n)}=1$ if $|i-j|\leq 1$ and $0$ otherwise.
Let $\xi_n=(1,0,\dots,0)^T \in \mathbb{C}^{n-1}$. 
\begin{lemma}
For $D=2\cos\frac{\pi}{n}+1$, the dimensions of the irreducible summands in $A_i$ can be given by $Q_n^i\xi_n$.
\end{lemma}
\begin{proof}
It suffices to prove for $i\leq n-2$. The cases $i\geq n-1$ reduce to the calculation of $l_{k,r}$'s which is given in Proposition 3.14. 

For $i\leq n-2$, consider the following two vectors in $\mathbb{C}^{n-1}$:
\begin{center}
$\beta_i=(m_{i,0},\dots,m_{i,i},0,\dots,0)$ and $\beta_{i+1}=(m_{i+1,0},\dots,m_{i+1,i+1},0,\dots,0)$. 
\end{center}
By Lemma 2.2, one can directly get $\beta_{i+1}=Q_n\beta_i$. 
Then it follows by $\beta_0=(1,0,\dots,0)^T=\xi_n$. 
\end{proof}

\begin{proposition}
For $D=2\cos\frac{\pi}{n}+1$ with $n\geq 3$, the dimensions are given by
\begin{center}
$\dim H_k=\langle Q_{n}^{k}\xi_n,\xi_n \rangle$    
\end{center}

\end{proposition}
 
\begin{proof}
Observe that the projection $p_1$ in $A_i$ can be expressed as 
\begin{center}
$p_1=\oplus_{r=0}^{n-2}I_{l_{i-1,r}}\oplus 0 \in A_{i}= \oplus_{r=0}^{n-2}\mathrm{Mat}_{l_{i,r}}$.
\end{center}
where the numbers $l_{i,r}$'s are the terms in the vector $Q_n^i\xi_n$.

For even space $H_{2i}=A_i$, its dimension is given by
\begin{center}
$\langle Q_n^i\xi_n,Q_n^i\xi_n\rangle =\langle {(Q_n^i)}^{*}Q_n^i\xi_n,\xi_n\rangle=\langle Q_n^{2i}\xi_n,\xi_n\rangle$
\end{center}
since $Q_n^*=Q_n\in \mathrm{Mat}_{n-1}(\mathbb{C})$. 

For the odd space $H_{2i+1}$, its dimension is given by the dimension of the range of $p_1$ in $H_{2i+2}=A_{i+1}$, which is
\begin{center}
$\langle Q_n^i\xi_n,Q_n^{i+1}\xi_n\rangle =\langle Q_n^{2i+1}\xi_n,\xi_n\rangle$. 
\end{center}

Hence $\dim H_k=\langle Q_{n}^{k}\xi_n,\xi_n \rangle$. 
\end{proof}

\begin{corollary}
For $D=2\cos\frac{\pi}{n}+1$ with $n\geq 3$, 
    \begin{center}
    $\dim H_k=\frac{2}{n}\sum_{j=1}^{n-1}(2\cos\frac{j\pi}{n}+1)^k\sin^2\frac{j\pi}{n}$.
    \end{center}
\end{corollary}
\begin{proof}
Recall that the eigenvalues  of $Q_n$ are given by $\lambda_j=2\cos\frac{j\pi}{n}+1$  with the eigenvectors
\begin{center}
    $v_j=(\sin\frac{j\pi}{n},\sin\frac{2j\pi}{n},\dots,\sin\frac{(n-1)j\pi}{n})^T$. 
\end{center}
for $1\leq j\leq n-1$. 

One can decompose $\xi_n$ as $\xi_n=\sum_{j=1}^{n-1}\mu_j v_j$ with
$\mu_j=\frac{2}{n}\sin\frac{j\pi}{n}$. Then we have $\dim H_k=\langle Q_{n}^{k}\xi_n,\xi_n \rangle=\sum_{j=1}^{n-1}\lambda_j^k\mu_jv_{j,1}$.  
\end{proof}
\begin{corollary}
The generating function is $\sum_{k=1}^{\infty}\dim H_k\cdot x^k=\frac{P_{n-1}(x)}{P_{n}(x)}$ when $D=2\cos\frac{\pi}{n}+1$.   
\end{corollary}
\begin{proof}
The generating function can be given as $\langle (I_{n-1}-xQ_n)^{-1}\xi_n,\xi_n\rangle$, 
which is just the $(1,1)$-entry of the matrix $(I_{n-1}-xQ_n)^{-1}$.

Let $\mathrm{adj}(A)$ be the adjugate of any matrix $A$, we have
\begin{center}
$(I_{n-1}-xQ_n)^{-1}=\frac{1}{\det(I_{n-1}-xQ_n)}\mathrm{adj}(I_{n-1}-xQ_n)$.
\end{center}
One can show $\det(I_{n-1}-xQ_n)=P_{n}(x)$. 
Observe the $(1,1)$-entry of 
$\mathrm{adj}(I_{n-1}-xQ_n)$ is $\det(I_{n-2}-xQ_{n-1})=P_{n-1}(x)$. 
Hence $\langle (I_{n-1}-xQ_n)^{-1}\xi_n,\xi_n\rangle=\frac{P_{n-1}(x)}{P_{n}(x)}$. 
\end{proof}

\section{The $\text{II}_1$ Factors and the Relative Commutants}

This section will be mainly devoted to the following representation $\pi$ of $\bigcup\limits_{k=1}^{\infty}M_k(D)$ which gives a $\text{II}_1$ factor.

Let $A_{\infty}=\bigcup\limits_{k=1}^{\infty}A_k$, where $A_k=\pi(M_k(D))$ is the $C^*$-algebra constructed in Section 3.
We have proven that $tr$ is positive on $A_\infty$ if $D\in \{2\cos\frac{\pi}{n}+1|n\geq 3\}\cup [3,\infty)$. 

Let us consider the GNS representation with respect to $\tr$ which we also denoted by $\pi$. 
\begin{enumerate}
\item Let $H=\overline{A_{\infty}}^{\tr}$ is the completion of the pre-Hilbert space $A_{\infty}$ with inner product given by $\langle x,y\rangle=\tr(xy^*)$.
\item $\pi: A_\infty \to B(H)$ is given by $\pi(x)y=xy$ for $x,y\in A_{\infty}$. The boundness comes from the fact each $A_k$ is a finite dimensional $C^*$-algebra. 
\item $M=\pi(A_\infty)''\subset B(H)$ is the von Neumann algebra generated by $\pi(A_\infty)$. 
\end{enumerate}

\subsection{The AFD $\text{II}_1$ Factors}
\begin{theorem}
If $D\in \{2\cos\frac{\pi}{n}+1|n\geq 3\}\cup [3,\infty)$, the trace $\tr$ is a factor trace.

In particular, $M=\pi(A_\infty)''$ is a $\text{II}_1$ factor.
\end{theorem}

Firstly, we have a lemma from \cite{BH} (Proposition 4.17).
\begin{lemma}
Any trace on $M_n(D)$ is uniquely determined by its values on $p_k$ with $1\leq k\leq n$.
\end{lemma}

This lemma can also be shown diagrammatically by the facts that $\tr(e_i)=\tr(e_ip_ip_ie_i)=\tr(p_ie_ie_ip_i)=\delta\tr(p_i)$, etc.

Now, we define $P$ to be the algebra generated by $\{p_i\}_{i\in \mathbb{N}}$.
And we have $P''$ is an abelian von Neumann algbera.

\begin{lemma}
For any finite permutation $\sigma$ of $\mathbb{N}$, we have a self-adjoint unitary $u\in A_\infty$ such that $up_iu=p_{\sigma(i)}$ for all $i \in \mathbb{N}$
\end{lemma}
\begin{proof}
It suffices to show that the transposition of $p_i,p_{i+1}$ can be effected by a self-adjoint unitary.
And it also suffices to show we have a partial isometry $v$ such that $vv^*=p_i$ and $v^*v=p_{i+1}$.
In fact, $v=r_i$ is the partial isometry.
$u=1+l_i+r_i-p_i-p_{i+1}$ is the self-adjoint unitary such that $up_iu=p_{i+1}$ and leaves other $p_j$ invariant.
\end{proof}

\begin{lemma}
Any normal normalized trace on $M$ is equal to $tr$ on $P''$.
\end{lemma}
\begin{proof}
It is well-known that $P''\cong L^{\infty}(X,\mu)$ with $X=\prod\limits_{i=1}^{\infty}\{a_i,b_i\}$ and $\mu=\prod\limits_{i=1}^{\infty}\mu_{i}$ such that $\mu_i(a_i)=D^{-1}, \mu_i(b_i)=1-D^{-1}$.

By Lemma 4.3, we have an action of $S_{\infty}=\cup_{k=1}^{\infty}S_k$ on $P''$.
Such an action is well known to be ergodic.
Hence any invariant measure from other trace which is absolutely continuous with respect to $\tr$ is proportional to $\tr$.
Since $\tr(1)=1$, it is the unique trace.
\end{proof}

\begin{proposition}
$M$ is a $\text{II}_1$ factor.
\end{proposition}
\begin{proof}
By Lemma 4.2 and 4.4, we have only one normal normalized trace on $M$. Thus $M$ is a $\text{II}_1$ factor.
\end{proof}

\begin{remark}
For the case $D=2\cos\frac{\pi}{n+1}$, one can also apply the Perron-Frobenius theory \cite{TK3} to get the uniqueness of the trace.
We show how it works as below.
\end{remark}

\begin{proposition}
Assume $D=2\cos\frac{\pi}{n+1}$, $\cup_{k\geq 1}A_k$ admits a unique trace.
\end{proposition}
\begin{proof}
By theorem 3.2, $\{A_k\}_{k\geq 0}$ has a truncated Bratteli diagram that has $n$ simple summands for all $k\geq n-1$.
And the inclusion matrix of these pairs $A_{k}\subset A_{k+1}, k\geq n-1$ is given by
\begin{equation}
Q_{n+1}=[q_{i,j}]_{n\times n}
=\begin{bmatrix}
1&1&0&0&\cdots\ &0\\
1&1&1&0&\cdots &0\\
0&1&1&1&\cdots &0\\
\vdots&\vdots &\vdots &\vdots&\ddots&\vdots\\
0&\cdots&0&1&1 &1\\
0&\cdots&0&0&1 &1
\end{bmatrix}
\end{equation}
where $q_{i,j}=1$ if $|i-j|\leq 1$ and $0$ otherwise.
One can check that the eigenvalues $\lambda_1,\lambda_2,\cdots,\lambda_n$ of $A$ are
\begin{center}
$\{2\cos(\frac{\pi}{n+1})+1,2\cos(\frac{2\pi}{n+1})+1,\cdots,2\cos(\frac{n\pi}{n+1})+1\}$
\end{center}
within the decreasing order.
The corresponding eigenvectors are
\begin{center}
$v_1=\left( {\begin{array}{*{30}{c}}
\sin(\frac{\pi}{n+1})\\
\sin(\frac{2\pi}{n+1})\\
\vdots\\
\sin(\frac{n\pi}{n+1})
\end{array}} \right)$
$v_1=\left( {\begin{array}{*{30}{c}}
\sin(\frac{2\pi}{n+1})\\
\sin(\frac{4\pi}{n+1})\\
\vdots\\
\sin(\frac{2n\pi}{n+1})
\end{array}} \right)$
,$\cdots$,
$v_n=\left( {\begin{array}{*{30}{c}}
\sin(\frac{n\pi}{n+1})\\
\sin(\frac{2n\pi}{n+1})\\
\vdots\\
\sin(\frac{n^2\pi}{n+1})
\end{array}} \right)$
\end{center}
Only $v_1$ is an eigenvector with positive entries.

Suppose $v$ is a weight vector for an arbitrary trace.
Write it as $v=\sum_{i=1}^{n}x_i\cdot v_i$.
Then we have
\begin{center}
$Q^kv=\sum_{i=1}^{n}\lambda_i^kx_i\cdot v_i=\lambda_1^k(x_1v_1+\sum_{i=2}^{n}(\frac{\lambda_i}{\lambda_1})^kx_i\cdot v_i)$.
\end{center}
Consequently, as $|\lambda_1|> |\lambda_i|$ for $2\leq i\leq n$,
\begin{center}
$\lim\limits_{k\to \infty}\frac{Q^kv}{||Q^kv||}=\frac{v_1}{||v_1||}$.
\end{center}
Hence the space of tracial state on $\{A_k\}_{k\geq 0}$ is a singleton.
%
%
\end{proof}

For $k\in \mathbb{Z}$, let $M_{k}$ to be the $\text{II}_1$ factor generated by $\{e_i,l_i,r_i|i\geq 1-k\}$. Note $M_0=M$ and we have the a tower of factors
\begin{center}
$\cdots\subseteq M_{-1}\subseteq M_0=M \subseteq M_1\subseteq M_2 \subseteq \cdots $
\end{center}
with $D^2$ for the index of each pair.

\begin{proposition}
$[M:M_{-k}]=D^{2k}$ and $[M_k:M]=D^{2k}$ for $k\geq 1$.
\end{proposition}
\begin{proof}
It suffices to prove $M_{-1}$.
As $M_{-1}$ is generated by the sequences of elements with same relation as $\{e_i,l_i,r_i|i\geq 1\}$ which generates $M$, it is a $\text{II}_1$ factor.

Moreover, let $\langle M_{-1},e\rangle$ be the basic construction of the pair $M_{-2}\subset M_{-1}=M$. Then the identity 1 in $\langle M_{-1},e\rangle$ is equivalent to a finite sum of finite projections in $M$.
This gives us a finite factor.
Then by Proposition in 3.1.7 \cite{J83}, we have $[M:M_{-1}]=\tr(e)^{-1}=D^2$.
\end{proof}

Note $M_{-k}$ is the $\text{II}_1$ factor generated by $\{l_i,r_i,e_i|i\geq1- k\}$. So we have $[M_0:M_{-k}]=\tau^{k}=D^{2k}$.
Then $M_{-k}'\cap M$ is a finite dimensional von Neumann algebra.

The rest of this section is mainly devoted to the structure of these relative commutants.


\subsection{The Relative Commutants for $D=2\cos(\frac{\pi}{n})+1$}

In this part, we consider the case $D=2\cos(\frac{\pi}{n})+1$ and the index of subfactors is $(2\cos(\frac{\pi}{n})+1)^2$.
For $1\leq s<t$, let $A_{s,t}$ to be the finite dimensional semisimple complex algebra generated by $\{e_i,l_i,r_i|s\leq i\leq t-1\}$.
Hence $A_{1,t}=A_t$ and $A_{0}=\mathbb{C}$.

\begin{lemma}
For $t\geq s+n-1$, $A_{s,t+1}=\langle A_{s,t},e_{t}\rangle$, the basic construction of $A_{s,t-1}\subset A_{s,t}$.
\end{lemma}
\begin{proof}
By sending $e_i,l_i,r_i$ to $e_{i+1-s},l_{i+1-s},r_{i+1-s}$respectively, we have $A_{s,t}\cong A_{1,t+1-s}$ with $t+1-s\geq n$.

Then the basic construction for $A_{s,t-1}\subset A_{s,t}$ is isomorphic to the one we get from $A_{1,t-s}\subset A_{1,t+1-s}$.
By Proposition 3.11, it is isomorphic to $A_{1,t+1-s}$ and hence $A_{s,t+1}$.
\end{proof}

\begin{definition}
Given finite von Neumann algebras $A\subset B,C\subset D$ with a faithful normal tracial state on $D$. If the diagram below is a commutative one of maps,
\begin{center}
$\xymatrix{
C&\hookrightarrow&D\\
                     A \ar[u]&\hookrightarrow&B \ar[u]
}$
\end{center}
we shall say it is commuting square of finite von Neumann algebras.

Moreover, we call it a symmetric commuting square if $z_{A_2}(e_c)=1$ where $e_C$ is the conditional expectation from $C\hookrightarrow D$ , $A_2=\{B,e_C\}$ and $z_{A_2}(e_C)$ is the central support projection of $e_A$ in $A_2$.
\end{definition}

\begin{theorem}[Ocneanu's compactness\cite{JS97}]

Given the following commuting squares
\begin{center}
$\xymatrix{
B_0&\subset&B_1&\subset&B_1&\subset&\cdots\\ A_0\ar[u]&\subset&A_1\ar[u]&\subset&A_1\ar[u]&\subset&\cdots
}$
\end{center}
where $B_{n+1}=\langle B_n,e_n\rangle$ is the tower of basic constructions, $A_{n+1}$ is the algebra generated by $A_n,e_n$.

Let $R$ be the $\text{II}_1$ factor constructed from  $\cup_{k=1}^{\infty}B_n$, $R_0$ be the $\text{II}_1$ factor constructed from  $\cup_{k=1}^{\infty}A_n$, which is a subfactor of $R$. Then
\begin{center}
$R_0'\cap R=A_1'\cap B_0$.
\end{center}
\end{theorem}

\begin{proposition}
Suppose $D=\cos(\frac{\pi}{n+1})+1$ ($n\geq 2$), then we have the relative commutant
\begin{center}
$M_{-k}'\cap M=A_{k-1}=\pi(M_{k-1}(D))$.
\end{center}
\end{proposition}
\begin{proof}
Let us consider the following commuting squares
\begin{center}
$\xymatrix{
A_{1,k+n-1}&\subset&A_{1,k+n}&\subset&A_{1,k+n+1}&\subset&\cdots\\ A_{k,k+n-1}\ar[u]&\subset&A_{k,k+n}\ar[u]&\subset&A_{k,k+n+1}\ar[u]&\subset&\cdots
}$
\end{center}
It is a tower of symmetric commuting squares by Lemma 4.8 that both $A_{1,k+n+i+1}=\langle  A_{1,k+n+i},e_{k+n+i}\rangle$ and $A_{k,k+n+i+1}=\langle  A_{k,k+n+i},e_{k+n+i}\rangle$ are basic constructions.

By Theorem 4.9, we have $M_{-k}'\cap M=A_{k,k+n}'\cap A_{1,k+n-1}$.
The inclusion $A_{k-1}\subset A_{k,k+n}'\cap A_{1,k+n-1}$ is straightforward.

For the other direction,
it can be divided into the following two cases.

\begin{enumerate}
\item For $k\geq n$, let us consider the commutant $\{e_k,\dots,e_{k+n-1}\}'\cap A_{1,k+n-1}$ which must contain $A_{k,k+n}'\cap A_{1,k+n-1}$. We have all $e_k,\dots,e_{k+n-1}$ coming from the basic constructions.
By Proposition 3.1.4(ii) of \cite{J83}, $e_{k+n-1}'\cap A_{1,k+n-1}=A_{1,k+n}$.
Then by induction, we have $\{e_k,\dots,e_{k+n-1}\}'\cap A_{1,k+n-1}=A_{k-1}$.

\item For $k\leq n-1$, we can also get $\{e_n,\dots,e_{k+n-1}\}'\cap A_{1,k+n-1}=A_{n-1}$ by the argument in case 1.
    It remains to show what is $\{e_k,\dots,e_{n-1}\}'\cap A_{n-1}$.
    Consider the basic construction $A_{n-2}\subset A_{n-1}$, we get $\langle A_{n-1}, e_{n_1}\rangle$, which is no longer the $A_{n-1}$.
    Then, also by Proposition 3.1.4(ii) of \cite{J83}, we have $e_{n-1}'\cap A_{n-1}=A_{n-2}$.
    Hence, by induction, $\{e_k,\dots,e_{n-1}\}'\cap A_{n-1}=A_{k-1}$.
\end{enumerate}
\end{proof}

We proved that in the case $D=2\cos(\frac{\pi}{n})+1$ with $n\geq 3$ the relative commutants are just the image of finite dimentional Motzkin algebras $M_k(D)$ under the GNS representation.
But for the case $D\geq 3$, no triple $(A_{k-1},A_k,A_{k+1})$ are obtained from the basic constructions.
Hence the Ocneanu's Compactness Theorem does not apply any more.

\subsection{The Relative Commutants for $D\geq 3$}

Let $\mathfrak{gl}_2$ be the complex semisimple Lie algebra with its Cartan subalgebra $\mathfrak{h}$ \cite{Hum}.
The quantum group $\text{U}_q(\mathfrak{gl}_2)$ \cite{L} associated with $\mathfrak{gl}_2$ is the unital associative algebra over $\mathbb{C}$ with parameter $q\neq 1$ generated by $E,F,K_1,K_1^{-1},K_2,K_2^{-1}$ with the relations
\begin{align*}
 &[K_1,K_2]=0,~~~K_iK_i^{-1}=K_{i}^{-1}K_i=1,~~~i=1,2\\
 &K_1EK_1^{-1}=qE,~~~K_2EK_2^{-1}=q^{-1}E\\
 &K_1FK_1^{-1}=q^{-1}F,~~~K_2FK_2^{-1}=qF\\
 &[E,F]=\frac{K_1K_2^{-1}-K_2K_1^{-1}}{q-q^{-1}}
\end{align*}
For $q=1$, we define it be $\text{U}(\mathfrak{gl}_2)$, the universal enveloping algebra of $\mathfrak{gl}_2$.
Let $U_q(\mathfrak{h})$ be the quantum group generated by $K_1^{-1},K_2,K_2^{-1}$.

There are irreducible $\text{U}_q(\mathfrak{gl}_2)$ modules $V(r)$ for $r\geq 0$.
In particular, we consider the first two modules \cite{Jan}: $V(0)$ and $V(1)$:

\begin{enumerate}
\item $V(0)=\spn_{\mathbb{C}}\{v_0\}$ with the action given by $Ev_0=Fv_0=0,K_iv_0=v_0$, $i=1,2$.

\item $V(0)=\spn_{\mathbb{C}}\{v_{-1},v_1\}$ with the action given by
    \begin{center}
    $Ev_{-1}=v_1,Fv_{-1}=0,K_1v_{-1}=v_{-1},K_2v_{-1}=qv_{-1}$\\
    $Ev_{1}=0,Fv_{1}=v_{-1},K_1v_{1}=qv_1,K_2v_{1}=v_{1}$~~~~~~
    \end{center}
\end{enumerate}

Now, we set $V=V(0)\oplus V(1)=\spn_{\mathbb{C}}\{v_{-1},v_{0},v_{1}\}$.
Note that the tensor products of modules $V(r)$ satisfy the Clebsch-Gordan fusion:
\begin{center}
$V(k)\otimes V(0)=V(k)$ and $V(k)\otimes V(1)=V(k-1)\oplus V(k+1)$
\end{center}
where $V(-1)=0$.
According to this and by induction, we can get the $\text{U}_q(\mathfrak{gl}_2)$ irreducible modules decomposition of $V^{\otimes k}$:
\begin{lemma}
$V^{\otimes k}=\bigoplus\limits_{r=0}^{k}m_{k,r}V(r)$.
\end{lemma}

Now we define the following complex algebras
\begin{enumerate}
\item $X_1=\End_{\text{U}_q(\mathfrak{gl}_2)}(V)$, $X_k=\End_{\text{U}_q(\mathfrak{gl}_2)}(V^{\otimes k})$ and $Y=\cup_{k=1}^{\infty}X_k$.
\item $Y_1=\End_{\text{U}_q(\mathfrak{h})}(V)$, $Y_k=\End_{\text{U}_q(\mathfrak{h})}(V^{\otimes k})$ and $Y=\cup_{k=1}^{\infty}X_k$.
\item $Z_k=\End(V^{\otimes k})$ and $Z=\cup_{k=1}^{\infty}Z_k$.
\end{enumerate}

\begin{proposition}
Assume $D\geq 3$, then $M_k(D)$ is the semisimple complex algebra isomorphic to $X_k=\End_{\text{U}_q(\mathfrak{gl}_2)}(V^{\otimes k})$.
\end{proposition}
\begin{proof}
As Theorem 3.2, we know $\pi(M_n(D))\cong \bigoplus\limits_{r=0}^{n}\mathrm{Mat}_{m_{k,r}}(\mathbb{C})$.
Then it follows Lemma 4.11.
\end{proof}

Fixed the basis $\{v_{0},v_{1},v_{-1}\}$ (with the order) of $V$, we have the representation as $\End(V)\cong\mathrm{Mat}_{3}(\mathbb{C})$ where the generators are sent to the following matrices:

\begin{center}
{$E \mapsto \left[ {\begin{array}{*{30}{cc}}
0&0&0\\
0&0&1\\
0&0&0
\end{array}} \right]$ $F \mapsto \left[ {\begin{array}{*{30}{cc}}
0&0&0\\
0&0&0\\
0&1&0
\end{array}} \right]$}
$K_1 \mapsto \left[ {\begin{array}{*{30}{cc}}
1&0&0\\
0&q&0\\
0&0&1
\end{array}} \right]$
$K_2 \mapsto \left[ {\begin{array}{*{30}{cc}}
1&0&0\\
0&1&0\\
0&0&q
\end{array}} \right]$
\end{center}
Also, we defined a state on $\End(V)\cong\mathrm{Mat}_{3}(\mathbb{C})$ by
\begin{center}
$\phi(x)=\tr(ax)\cdot\frac{1}{1+q+q^{-1}}$
\end{center}
where $a$ is the diagonal matrix with $1,q,q^{-1}$.

By the GNS construction associated to the state $\phi$ and taking the completion and the closure of the weak topology,
we get three von Neumann algebras $X_q\subset Y_q\subset Z_q$ from $X\subset Y\subset Z$.

Then, by Proposition 6 of \cite{Sa95}, we have
\begin{lemma}
$Z_q$ is a the hyperfinite $III_q$ factor if $q<1$, the hyperfinite $III_{1/q}$ factor if $q>1$ and the hyperfinite factor if $q=1$.
\end{lemma}
Moreover, it is straightforward to see that $X_q\cong M$, the $\text{II}_1$ factor generated by $\{e_i,l_i,r_i|i\geq 1\}$ in section 4.1.

Consider the map $\rho: Z=\cup_{k=1}^{\infty}\End(V^{\otimes k})\to Z_q$ defined by
\begin{center}
$\rho:x\mapsto \text{1}_V \otimes x$
\end{center}
which is an injective normal, $\phi$-preserving $*$-homomorphism.
So it extends to a normal isometric $C^*$ endomorphism of $Z_q$.
And $\rho(X_q)$ is a $\text{II}_1$ subfactor of $X_q$.
\begin{lemma}\cite{Sa95}
$\rho^k(X_q)'\cap X_q=Y_k$ for $k\geq 1$.
\end{lemma}

Consider the trinomial expansion
\begin{center}
$(a+b+c)^k=\sum\limits_{i,j,l\geq 0,i+j+l=k}n(k;i,j,l)a^ib^jc^l$.
\end{center}
where $n(k;i,j,l)=\binom{k}{i,j,l}=\frac{k!}{i!j!l!}$ is the trinomial coefficient.
Then we have
\begin{proposition}
For $D\geq 3$, $M_{-k}'\cap M\cong \bigoplus\limits_{i,j,l\geq 0,i+j+l=k-1}\mathrm{Mat}_{n(k-1;i,j,l)}(\mathbb{C})$,
where the number of summands is $k(k+1)/2$.
\end{proposition}
\begin{proof}
With Lemma 4.11 above, the relative commutant is $Y_k=\End_{\text{U}_q(\mathfrak{h})}(V^{\otimes k})$, which is the fixed point algebra of the action $K_i^{\pm 1}$ on $\End(V^{\otimes k})$ with the action $T$ given by
\begin{center}
$T(K_i)(x_1\otimes x_2\otimes \cdots \otimes x_k)=K_ix_1K_i^{-1}\otimes K_ix_2K_i^{-1}\otimes \cdots \otimes K_ix_kK_i^{-1}$
\end{center}
with $i=1,2$.
So the relative commutant is the fixed point subalgebra of $\bigotimes_{1\leq i\leq k}\mathrm{Mat}_{3}(\mathbb{C})$ by the action of $T(K_1),T(K_2)$ above.
Then it follows by the induction on $k$.
\end{proof}

\begin{remark}
The tower of (higher) relative commutants above can be described by Pascal's pyramid.
One can refer to \cite{GHJ89} for the relative commutants for Temperley-Lieb algebras for the generic case i.e $d\geq 2$.
The tower of their relative commutants there are the Pascal's triangles.
\end{remark}

\begin{corollary}
For $D\geq 3$, the traces of minimal projections in $M_{-k}'\cap M$ are given by
\begin{center}
$\frac{D^i}{(1+D+D^{-1})^{k-1}}$, with $-(k-1) \leq i\leq k-1$.
\end{center}
\end{corollary}
We denote a minimal projection of such a trace by $p_{k,i}$.

\section{Bimodules and the Fusion Rule $A_n$}

We review the concept of bimodules over $\text{II}_1$ factors \cite{B97} and construct a family of bimodules over the $\text{II}_1$ factor $M$ that we obtained in Section 4. 
We get a tensor category with the fusion rule $A_n$. 

\subsection{Bimodules over $\text{II}_1$ factors}

Let $A$ and $B$ be $\text{II}_1$ factors.
An $A-B$ bimodule $\prescript{}{A}{H}_B$ is a pair of commuting normal (unital) representations $\pi_L,\pi_R$ of $A$ and $B^{\text{op}}$ respectively on the Hilbert space $H$.
Here $B^{\text{op}}$ is the opposite algebra of $B$, i.e $b_1\cdot b_2=b_2b_1$, which is also a $\text{II}_1$ factor.
Note that $\prescript{}{A}{H}_B$ is a left $A$-module and right $B$-module with the action denoted as $\pi_L(a)\pi_R(b)\xi=a\cdot \xi \cdot b$ with $a\in A,b\in B,\xi\in H$.

We say $\prescript{}{A}{H}_B$ is bifinite if the left dimension $\dim_A^L{H}<\infty$ and right dimension $\dim_B^R{H}<\infty$.

\begin{definition}
Let $H,K$ be two $A-B$ bimodules.
We say $H,K$ are equivalent if we have a unitary $u:H\to K$ such that $u(a\cdot \xi \cdot b)=a\cdot u(\xi) \cdot b$ for all $a\in A,b\in B,\xi\in H$ and denoted by $\prescript{}{A}{H}_B\cong\prescript{}{A}{K}_B$.

Moreover, we denote by
\begin{center}
$\Hom_{A,B}(H,K)=\{T\in B(H,K)|T(a\cdot \xi \cdot b)=a\cdot T(\xi) \cdot b\text{~for all~}a\in A,b\in B,\xi\in H\}$
\end{center}
the space of $A-B$ intertwiners from $H$ to $K$.
Let $\Hom_{A,B}(H)=\Hom_{A,B}(H,H)$
And we call an $A-B$ bimodule $H$ irreducible if $\Hom_{A,B}(H)=\mathbb{C}$.
\end{definition}

Note that $\Hom_{A,B}(H)\subset B(H)$ is a von Neumann algebra. 

Recall that for a $A$-module $H$, $v\in H$ is called $A$-bounded if we have a positive constant $c_v$ such that
\begin{center}
$\|xv\|\leq c_v\|x\|_2$ for all $x\in A$,
\end{center}
where $\|x\|_2=\tr(x^*x)^{1/2}$. We write $H^{\text{bdd}}$ for the set of all $A$-bounded vectors in $H$, which is a dense subspace of $H$ and also invariant under the action of $A$ and $A'$ \cite{EK98}\cite{J08}.

\begin{lemma}\cite{J08}
If $\prescript{}{A}{H}_B$ is bifinite, then a vector is $A$-bounded if and only if it is $B$-bounded.
\end{lemma}

\begin{lemma}
If $\prescript{}{A}{H}_B$ is bifinite, then $\Hom_{A,B}(H)$ is a finite dimensional von Neumann algebra.
\end{lemma}
\begin{proof}
Note that $\Hom_{A,B}(H)=A'\cap(B^{\text{op}})'\cap B(H)$ is of course a von Neumann algebra.
If $\prescript{}{A}{H}_B$ is bifinite, we have
$A\subset (B^{\text{op}})'\cap B(H)$ by the commuting action.
This imlies an inclusion of $\text{II}_1$ factors where
\begin{center}
$[(B^{\text{op}})'\cap B(H):A]=\frac{\dim_{B^{\text{op}}}(H)}{\dim_A(H)}=\frac{1}{\dim_A(H)\dim_B(H)}<\infty$.
\end{center}
Hence $\Hom_{A,B}(H)=A'\cap(B^{\text{op}})'\cap B(H)$ is a relative commutant of a pair of factors with finite index.
So by \cite{J83}, it is finite dimensional.
\end{proof}

\begin{corollary}
If $\prescript{}{A}{H}_B$ is bifinite and $p$ is a projection in $\Hom_{A,B}(H)$, then $Hp$ is an irreducible $A-B$ bimodule if and only if $p$ is minimal.
\end{corollary}
\begin{proof}
If $p$ is minimal, $\Hom_{A,B}(Hp)=p\Hom_{A,B}(H)=\mathbb{C}p\cong \mathbb{C}$.
Otherwise, assume $p=p_1+p_2$ is a decomposition into two subprojections, then $Hp=Hp_1\oplus Hp_2$, which is a direct sum of $A-B$ bimodules.
\end{proof}

\subsection{Construction of the Bimodules}

Let $m,k$ be two non-negative integers.  
Take any element $p$ from the Motzkin $(k,k)$-tangles (may also the Motzkin algebra $M_k(D)$). 
Consider the following element which is a Motzkin $(2m+k)$-tangle. 

\begin{figure}[H]
  \centering
  \includegraphics[width=7cm]{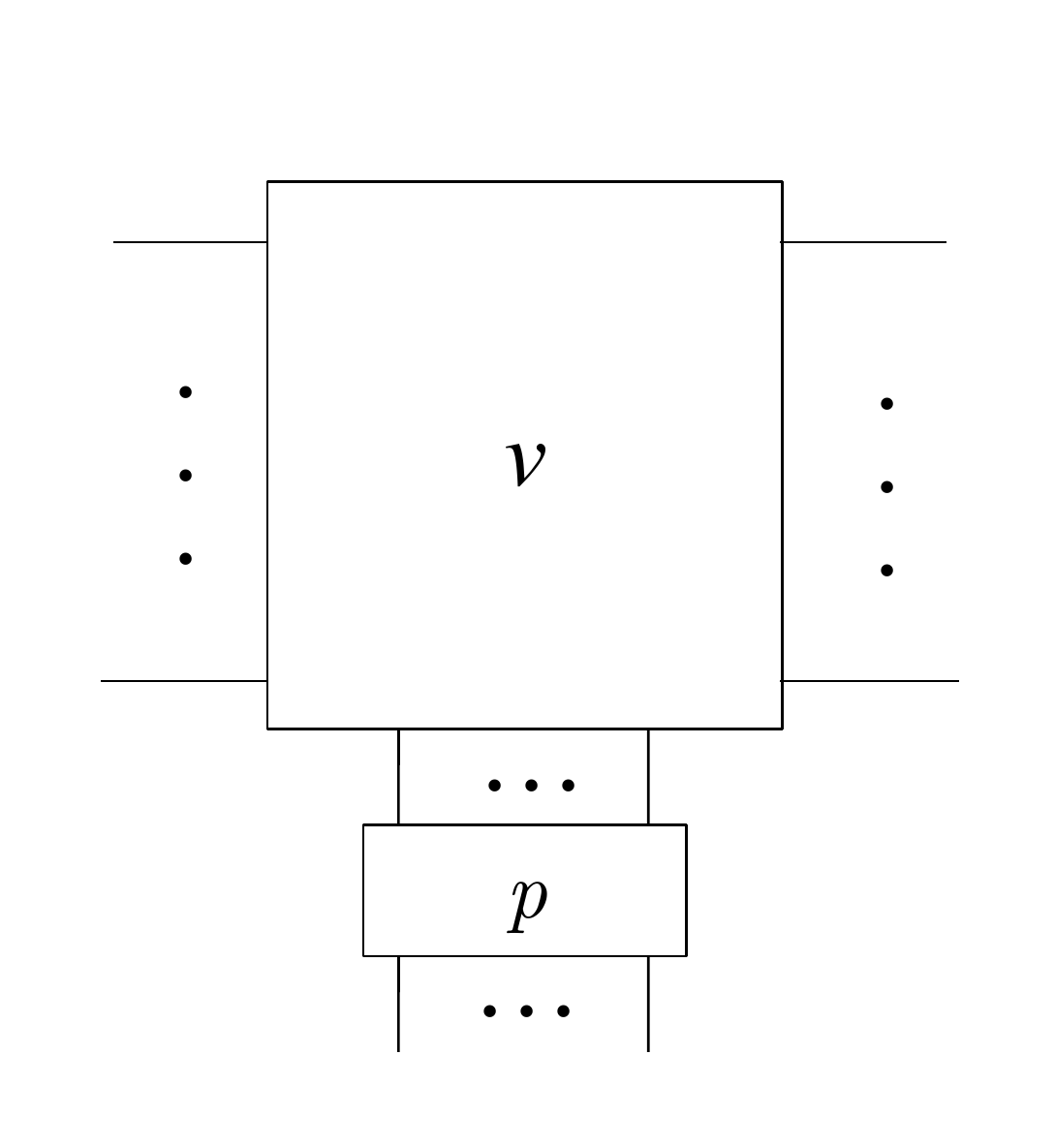}\\
\end{figure}
\begin{enumerate}
\item There are $m$ vertices on the left side and $m$ on the right side.
\item The $k$ vertices on the bottom are connected to the top $k$ vertices of the Motzkin $(2m+k)$-tangle $p$.
\item $v$ in the rectangle is an arbitrary Motzkin $(2m+k)$-tangle. 
\end{enumerate}

Let $V_m(p)$ the complex span of the elements in such a form as above. 
Let $p$ run through all the Motzkin $k$-tangles and denote the union of all these spaces by $V_{m,k}=\cup_{p\in \mathbf{M}(k,k)}V_m(p)$. 
Also, we define an inner product on $V_{m,k}$ as the following graph. 
\begin{figure}[H]
  \centering
  \includegraphics[width=12cm]{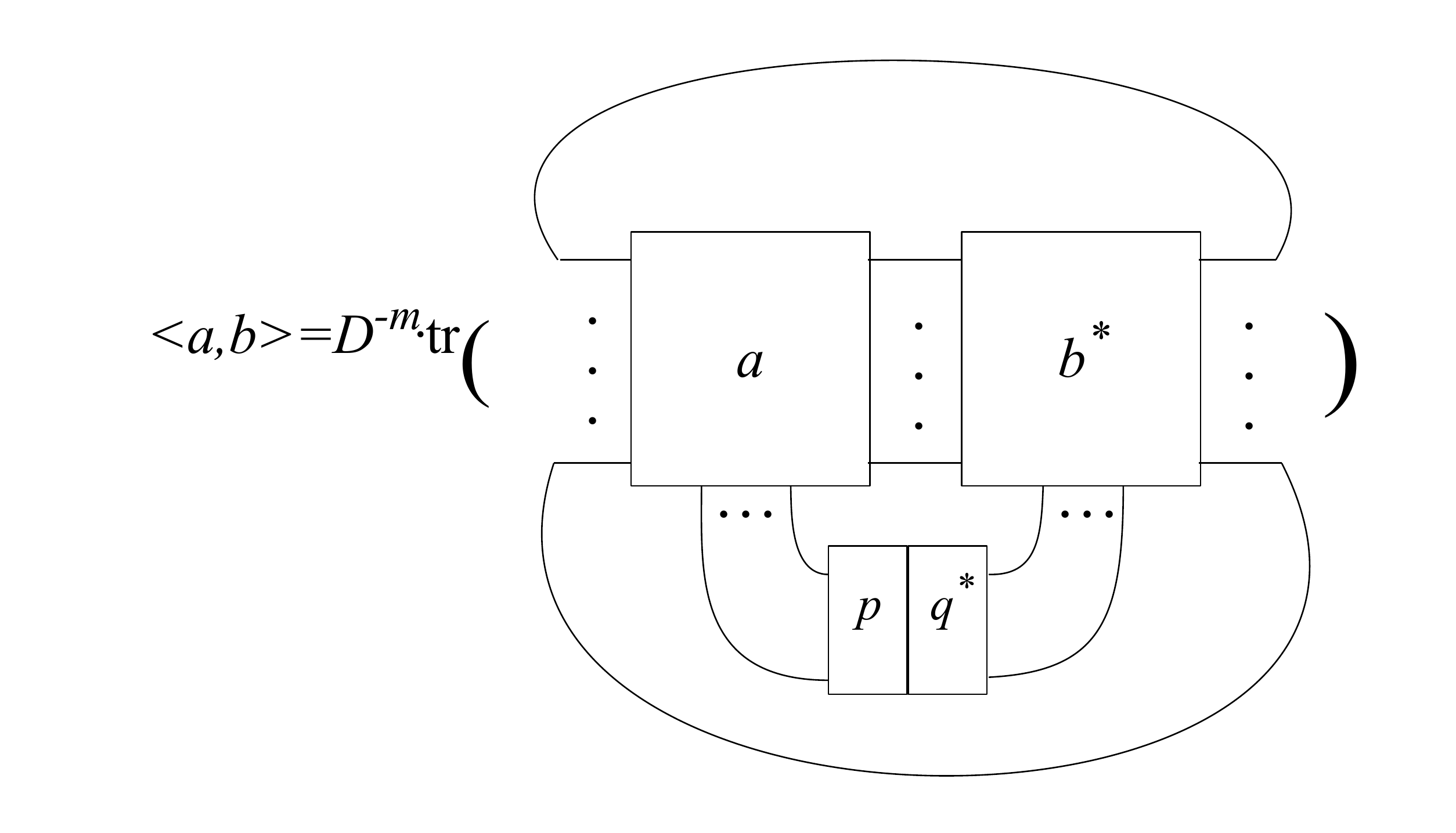}\\
  \caption{The inner product of $a \in V_m(p),b\in V_m(q)$} \label{}
\end{figure}

\begin{lemma}
The inner product defined above is positive.
\end{lemma}
\begin{proof}
Note that $V_{m,k}$ is just the complex span of Motzkin $2m+k$-tangles. 
Then it follows Lemma 3.17 for odd $2m+k$ and Section 3.3.  
\end{proof}

Now we assume $p$ be a self-adjoint idempotent in $M_k(D)$. 
Let $H_m=V_m(p)/R_m(p)$ be the quotient by the radical of the inner product restricted on $V_m(p)$. 
\begin{corollary}
$H_m(p)$ is a finite dimensional Hilbert space.
\end{corollary}
Observe $V_m(p)$ is a left $M_m(D)$ module and also a right $M_m(D)$ module that the two actions commute. 
As the quotient space, $H_m(p)$ inherits these actions. 
For $v\in H_m(p)$ and $x,y\in M_m(D)$, we denote $x\cdot v \cdot y$ by the action of $x,y$ on $v$ respectively.

\begin{figure}[H]
  \centering
  \includegraphics[width=7cm]{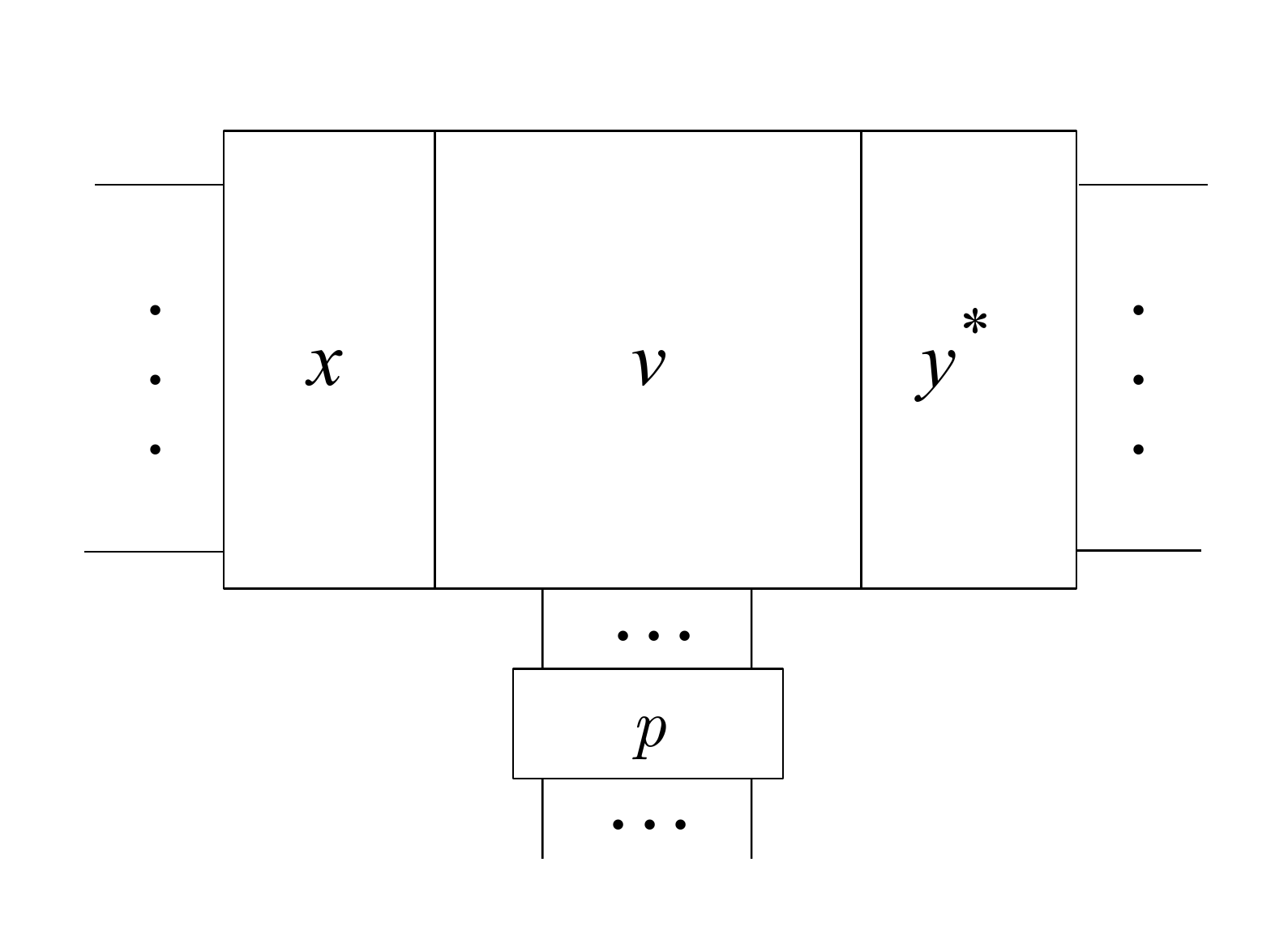}\\
  \caption{the vector $x\cdot v \cdot y$} \label{}
\end{figure}

\begin{remark}
One can get the dimensions of these $H_m(p)$ as what we did in Section 3.4 and by the Bratteli diagrams. 
\end{remark}

\begin{lemma}
$H_m(p)$ is a subrepresentation of $M_m(D)$ on $M_{m+\lceil k/2 \rceil}(D)$. 
\end{lemma} 

There is a natual embedding of $V_m(p)$ into $M_{m+\lceil k/2 \rceil}(D)$. 
The embedding can be constructed in the following way:
\begin{figure}[H]
  \centering
  \includegraphics[width=12cm]{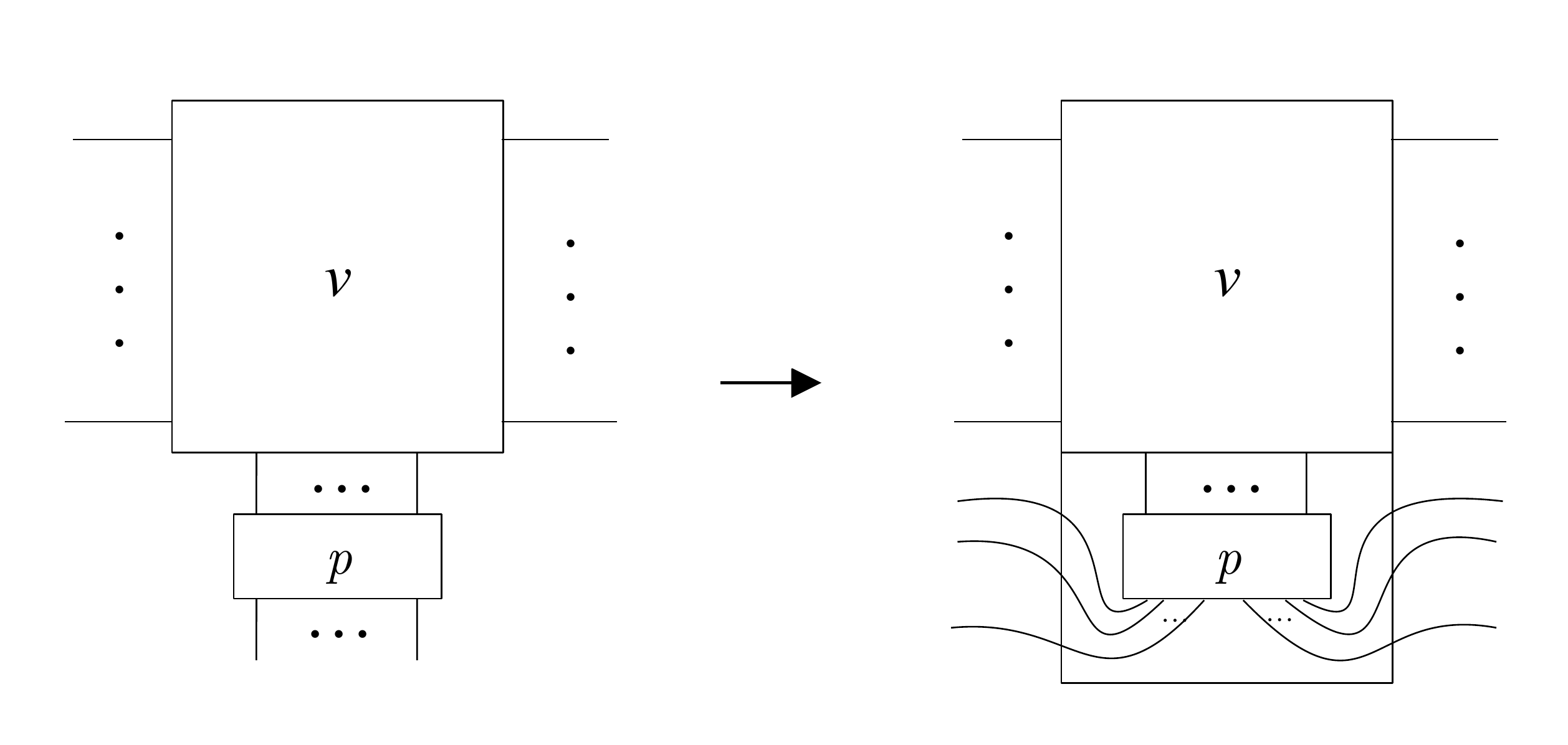}\\
  \caption{the embedding:
  $V_m(p)\to M_{m+\lceil k/2 \rceil}(D)$} \label{}
\end{figure}

\begin{enumerate}
\item If $k$ is even, say $k=2l$, we can move the left $l$ top vertices of $p$ to the left edge of $v$ and  move the right $l$ top vertices of $p$ to the right edge of $v$. 
\item If $k$ is odd, say $k=2l-1$, we first insert a $p_1$ on the right of $p$ which makes $p$ into $p\cdot p_{2l}\in M_{2l}(D)$ which is also a self-adjoint idempotent. Then the process is the same as when $k$ is even. 
\end{enumerate}

\begin{lemma}
$H_m(1_k)$ is faithful $A_m-A_m$-bimodule.
\end{lemma}


Consider the action of $A_{m+k}$ on $H_{m+k}(p)$. 
As $A_m\subset A_{m+k}$, $A_m$ has a restricted action on $H_{m+k}(p)$. 

\begin{lemma}
The restricted action of $A_m$ on $H_{m+k}(p)$ is faithful if $\tr(pp^*)\neq 0$. 
\end{lemma}  

It suffices to prove that if $\langle xv,xv\rangle=0$ for all $v\in H_m(p)$ then we have $x=0$. 

Indeed, let us consider the left action of $x$ on the following vector $v_0$. 
\begin{figure}[H]
  \centering
  \includegraphics[width=9cm]{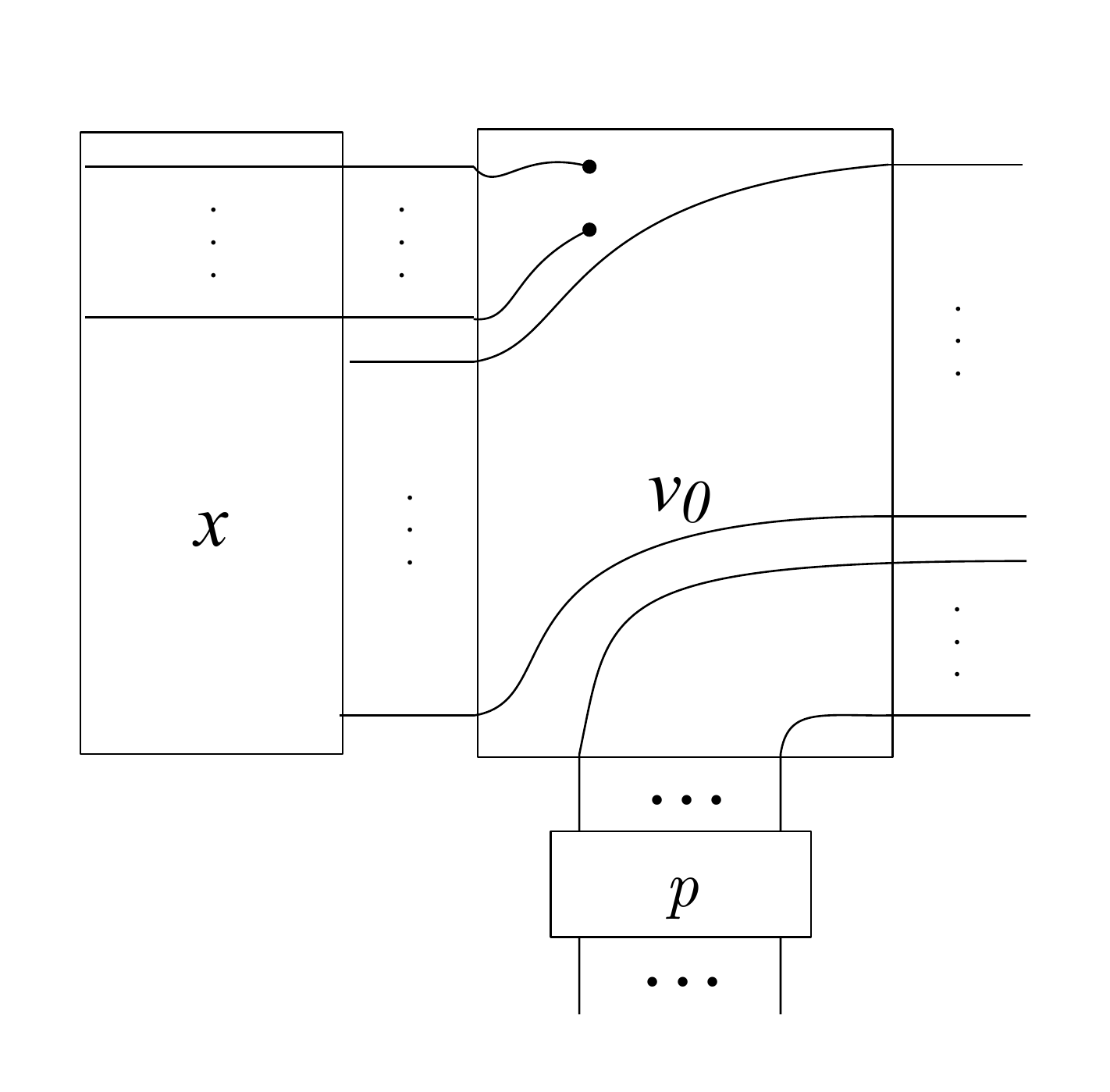}\\
  \caption{$x$'s action on the vector $v_0$} \label{}
\end{figure}
We can check that
\begin{center}
$\langle xv_0,xv_0\rangle=\frac{1}{D^k}\tr(xx*)\tr(pp^*)\neq 0$
\end{center}
Hence the restricted action is faithful.
 
One can embed $H_m(p)$ to $H_{m+1}(p)$ by adding a horizontal string on the top. 
Note this embedding $H_m(p)\to H_{m+1}(p)$ is compatible with the inner products.  
We take the direct limit
\begin{center}
${H_{\infty}(p)}=\varinjlim {H_m(p)}$,
\end{center}
which also admits an inner product. 

By taking the completion with respect to the inner product,  we get a Hilbert space $H_{p}=\overline{H_{\infty}(p)}^{||\cdot ||}$. 
Note we have left and right actions $\pi^L_{m},\pi^R_{m}$ of $M_m(D)$ on  $H_m(p)$ for all $m\in \mathbb{N}$. 
This induces the left and right actions
$\pi^L_{\infty},\pi^R_{\infty}$ of $M_{\infty}(D)$ on $H_{\infty}(p)$
by the direct limit: 
\begin{center}
   $\{M_\infty(D),\pi^L_{\infty},\pi^R_{\infty};{H_{\infty}(p)}\}=\varinjlim \{M_m(D),\pi^L_{m},\pi^R_{m};{H_{m}(p)}\}$ 
\end{center}

\begin{lemma}
$\pi^L_{\infty}(M_\infty(D)),\pi^R_{\infty}(M_\infty(D))$
extend to bounded operators on $V_{p}$. 
\end{lemma} 
\begin{proof}
This is guaranteed by the fact that $V_{p}$ is a closed subspace of $L^2(M_{\lceil k/2 \rceil},\tr)$ (as Claim 2) on which $M_{\infty}(D)$'s action is bounded.
\end{proof}

Take the strong operator topology of $\pi^L_{\infty}(M_\infty(D)),\pi^R_{\infty}(M_\infty(D))$, we get a von Neumann algebra $M$. 
One can prove $M$ is a $\text{II}_1$ factor as we showed in Section 4.1. 
Moreover, $M$ is isomorphic to the factor $M$ we constructed in Section 4.1.

\begin{proposition}
$V(p)$ is an $M-M$ bimodule where $M$ is the $\text{II}_1$ factor generated by $\{e_i,l_i,r_i|i\geq 1\}$.
\end{proposition}

\begin{figure}[H]
  \centering
  \includegraphics[width=10cm]{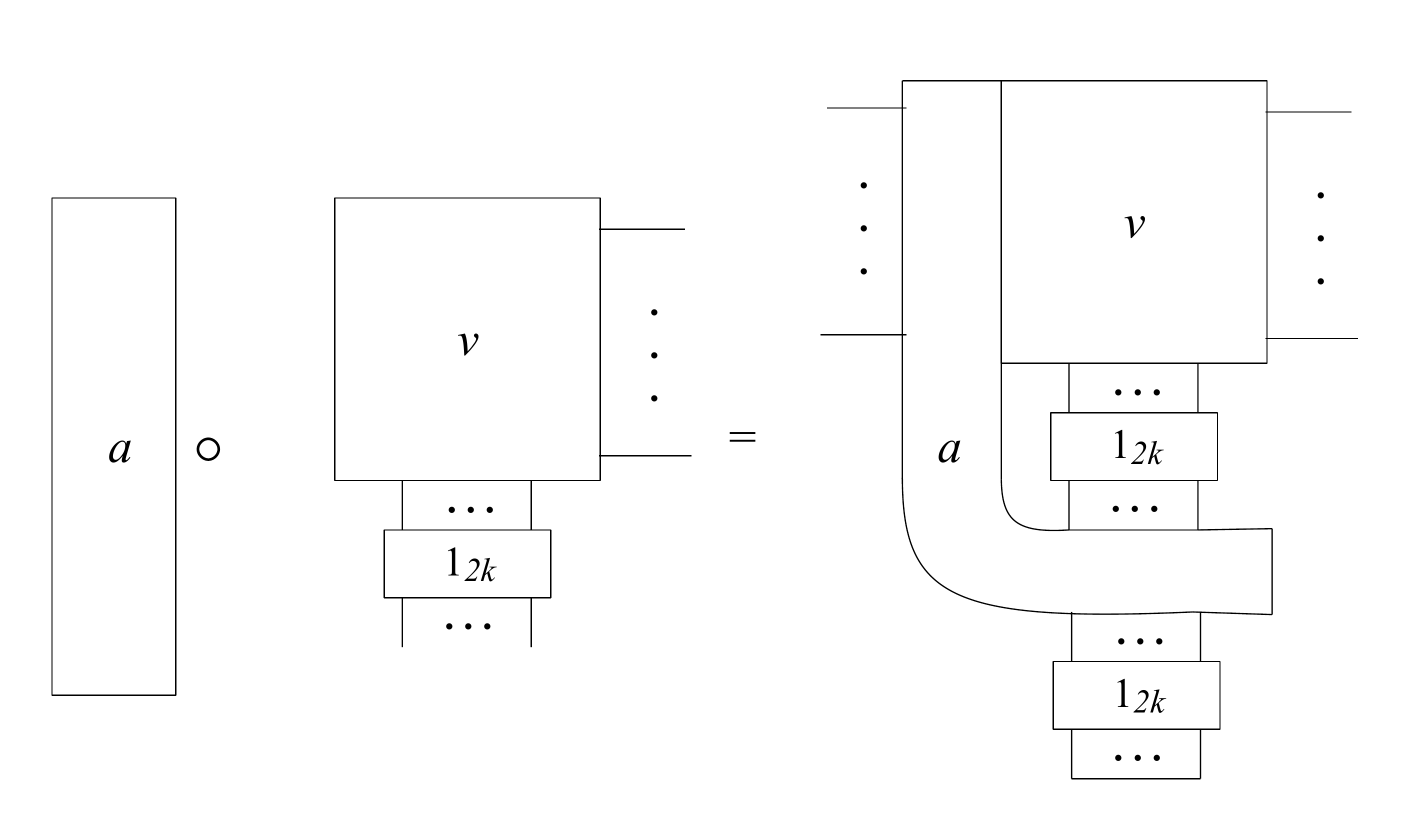}\\
  \caption{An action $\rho_L$ of $M_{2k}$ on $H(1_{2k})$ } \label{}
\end{figure}

\begin{example}
Let $1_{2k}\in \mathbf{M}_{2k}(D)$ be the identity element of $M_{2k}(D)$.
Then $H(1_{2k})$ is nothing but the space $L^2(M_k,\tr)$.
And we have $\dim_M^L H(1_{2k})=\dim_M^R H(1_{2k})=\dim_M L^2(M_k,\tr)=[M_k:M]=D^{2k}$.
\end{example}

And there a representation $\rho_L$ of the $\text{II}_1$ factor $M_{2k}$ on $H(1_{2k})$ which is given by extending the action
\begin{center}
$\rho_L:M_{m+2k}(D)\to B(H_m(1_{2k}))$
\end{center}
to $H(1_{2k})$ and also extending to the weak closure $M_{2k}$ of $\cup_{m=1}^{\infty}M_{m+2k}(D)$.

Observe that $\rho_L|_{M}=\pi_L$, we will use $\pi_L$ instead of $\rho_L$.
Similarly, we have $\pi_R:M_{m+2k}(D)\to B(H_m(1_{2k}))$.

By the construction above, we have that the action of $\pi_L(M_{2k})$ and $\pi_R(M)$ commutes.
So $\pi_L(M)\subset\pi_L(M_{2k})\subset \pi_R(M)'\cap B(H(1_{2k}))$ are inclusions of $\text{II}_1$ factors.
Moreover,
\begin{center}
$\dim_{\pi_R(M)'}(H(1_{2k}))=(\dim_{\pi_R(M)}(H(1_{2k})))^{-1}=[M_k:M]^{-1}=D^{-2k}$.
\end{center}
And as $[\pi_L(M_{2k}):\pi_L(M)]=D^{4k}$, we have
\begin{center}
$\dim_{\pi_L(M_{2k}}(H(1_{2k}))=\dim_{\pi_L(M)}(H(1_{2k}))[\pi_L(M_{2k}):\pi_L(M)]^{-1}=D^{-2k}$.
\end{center}
Hence $[\pi_R(M)'\cap B(H(1_{2k})):\pi_L(M_{2k})]=1$ and $\pi_R(M)'=\pi_L(M_{2k})$.

As a result, we get
\begin{proposition}
$\Hom_{M,M}(H(1_{2k}))=\pi_L(M)'\cap \pi_L(M_{2k})$.
\end{proposition}
This is the (higher) relative commutant that are known in section 4.

\begin{remark}
Assume $\tr_{M'}$ is the unique trace on the $\text{II}_1$ factor $M'\cap B(H(1_k))$ whenever $k$ is odd or even.
The trace $\tr$ defined above on the planar graphs must be this trace.
\end{remark}

\begin{example}
Let $1_{2k-1}\in \mathbf{M}_{2k-1}(D)$ be the identity element of $M_{2k-1}(D)$.
we have a natural $M-M$-bilinear surjective isometry
\begin{center}
$\phi:H(1_{2k-1})\to H(p_{2k})$
\end{center}
by adding a pair of disconnected vertices on the right of bottom side.
Observe $p_2k\in \pi_L(M)'\cap \pi_L(M_{2k})$.
We have
\begin{center}
$\dim_M^L(H(1_{2k-1}))=\dim_M^L(H(p_{2k})=\tr_{M'}(p_2k)\dim_M^L(H(1_{2k})=D^{-1}\cdot D^{2k}=D^{2k-1}$,
\end{center}
and also $\dim_M^R(H(1_{2k-1}))=D^{2k-1}$.
\end{example}

As the even case above, we obtain
\begin{corollary}
$\Hom_{M,M}(H(1_{2k-1}))=\pi_L(M)'\cap \pi_L(M_{2k-1})$.
\end{corollary}

Now w take the self-adjoint idempotent $p\in M_k(D)$ and consider the $M-M$ bimodule $H(p)$.
Note $p\in M_k(D)$ and $\pi_L(p)\subset \pi_L(M)'\cap \pi_L(M_{k})=\Hom_{M,M}(H(1_{k}))$.

\begin{proposition}
$\dim_M^L(H(p))=\dim_M^R(H(p))=\tr_{M'}(p)\dim_M(H(1_k))=D^k\tr(p)$.
In particular, $\dim_M^L(H(g_k))=\dim_M^R(H(g_k))=d^k\cdot P_k(\tau)$.
\end{proposition}

\subsection{The Tensor Map}

Let $A,B,C$ be $\text{II}_1$ factors.
Given an $A-B$ bimodule $\prescript{}{A}{H}_B$ and a $B-C$ bimodule $\prescript{}{B}{K}_C$, we define the $A-C$ bimodule of their tensor as \cite{EK98}, which is given by the completion of the algebraic tensor product $\prescript{}{A}{H}^{\text{bdd}}_B\otimes \prescript{}{B}{K}^{\text{bdd}}_C$ of bounded subspace with respect to the inner product defined by
\begin{center}
$\langle v_1\otimes u_1,v_2\otimes v_2\rangle=\langle v_1\langle u_1,u_2 \rangle_B,v_2 \rangle$
\end{center}
Here $\langle u_1,u_2 \rangle_B \in B$ is uniquely determined by
\begin{center}
$\tr(x\langle u_1,u_2 \rangle_B)=\langle xu_1,u_2 \rangle_B$ for all $x\in B$.
\end{center}

It is easy to check the following properties \cite{EK98}:
\begin{enumerate}
\item $\langle \lambda u_1+\mu u_2,u_3 \rangle_B=\lambda \langle u_1,u_3 \rangle_B+\mu \langle u_2,u_3 \rangle_B$,
\item $\langle u_1,u_2 \rangle_B=\langle u_2,u_1 \rangle_B^*$,
\item $\langle xu_1,u_2 \rangle_B=x\langle u_1,u_2 \rangle_B$ and $\langle u_1,xu_2 \rangle_B=\langle u_1,u_2 \rangle_B x^*$.
\end{enumerate}

Now we consider the tensor of $M-M$ bimodules $H(p),H(q)$ where $p\in M_k(D),q\in M_l(D)$ are self-adjoint idempotent.
Let $p|q$ be the self-adjoint idempotent in $M_{k+l}(D)$.
The goal of the following part is mainly contributed to prove the following theorem.
\begin{theorem}
Define a linear map
\begin{center}
$T_0:\cup_{m=1}^{\infty}H_m(p)\otimes \cup_{m=1}^{\infty}H_m(q)\to \cup_{m=1}^{\infty}H_m(p|q)$
\end{center}
which is given by
\begin{figure}[H]
  \centering
  \includegraphics[width=10cm]{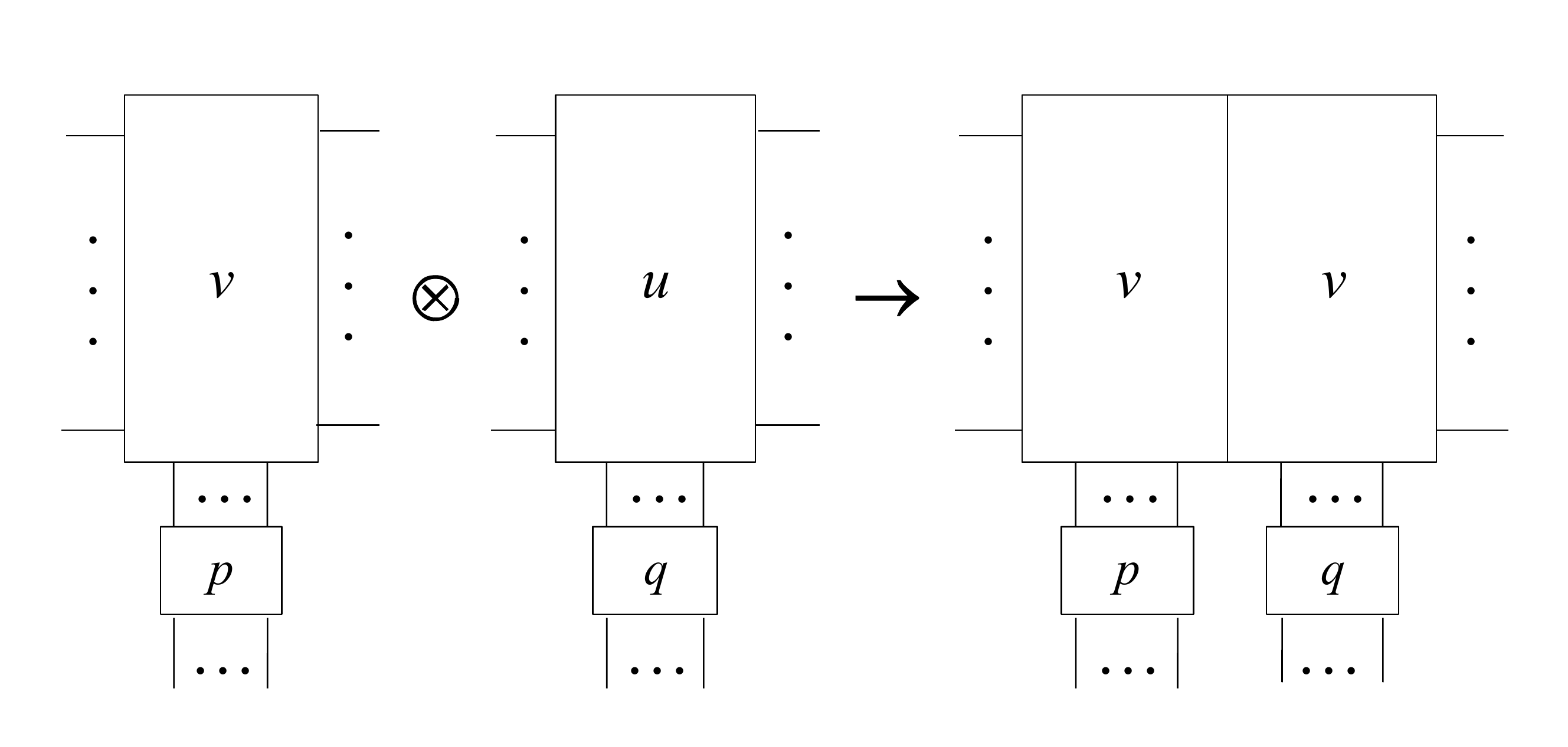}\\
  \caption{} \label{}
\end{figure}
Then $T_0$ has a extension $T:H(p)\otimes H(q)\to H(p|q)$.

We have $T$ is a surjective isometric $M-M$ bilinear isomorphism.
\end{theorem}

Such a map defined above is certainly $M-M$ bilinear.
It remains to prove the isometry and surjectivity.

\begin{proposition}
The map $T:H(p)\otimes H(q)\to H(p|q)$ defined above is an isometry.
\end{proposition}
\begin{proof}
Take $m\in\mathbb{N}$ large enough and $v_1,v_2 \in H_m(P),u_1,u_2\in H_m(q)$.
By definition, the inner product is determined by $\langle v_1\otimes u_1,v_2\otimes v_2\rangle=\langle v_1\langle u_1,u_2 \rangle_M,v_2 \rangle$.

Consider the element $\langle u_1,u_2 \rangle_M$ which is given by $\tr(x\langle u_1,u_2 \rangle_M)=\langle xu_1,u_2 \rangle_M$ for all $x\in M$.
Its right hand side is given by the trace of
\begin{figure}[H]
  \centering
  \includegraphics[width=7cm]{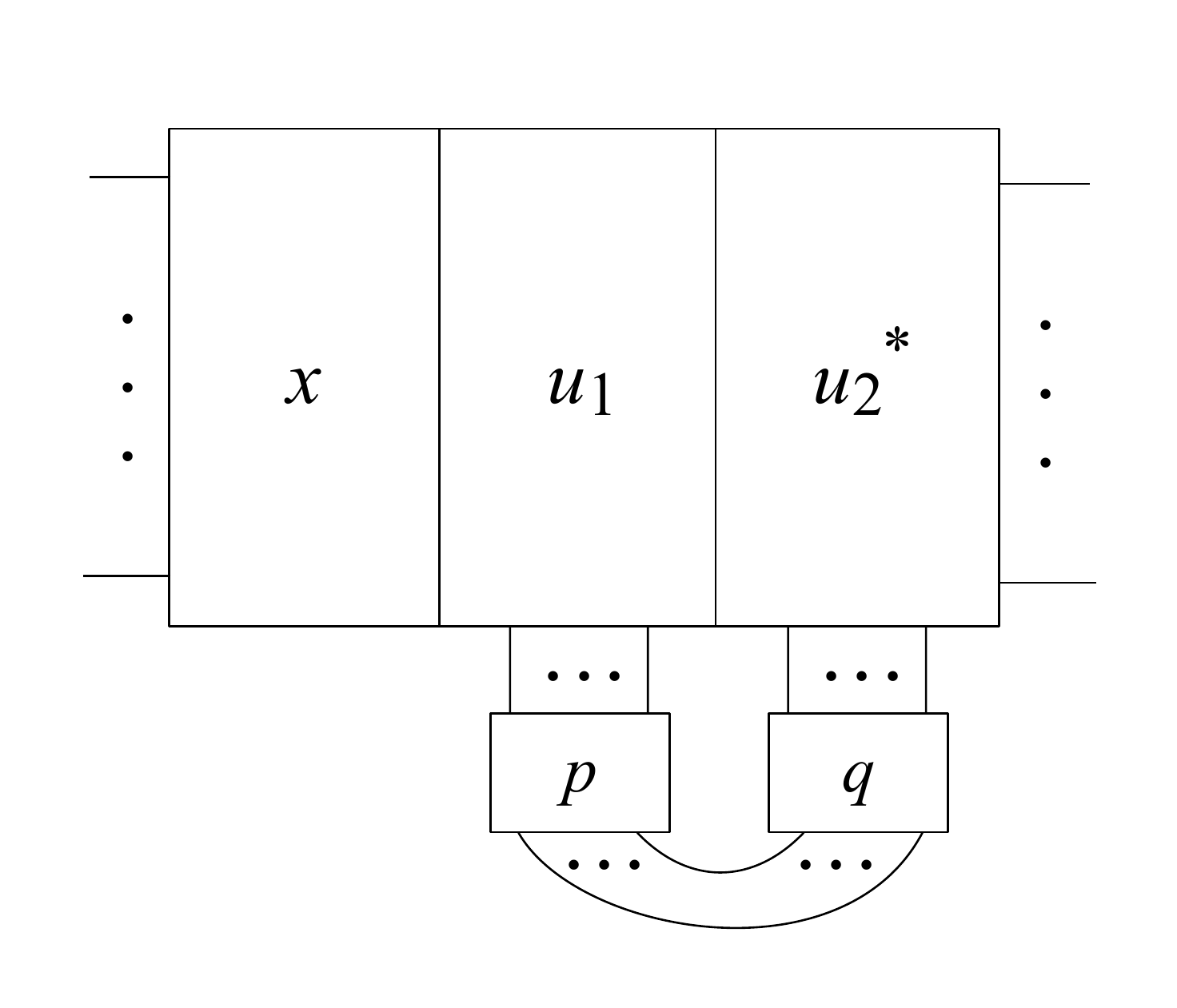}\\
  \caption{} \label{}
\end{figure}
Hence $\langle u_1,u_2 \rangle_M$ is the one presented as the graph below:
\begin{figure}[H]
  \centering
  \includegraphics[width=5cm]{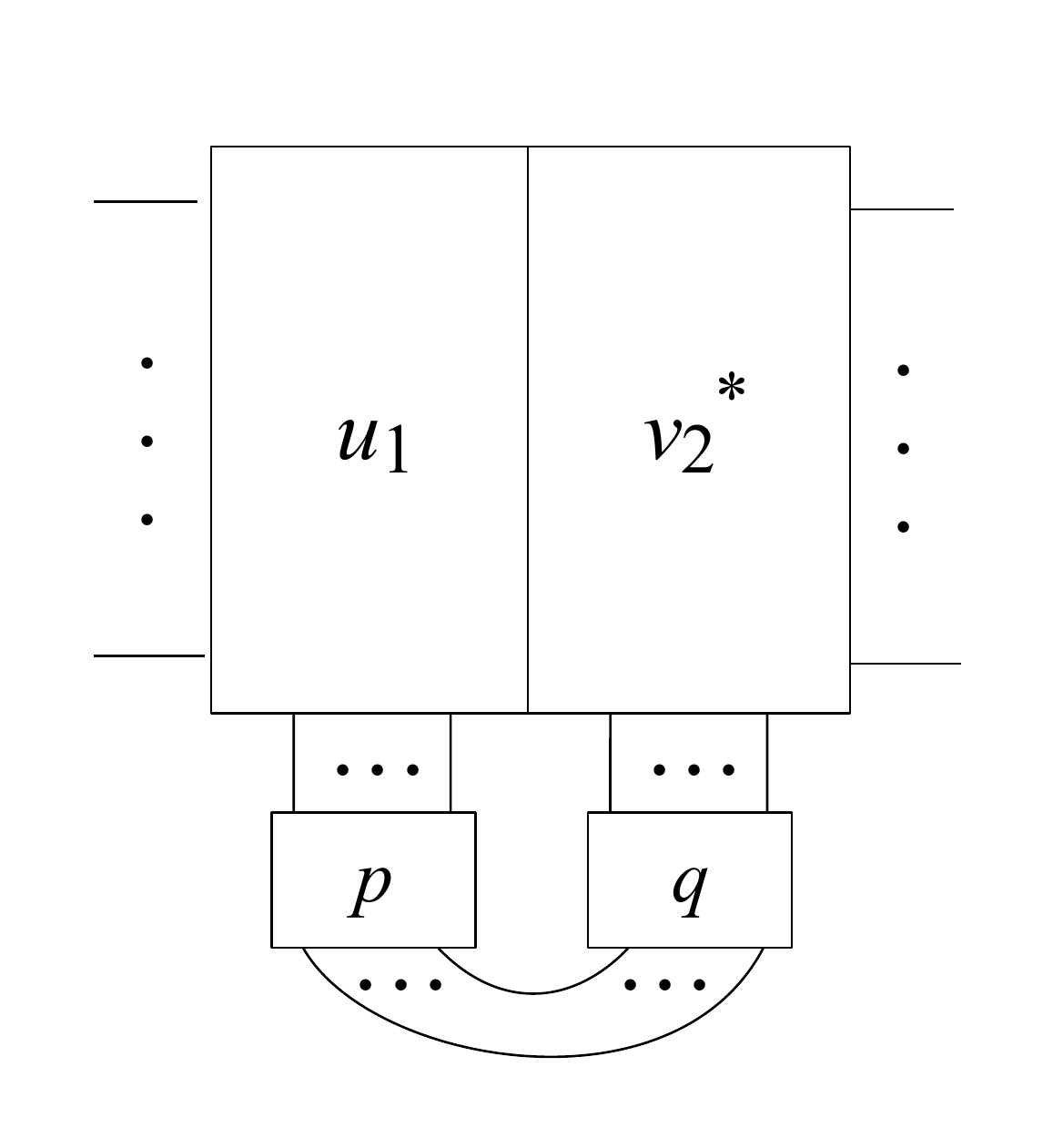}\\
  \caption{} \label{}
\end{figure}

Now, by definition, the inner product $\langle v_1\otimes u_1,v_2\otimes v_2\rangle=\langle v_1\langle u_1,u_2 \rangle_M,v_2 \rangle$ is the trace
of the following graph
\begin{figure}[H]
  \centering
  \includegraphics[width=8cm]{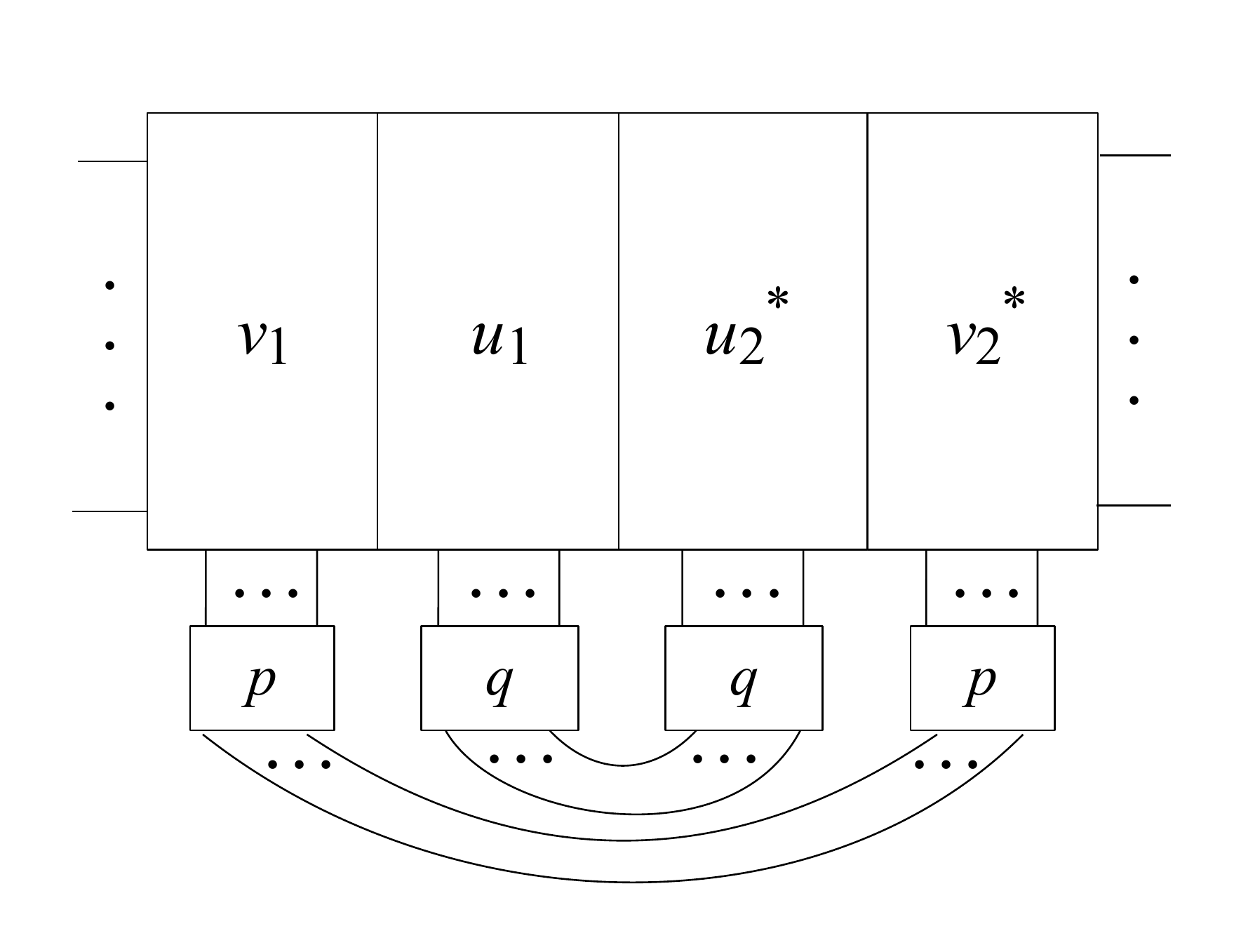}\\
  \caption{} \label{}
\end{figure}
This is just the inner product of the following two vectors in $H_m(p|q)$.
\begin{figure}[H]
  \centering
  \includegraphics[width=8cm]{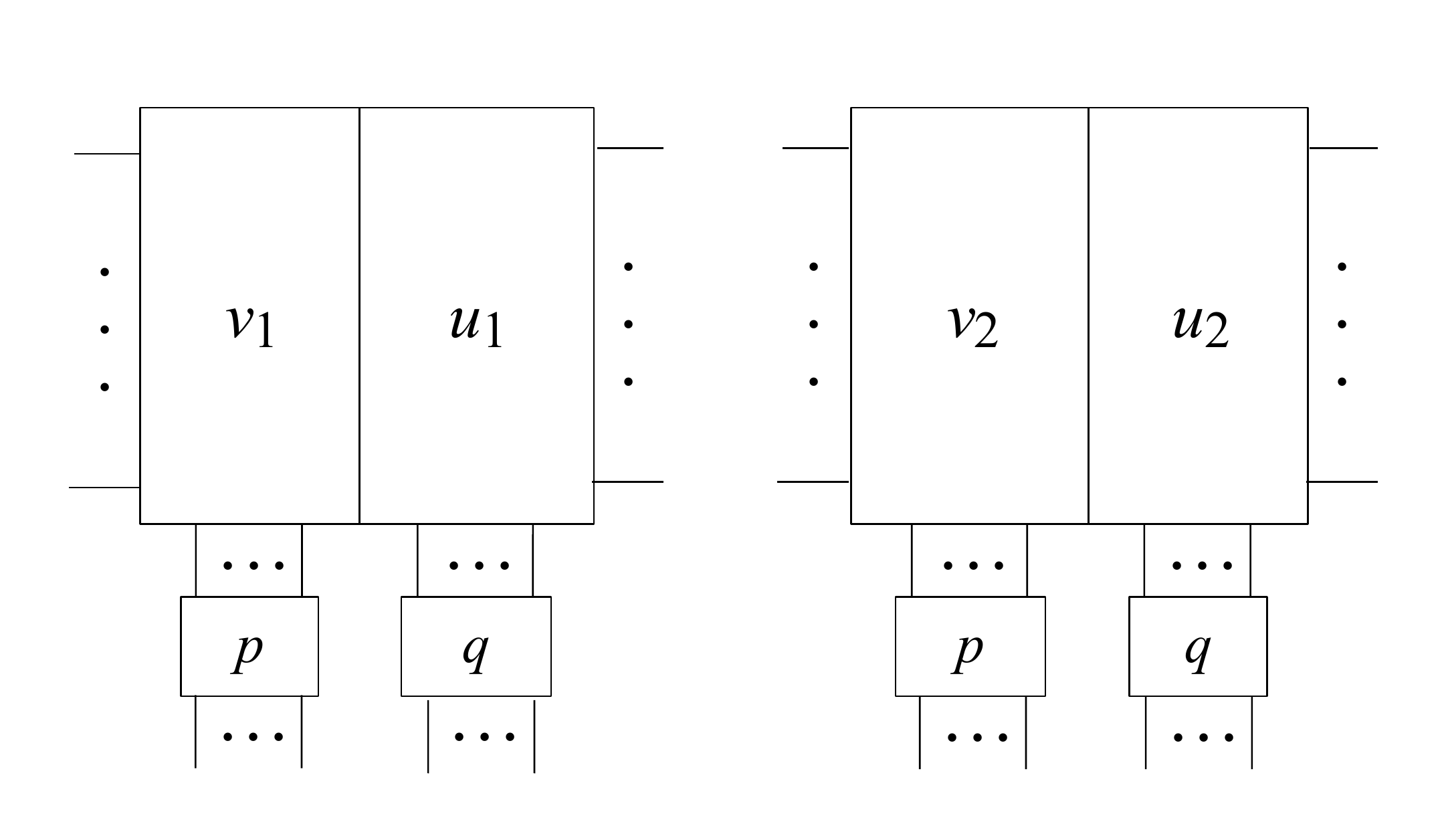}\\
  \caption{} \label{}
\end{figure}
So we have
\begin{center}
$\langle v_1\otimes u_1,v_2\otimes u_2\rangle=\langle T_0(v_1\otimes u_1),T_0(v_2\otimes u_2)\rangle$
\end{center}
which says $T_0$ is an isometry.
As $T_0$ are defined on dense subspaces, $T$ is an isometry.
\end{proof}

\begin{proposition}
The map $T:H(p)\otimes H(q)\to H(p|q)$ defined above is surjective.
\end{proposition}
\begin{proof}
We know that $\dim_M^L(H(p))=\dim_M^R(H(p))=\tr_{M'}(p)/D^k$ and $\dim_M^L(H(q))=\dim_M^R(H(q))=\tr_{M'}(q)/D^l$.

Moreover, as $p|q\in M_{k+l}(D)$ is the juxtaposition of $p$ and $q$, $\tr(p|q)=\tr(p)\tr(q)$.
Hence $\dim_M^L(H(p|q))=\dim_M^R(H(p|q))=\tr_{M'}(p|q)/D^{l+k}$.
This is the product of the two dimensions above.
Hence the image of $T$ is the whole space $H(p|q)$.
\end{proof}

\subsection{The Tensor Category and Fusion Rule}

We construct a family of $M-M$ bimodules which form a tensor category. Then we give the whole fusion algebra.
A subcategory with fusion rule $A_n$ is also obtained here.
In this section, we let the elements in $M_k(D)$ also stand for their images in $A_k$ under the representation $\pi_k$ or $\pi$ (the ones we get in Section 3 and 4 from the trace $\tr$.) 

Let $D=2\cos(\frac{\pi}{n+1})+1$. 
Hence $g_n=g_{n+1}=\dots=0$. 
For $k\geq 0$, we define $H_{k,i}=H(q_{k,i})$ with $0\leq i\leq \min\{k,n-1\}$ and $H_{0,0}=L^2(M)$.
Here $\{q_{k,i}\}$ are the minimal projections defined in section 3.3,
which are also in the commutants by section 4.2.
By Corollary 5.3, $\{H_{k,i}\}_{k\geq 0,0\leq i\leq \min\{k,n-1\}}$ is a family of irreducible $M-M$ bimodules.
\begin{proposition}
Assume $D=2\cos(\frac{\pi}{n+1})+1$, the principal graph of the inclusion $M_0\subset M_1$ is $A_n$.
\end{proposition}
\begin{proof}
It follows Proposition 4.10 and Proposition 3.13.
\end{proof}

\begin{theorem}
The fusion rule of the irreducible $M-M$ bimodules $H_{k,i}$'s is given by
\begin{center}
$H_{k,i}\otimes H_{l,j}=H_{k+l,|i-j|}\oplus H_{k+l,|i-j|+2}\oplus\dots \oplus H_{k+l,k+l-|k
+l-(i+j)|}$
\end{center}

\end{theorem}
\begin{proof}
Firstly, under the representations, we notice $g_k=0$ for $k\geq n$ since $g_n=0$ and $g_n\geq g_m$ (as projections) if $m\geq n$.

It suffices to fix $j=1$ and $1\leq i\leq \min\{n-1,k\}$.
By Proposition 5.17, $H_{k,i}\otimes H_{l,1}\cong H(q_{k,i}|q_{l,1})=H((g_ip_{i+1}\dots p_k)|(g_1p_2\cdots p_l))$ as $M-M$ bimodules.

Now let us consider the following partial isometry (and its image under $\pi_{k+l}$) in $A_{k+l}$:
\begin{center}
$w=g_i~|~(((p_1\dots p_{k-i})|g_1)r_{k-i}\dots r_1)~|~(p_1\dots p_{l-1})$,
\end{center}
which is a juxtaposition of three elements
\begin{enumerate}
\item $g_i\in M_i(D)$,
\item $((p_1\dots p_{k-i})|g_1)r_{k-i}\dots r_1 \in M_{k-i+1}(D)$,
\item $(p_1\dots p_{l-1}) \in M_{l-1}(D)$.
\end{enumerate}
One can check $ww^*=q_{k,i}|q_{l,1}$ and $w^*w=(g_i|g_1)p_{k+l-i-1}\dots p_{k+l}$, which establishes the equivalence of the following projections in $A_{k+l}$ (under the representations): $(g_ip_{i+1}\dots p_k)|(g_1p_2\cdots p_l)$ and $(g_i|g_1)p_{i+2}\dots p_{k+l}$.

Recall $g_{i+1}=g_i\cdot (1-p_{i+1})-\frac{D}{d}\frac{P_{i-1}(\tau)}{P_{i}(\tau)}g_i e_i g_i$.
Let $\mu_i=\sqrt{\frac{D}{d}\frac{P_{i-1}(\tau)}{P_{i}(\tau)}}$.
Then $g_i|g_1=g_{i+1}+(\mu_i g_i e_i)(\mu_i e_i g_i)$.
And $w_i=(\mu_i g_i e_i)(\mu_i e_i g_i)$ is a self-adjoint idempotent satisfying $w_ig_{i+1}=g_{i+1}w_{i}=0$ by Proposition 3.1.

One can show $w_i$ is equivalent to $e_ig_{i-1}=g_{i-1}e_i=g_{i-1}|e_1$:
\begin{align*}
 w_i\sim (\mu_i e_i g_i)(\mu_i g_i e_i)&=\mu_i^2 e_i \frac{d\cdot P_{i}(\tau)}{D\cdot P_{i-1}(\tau)}g_{i-1}
=\frac{D}{d}\frac{P_{i-1}(\tau)}{P_{i}(\tau)} \frac{d\cdot P_{i}(\tau)}{D\cdot P_{i-1}(\tau)} e_i g_{i-1}\\
&=e_ig_{i-1}=g_{i-1}|e_1=g_{i-1}e_i.
\end{align*}
also by Proposition 3.1.

So $g_i|g_1\cong g_{i+1}\oplus g_{i-1}|e_1=g_{i+1}\oplus g_{i-1}e_i$.
Then we have
\begin{center}
$g_ip_{i+1}\dots p_k|g_1p_2\cdots p_l=(g_{i+1}p_{i+2}\dots p_{k+l})\oplus (g_{i-1}e_ip_{i+2}\dots p_{k+l})$.
\end{center}
But $H(g_{i-1}e_ip_{i+2}\dots p_{k+l})\cong g_{i-1}p_{i}\dots p_{k+l}=H(q_{k+1,i-1})$.
Hence we have
\begin{center}
$H_{k+l,i+1}\oplus H_{k+l,i-1}=H_{k,i}\otimes H_{l,1}$,
\end{center}
which completes the proof.

\end{proof}

The dimensions of both sides can partially confirm this.
In fact, the left and right dimension of $H_{k,i}$ are the same and given by
\begin{center}
$\dim_M^L(H_{k,i})=\dim_M^R(H_{k,i}))=\frac{d^i}{D^k}P_{i}(\tau)$
\end{center}
where $0\leq i\leq \min\{k,n-1\}$.
For $H_{k+l,i+1}$ and $H_{k+l,i-1}$,
as $q_{k+l,i+1}\bot q_{k+l,i-1}$, they form a direct sum of bimodules whose left and right dimension are
\begin{center}
$\frac{d^{i+1}}{D^{k+l}}P_{i+1}(\tau)+\frac{d^{i-1}}{D^{k+l}}P_{i-1}(\tau)
=\frac{d^{i+1}}{D^{k+l}}(P_{i+1}(\tau)+d^{-2}P_{i-1}(\tau))
=\frac{d^{i+1}}{D^{k+l}}P_{i}(\tau)$.
\end{center}
Hence $H_{k+l,i+1}\oplus H_{k+l,i-1}$ and $H_{k,i}\otimes H_{l,1}$
have the same left and right dimensions.

We further define $H_k=H_{k,k}\cong H(g_k)$ for $k \geq 1$ and $H_0=H_{0,0}=L^2(M)$.
They form a subcategory with a fusion rule of type $A_n$ \cite{BK01}.
Moreover, the case $D\geq 3$ is also included in the following result.

\begin{proposition}
The fusion rule of $\{H(g_k)\}_{k\geq 0}$ depends on $D$:
\begin{enumerate}
\item For $D=2\cos(\frac{\pi}{n+1})+1$, the $M-M$ bimodules $\{H_k\}_{0\leq k\leq n-1}$ are irreducible and have a fusion rule of $A_{n}$.
\item For $D\geq 3$, the $M-M$ bimodules $\{H_k\}_{k\geq 0}$ are not irreducible and have a fusion rule of $A_{\infty}$.
\end{enumerate}
\end{proposition}

\begin{proof}
If $D=2\cos(\frac{\pi}{n})+1$, we have $g_k=0$ for $k\geq n$.
Then it follows Theorem 5.20 and $H_{k,k-2}\cong H_{k-2,k-2}$ if $k\leq n-1$.

For the case $D\geq 3$, all $g_k$ are nonzero and the relative commutants are strictly bigger than the $A_k$ by Section 4.3.
By computing the trace, $g_k$ is no longer minimal projections and $H_k$ is no longer irreducible.
And the fusion rule is $A_{\infty}$ as the proof of Theorem 5.20.

\end{proof}

\begin{remark}
The bimodules $H_k$'s are indeed $g_kL^2(M_k)g_k$ where $L^2(M_k)$ is an $M-M$ bimodule and $g_k\in M'\cap M_k$ is a minimal projection (of rank $1$).
For the case $D\geq 3$, the projections above are no longer minimal.
And the irreducible bimodules here are $L^2(M_k,\tr)p_{k_i}$, where $p_{k_i}$ are the minimal projections obtained in Section 4.3.

\end{remark}

\end{document}